\documentclass[a4paper,10pt]{amsart}
\usepackage{amssymb,amsfonts,amsmath,amsopn,amstext,amsthm,amscd,color,
graphicx,latexsym,mathrsfs,todonotes,url, verbatim,tikz-cd,enumitem}
\usepackage{hyperref}
\usepackage{geometry}
\usepackage[commandnameprefix=always]{changes}
\newtheorem{theorem}{Theorem}[section]

\theoremstyle{remark}

\newtheorem{example}[theorem]{\bf Example}

\newtheorem{remark}[theorem]{Remark}
\newtheorem{construction}[theorem]{\bf Construction}

\theoremstyle{definition}
 
\newtheorem{definition}[theorem]{Definition}
\newtheorem{hypothesis}[theorem]{Hypothesis}

\theoremstyle{plain}

\newtheorem{lemma}[theorem]{Lemma}

\newtheorem{corollary}[theorem]{Corollary}
\newtheorem{proposition}[theorem]{Proposition}
\newtheorem{question}[theorem]{Question}

\numberwithin{equation}{section}

\usepackage{hyperref}

\newcommand{\bbD}{{\mathbb D}}

\newcommand{\bbN}{{\mathbb N}}

\newcommand{\bbQ}{{\mathbb Q}}

\newcommand{\bbU}{{\mathbb U}}

\newcommand{\bbZ}{{\mathbb Z}}

\newcommand{\bG}{{\bf G}}

\newcommand{\bN}{{\bf N}}

\newcommand{\bP}{{\bf P}}

\newcommand{\bh}{{\bf h}}

\newcommand{\frb}{{\mathfrak b}}
\newcommand{\frc}{{\mathfrak c}}

\newcommand{\frg}{{\mathfrak g}}
\newcommand{\frh}{{\mathfrak h}}

\newcommand{\frm}{{\mathfrak m}}
\newcommand{\frn}{{\mathfrak n}}

\newcommand{\frp}{{\mathfrak p}}

\newcommand{\frt}{{\mathfrak t}}
\newcommand{\fru}{{\mathfrak u}}

\newcommand{\frz}{{\mathfrak z}}

\newcommand{\frA}{{\mathfrak A}}

\newcommand{\frS}{{\mathfrak S}}

\newcommand{\frX}{{\mathfrak X}}
\newcommand{\frY}{{\mathfrak Y}}

\newcommand{\cC}{{\mathcal C}}
\newcommand{\cD}{{\mathcal D}}

\newcommand{\cF}{{\mathcal F}}
\newcommand{\cG}{{\mathcal G}}
\newcommand{\cH}{{\mathcal H}}
\newcommand{\cI}{{\mathcal I}}
\newcommand{\cJ}{{\mathcal J}}

\newcommand{\cL}{{\mathcal L}}
\newcommand{\cM}{{\mathcal M}}
\newcommand{\cN}{{\mathcal N}}
\newcommand{\cO}{{\mathcal O}}

\newcommand{\cR}{{\mathcal R}}

\newcommand{\cU}{{\mathcal U}}
\newcommand{\cV}{{\mathcal V}}

\makeatletter
\newcommand*\bdot{{\mathpalette\bdot@{.9}}}
\newcommand*\bdot@[2]{\mathbin{\vcenter{\hbox{\scalebox{#2}{$\m@th#1\bullet$}}}}}
\makeatother

\newcommand{\Ad}{{\rm Ad}}

\newcommand{\alg}{{\rm alg}}

\newcommand{\diag}{{\rm diag}}

\newcommand{\dif}{{\rm dif}}
\newcommand{\dR}{{\rm dR}}
\newcommand{\End}{{\rm End}}

\newcommand{\Fil}{{\rm Fil}}

\newcommand{\Gal}{{\rm Gal}}
\newcommand{\GL}{{\rm GL}}

\newcommand{\Hom}{{\rm Hom}}

\newcommand{\Mod}{{\rm Mod}}

\newcommand{\pr}{{\rm pr}}
\newcommand{\pdR}{{\rm pdR}}

\newcommand{\Q}{{\mathbb Q}}

\newcommand{\Rep}{{\rm Rep}}

\newcommand{\rig}{{\rm rig}}
\newcommand{\Rig}{{\rm Rig}}

\newcommand{\Sen}{{\rm Sen}}

\newcommand{\Spec}{{\rm Spec}}
\newcommand{\Spf}{{\rm Spf}}

\newcommand{\Sym}{{\rm Sym}}

\newcommand{\tildefrg}{\widetilde{\frg}}
\newcommand{\tildefrY}{\widetilde{\frY}}
\newcommand{\Tor}{{\rm Tor}}

\newcommand{\wt}{\rm wt}
\newcommand{\Z}{{\mathbb Z}}

\newcommand{\mtwo}[4]{\begin{pmatrix}
	#1&#2\\#3&#4
\end{pmatrix}}

\DeclareMathOperator{\Sp}{Sp}

\begin{document}

\title{Geometric translations of $(\varphi,\Gamma)$-modules for $\GL_2(\Q_p)$}
\author{Zhixiang Wu}
\address{Universit\"at M\"unster, Fachbereich Mathematik und Informatik,
Einsteinstrasse 62,
48149 Münster, Germany}
\email{zhixiang.wu@uni-muenster.de}
\begin{abstract}
    We study ``change of weights'' maps between loci of the stack of $(\varphi,\Gamma)$-modules over the Robba ring with integral Hodge-Tate-Sen weights. We show that in the $\GL_2(\Q_p)$ case these maps can realize translations of $(\varphi,\Gamma)$-modules geometrically. The motivation is to investigate translations of locally analytic representations under the categorical $p$-adic Langlands correspondence. 
\end{abstract}
\thanks{The project was funded by the Deutsche Forschungsgemeinschaft (DFG, German Research Foundation) – Project-ID 427320536 – SFB 1442, as well as under Germany's Excellence Strategy EXC 2044 390685587, Mathematics Münster: Dynamics–Geometry–Structure.}
\maketitle
\setcounter{tocdepth}{1}
\tableofcontents

\section{Introduction}
In this introduction we consider translations of locally analytic representations of $p$-adic Lie groups from the point of view of the categorical $p$-adic Langlands program proposed by Emerton-Gee-Hellmann in \cite{emerton2023introduction}. We will give hints for the categorical story in the $\GL_2(\Q_p)$ case by realizing Ding's result in \cite{ding2023change} on translations of $(\varphi,\Gamma)$-modules in arithmetic families. 
\subsection{Translations for locally analytic representations} 
Translation is a fundamental tool to study modules over a reductive lie algebra $\mathfrak{g}$, an operation changing infinitesimal characters. Let us take $\frg=\mathfrak{gl}_n, n\geq 2$ to be the Lie algebra of $\GL_n$ over a $p$-adic coefficient field $L$ for a prime number $p$ with the Cartan subalgebra $\frt$ of the diagonal matrices. Let $U(\frg)$ be the universal enveloping algebra and $Z(\frg)$ be the center of $U(\frg)$. Via the Harish-Chandra isomorphism an (infinitesimal) character $\chi_{\lambda}:Z(\frg)\rightarrow L$ is determined by a weight $\lambda=(\lambda_1,\cdots,\lambda_n)\in \frt^*$. We can consider the category $\Mod(U(\frg))_{\chi_{\lambda}}$ of $U(\frg)$-modules which are generalized eigenspaces for the action of $Z(\frg)$ of eigenvalues given by $\chi_{\lambda}$. For another $\mu\in\frt^*$ such that $\lambda-\mu\in\Z^n$ is integral, the translation operator gives a functor
\[T_{\lambda}^{\mu}:\Mod(U(\frg))_{\chi_{\lambda}}\rightarrow \Mod(U(\frg))_{\chi_{\mu}}.\] 
If $\lambda$ and $\mu$ are both dominant integral and have the same regularity (in the sense of the stabilizers in the Weyl group for the dot action), $T_{\lambda}^{\mu}$ induces an equivalence of categories. While translations between regular and non-regular characters (into and out of the walls) are more interesting.

Locally analytic representations of a $p$-adic Lie group $G$, say $G=\GL_n(\Q_p)$, are naturally $\frg$-modules by differentiating the $G$-actions. Under $p$-adic Langlands correspondence, infinitesimal characters of locally analytic representations correspond to generalized Hodge-Tate(-Sen) weights of the associated $p$-adic Galois representations, cf. \cite{dospinescu2020infinitesimal}. Translations for locally analytic representations were studied by Jena-Lahiri-Strauch in \cite{jena2021translation}.  If a locally analytic representation $\pi $ is in $\Mod(U(\frg))_{\chi_{\lambda}}$, then its translation $T_{\lambda}^{\mu}\pi$ is still a locally analytic representation, with generalized infinitesimal character $\chi_{\mu}$. It is then extremely interesting to investigate how translations intertwine with the Langlands correspondence. The operations that change weights were already observed by Colmez in \cite{colmez2018poids}. A more systematic study was carried out by Ding in \cite{ding2023change}, based on Colmez's construction $D\mapsto D^{\natural}\boxtimes \bP^1$ of $p$-adic local Langlands \cite{colmez2010representations,colmez2016representations} from rank two $(\varphi,\Gamma)$-modules over the Robba ring to locally analytic representations of $\GL_2(\Q_p)$. Ding's idea is to translate $(\varphi,\Gamma)$-modules firstly which can be equipped with $\frg$-actions using Colmez's method (by the infinitesimal action of $G$ on $D=D\boxtimes \Z_p\subset D\boxtimes \bP^1$). Ding proposed recently in \cite{Ding2024wall} conjectures to study $p$-adic Langlands correspondences for general $\GL_n$ via translation functors.
\subsection{Categorical $p$-adic Langlands conjecture}\label{subsectionintroductioncategorical}
Let $\Rig_L$ be the category of rigid analytic spaces over $L$. Emerton-Gee-Hellmann consider in \cite{emerton2023introduction} the moduli stack $\frX_n$ (over $\Rig_L$) of $(\varphi,\Gamma)$-modules of rank $n$ over the Robba ring, which should be viewed as the $p$-adic analytic version of the stack of Langlands parameters for $\GL_n(\Q_p)$. Let $\mathrm{D}^b_{\mathrm{f.p}}(\mathrm{an. }G)$ be the derived category of locally analytic representations of $G=\GL_n(\Q_p)$ (with conjectural finiteness condition discussed in \cite[\S 6.2]{emerton2023introduction}) and let $D^b_{\mathrm{Coh}}(\frX_n)$ be the derived category of coherent sheaves on $\frX_n$. The analytic version of the categorical $p$-adic Langlands correspondence predicts the existence of a functor
\[\frA_G^{\rig}: \mathrm{D}^b_{\mathrm{f.p}}(\mathrm{an. }G)\rightarrow D^b_{\mathrm{Coh}}(\frX_n)\]
which should satisfy various properties, particularly including the compatibility between infinitesimal characters and Hodge-Tate-Sen weights. 

Let $\bh=(h_1,\cdots,h_n)\in \bbZ^n, h_1\leq\cdots\leq h_n$ be fixed integral Sen weights and $\lambda=\lambda_{\bh}:=(h_n-(n-1),\cdots,h_i-(i-1),\cdots, h_1)$ be the corresponding (automorphic)  weight of $\frt$. Let $\mathrm{D}^b_{\mathrm{f.p}}(\mathrm{an. }G)_{\chi_{\lambda}}\subset \mathrm{D}^b_{\mathrm{f.p}}(\mathrm{an. }G)$ be the full subcategory consisting of representations with generalized infinitesimal character $\chi_{\lambda}$. We consider the substack $(\frX_n)_{\bh}^{\wedge}$ (appeared in \cite[\S 5.3.23]{emerton2023introduction}), the formal completion of $\frX_n$ along the weight $\bh$ locus. For an affinoid algebra $A$, the $A$-value of $(\frX_n)_{\bh}^{\wedge}$ is the groupoid of $(\varphi,\Gamma)$-modules $D_A$ of rank $n$ over $\Sp(A)$ such that for any point $x\in\Sp(A)$, the specialization $D_A\otimes_Ak(x)$ has Sen weights $\bh$. Then $\frA_{G}^{\rig}$ should restrict to a functor:
\[\frA_{G}^{\rig}: \mathrm{D}^b_{\mathrm{f.p}}(\mathrm{an. }G)_{\chi_{\lambda}}\rightarrow D^b_{\mathrm{Coh}}((\frX_n)_{\bh}^{\wedge}).\]
For different integral weights $\lambda_{\bh},\lambda_{\bh'}$, the composite of $\frA^{\rig}_{G}$ and the translation functor $T_{\lambda_{\bh}}^{\lambda_{\bh'}}:\mathrm{D}^b_{\mathrm{f.p}}(\mathrm{an. }G)_{\chi_{\lambda_{\bh}}}\rightarrow \mathrm{D}^b_{\mathrm{f.p}}(\mathrm{an. }G)_{\chi_{\lambda_{\bh'}}}$ translates sheaves on $(\frX_n)_{\bh}^{\wedge}$ to $(\frX_n)_{\bh'}^{\wedge}$. If one believes in an ultimate equivalence of categories statement of the categorical $p$-adic Langlands correspondence as Fargues-Scholze (see \cite[Rem. 1.4.6]{emerton2023introduction}), it is then natural to ask if there exists a morphism between spaces $(\frX_n)_{\bh}^{\wedge}$ and $(\frX_n)_{\bh'}^{\wedge}$ (that induces translations of sheaves). 

\subsection{Change of weights} The functor $\frA_G^{\rig}$ is in conjectural and the geometric properties of $\frX_n$ are largely unknown. However, the answer to the question above is positive. We suppose that $\bh$ is regular (i.e. $h_1<\cdots<h_n$) for simplicity and let $0=(0,\cdots,0)$ be the zero weight. Let $B$ be the Borel subgroup of upper triangular matrices of $\GL_n$ with Lie algebra $\frb$. Consider the Grothendieck resolution 
\[f:\tildefrg=\GL_n\times^{B}\frb=\{(\nu,gB)\in\frg\times \GL_n/B\mid \Ad(g^{-1})(\nu)\in\frb\} \rightarrow \frg,(\nu,gB)\mapsto \nu\]
where $\Ad$ denotes the adjoint action.
\begin{proposition}[Proposition \ref{propositionfiberproduct}]\label{propositionintroductionfiberproduct}
    There exists a (change of weights) morphism of stacks 
    \begin{equation}\label{equationintroductionfh}
        f_{\bh}:(\frX_n)_{\bh}^{\wedge}\rightarrow (\frX_n)_{0}^{\wedge}
    \end{equation}
    such that the following commutative diagram of stacks over $\Rig_L$
    \begin{equation}\label{diagramintroductioncartesian}        
        \begin{tikzcd}
            (\frX_{n})^{\wedge}_{\bh}\arrow[r,"D_{\rm pdR}"]\arrow[d,"f_{\bh}"]
            & \tildefrg/\GL_n \arrow[d,"f"]\\ 
            (\frX_n)^{\wedge}_{0}\arrow[r,"D_{\rm pdR}"]&  \frg/\GL_n
        \end{tikzcd}
    \end{equation}
    is Cartesian.
\end{proposition}
The morphisms $D_{\pdR}$ are the local model maps defined in \textit{loc. cit}. It firstly sends a $(\varphi,\Gamma)$-module $D$ to the associated $B_{\dR}^+$-representation $W_{\dR}^+(D)$ of $\Gal(\overline{\Q}_p/\Q_p)$ which, using Fontaine's classification, gives a rank $n$ bundle $D_{\pdR}(D)$ with a (Hodge) filtration $\Fil^{\bullet}D_{\pdR}(D)$ stabilized under a nilpotent linear endomorphism $\nu$. The filtration depends on the regularity of Sen weights and is parametrized by the stack $*/P=(\GL_n/P)/\GL_n$ where $\GL_n/P$ is a flag variety and $P=B$, resp. $P=\GL_n$, in the case of weight $\bh$, resp. $0$. In the language of $B$-pairs of Berger \cite{berger2008construction}, the map $f_{\bh}$, already pointwisely described in \cite[Lem. 2.1]{Ding2024wall}, sends the $B$-pair $(W_e,W_{\dR}^+)$ attached to $D$ to $(W_e,W_{\dR,0}^+)$ where $W_{\dR,0}^+$ is the unique $\Gal(\overline{\Q}_p/\Q_p)$-invariant $B_{\dR}^+$-lattice inside $W_{\dR}(D)=W_{\dR}^+(D)[\frac{1}{t}]$ of weight $0$ associated to the trivial filtration of $D_{\pdR}(D)$. The proof of the isomorphism 
\[(\frX_{n})^{\wedge}_{\bh}\simeq (\frX_{n})^{\wedge}_{0}\times_{\frg/\GL_n}\tildefrg/\GL_n\] 
is a simple combination of the known family versions of the equivalence between $(\varphi,\Gamma)$-modules and $B$-pairs and the classification of $B_{\dR}^+$-representations with integral weights (see Appendix \ref{sectionappendixfamilies}).
\begin{remark}
    The condition for a $(\varphi,\Gamma)$-module $D$ with integral weights being de Rham is equivalent to the vanishing of the nilpotent endomorphism $\nu$ on $D_{\pdR}(D)=D_{\dR}(D)$. Let $\frX_{n}^{\rm DE}$ be the stack of rank $n$ de Rham $(\varphi,\Gamma)$-modules of weight zero (so called $p$-adic differential equations). The restriction of the diagram (\ref{diagramintroductioncartesian}) to $\nu=0$ locus is
    \begin{center}
        \begin{tikzcd}
            \mathrm{\frX}^{\rm DE}_{n}\times_{*/\GL_n}*/B\arrow[r]\arrow[d,"f_{\bh}"]
            & */B=(\GL_n/B)/\GL_n \arrow[d,"f"]\\ 
            \mathrm{\frX}^{\rm DE}_{n}\arrow[r]&  */\GL_n.
        \end{tikzcd}
    \end{center}
    Thus $\mathrm{\frX}^{\rm DE}_{n}\times_{*/\GL_n}*/B$ is isomorphic to the stack of de Rham $(\varphi,\Gamma)$-modules of weight $\bh$. On the other hand, using Berger's equivalence \cite{berger2008equations}, this stack is locally isomorphic to $\mathrm{WD}_{n}\times_{*/\GL_n}*/B$ (\cite[Thm. 5.2.4]{emerton2023introduction}) where $\mathrm{WD}_n$ is the analytification of the stack of Weil-Deligne representations of rank $n$. 
\end{remark}

Proposition \ref{propositionintroductionfiberproduct} works for general $\GL_n(K)$ and non-regular $\bh$ with suitable modifications. For example, suppose that $K=\bbQ$ and $\bh=(h_1,\cdots,h_n)$ where $h_1\leq \cdots \leq h_n$. Let $P_{\bh}\supset B$ be the parabolic subgroup of $\GL_n$ with a Levi subgroup the centralizer of $\diag(h_1,\cdots,h_n)\in \GL_n$. There is the partial Grothendieck resolution $\tildefrg_{P_{\bh}}\rightarrow \frg$ (see \S\ref{sectinGrothendieckresolution}) and we have 
\[(\frX_{n})^{\wedge}_{\bh}\simeq (\frX_{n})^{\wedge}_{0}\times_{\frg/\GL_n}\tildefrg_{P_{\bh}}/\GL_n\] 
similarly as in the regular case when $P_{\bh}=B$. If $\bh'$ is a weight which is less regular than $\bh$ in the sense that $P_{\bh'}\supset P_{\bh}$. Then the above description of $(\frX_{n})^{\wedge}_{\bh},(\frX_{n})^{\wedge}_{\bh'}$ together with the natural map $\tildefrg_{P_{\bh}}\rightarrow \tildefrg_{P_{\bh'}}$ induce a change of weigths morphism $f_{\bh,\bh'}:(\frX_{n})^{\wedge}_{\bh}\rightarrow (\frX_{n})^{\wedge}_{\bh'}$. Moreover, the change of weights maps exist for possibly non-integral weights whenever the change doesn't increase the regularity (not ``go out of the wall'' of the Weyl chamber). See \S\ref{subsectiongeneralchange} for a general construction of ``change of weights'' of $(\varphi,\Gamma)$-modules and see \cite[Ch. 5]{wu:tel-04116114} for a discussion of local models in the non-integral weights cases.  

An immediate consequence of Proposition \ref{propositionintroductionfiberproduct} is the existence of isomorphisms between loci of $\frX_n$ with different regular Hodge-Tate-Sen weights. We will use $f_{\bh}$ to realize translations of $(\varphi,\Gamma)$-modules in $\GL_2(\Q_p)$-case and then discuss a general speculation. 
\subsection{Geometric translations of $(\varphi,\Gamma)$-modules}
Now we focus on the case $G=\GL_2(\Q_p)$ where we have Colmez's construction. The main result of this paper can only be stated and proved in this case. Following Colmez, there is a unique way to make a $(\varphi,\Gamma)$-module $D_A$ over an affinoid $\Sp(A)$ a $\frg$-module so that $Z(\frg)$ acts via a character determined by the Sen weights of $D_A$ (with explicit formula of the $\frg$-action on $D_A$ given by Dospinescu \cite{dospinescu2012actions}). If $D_A\in (\frX_{2})^{\wedge}_{\bh}(A)$ for some fixed weight $\bh$ with associate $\lambda=\lambda_{\bh}\in\frt^*$, then one can talk about the translation $T_{\lambda}^{\mu}D_A$ to another integral weight as $\frg$-modules. Ding's method shows that $T_{\lambda}^{\mu}D_A$ is still a $(\varphi,\Gamma)$-module. The following is our main theorem.
\begin{theorem}[Theorem \ref{theorem}]\label{theoremintroduction}
    Suppose $\bh=(h_1,h_2)\in\Z^2,h_1<h_2$, $\lambda=\lambda_{\bh}$, $\mu=\lambda_0=(-1,0)$. Let $D_{(\frX_{2})^{\wedge}_{\bh}}$ (resp. $D_{(\frX_{2})^{\wedge}_{0}}$) be the restriction of the universal $(\varphi,\Gamma)$-module on $\frX_2$ to $(\frX_{2})^{\wedge}_{\bh}$ (resp. to $(\frX_{2})^{\wedge}_{0}$). Then the following statements are true locally on affinoid charts of $\frX_2$ (namely on any $\Sp(A)$ with a formally smooth map $\Sp(A)\rightarrow \frX_2$). 
    \begin{enumerate}
        \item There exists an isomorphism
        \[T_{\lambda}^{\mu}D_{(\frX_{2})^{\wedge}_{\bh}}\simeq f_{\bh}^*D_{(\frX_{2})^{\wedge}_{0}}\]
        of $(\varphi,\Gamma)$-modules of rank two which induces the map $f_{\bh}:(\frX_{2})^{\wedge}_{\bh}\rightarrow (\frX_{2})^{\wedge}_{0}$.
        \item There exists an isomorphism
        \[T_{\mu}^{\lambda}D_{(\frX_{2})^{\wedge}_{0}}\simeq Rf_{\bh,*}D_{(\frX_{2})^{\wedge}_{\bh}}\]
        of $(\varphi,\Gamma)$-modules of rank four and in degree $0$ on $(\frX_{2})^{\wedge}_{0}$. 
    \end{enumerate}
\end{theorem}
Certainly, all objects in the above theorem need proper definitions. The theorem will be stated and proved without the language of stacks. For a chart $\Sp(A)\rightarrow \frX_2$ with $D_A$ the pullback of the universal $(\varphi,\Gamma)$-module, we construct the space $\Sp(A)^{\wedge}=\Sp(A)\times_{\frX_2}(\frX_2)^{\wedge}_0$ as what should be called an affinoid formal rigid space. Note that it is still a conjecture \cite[Conj. 5.1.19]{emerton2023introduction} that $\frX_2$ is a rigid analytic Artin stack \cite[Def. 5.1.10]{emerton2023introduction}. But it is konwn that such charts, namely smooth maps $\Sp(A)\rightarrow \frX_2$ which will guarantee our Hypothesis \ref{hypothesisflat}, exist locally around $L'$-points of $\frX_2$ with $L'$ a field finite over $L$ \cite[Thm. 5.1.20]{emerton2023introduction}. Using Proposition \ref{propositionintroductionfiberproduct}, we only need to prove the results for the map $f_{\bh}^{-1}(\Sp(A)^{\wedge})=\Sp(A)^{\wedge}\times_{\frg/\GL_2}\tildefrg/\GL_2\rightarrow \Sp(A)^{\wedge}$ between formal rigid spaces and $(\varphi,\Gamma)$-modules over these spaces. Fortunately, the map is proper and the cohomologies can be studied via GAGA theorems (see Appendix \ref{sectionappendixformalfunction}). 

The key of the proof is the geometric properties of the Grothendieck resolution $f:\tildefrg\rightarrow \frg$. For example, $Rf_*\cO_{\tildefrg}$ concentrates in degree $0$ and is locally free of rank two over $\cO_{\frg}$, which is basically the reason that $Rf_{\bh,*}D_{(\frX_{2})^{\wedge}_{\bh}}$ has rank four and concentrates in degree $0$. Another vital input is the flatness of the local model map $D_{\pdR}$ (to use flat base change). We can prove the flatness of $D_{\pdR}$ in $\GL_2(\Q_p)$ case (\S\ref{sectionflatlocalmodel}) and we expect it is always flat for other $G$.
\subsection{Speculation}
Finally, we explain the motivation of Theorem \ref{theoremintroduction}. For $U(\frg)$-modules with a generalized infinitesimal character, translation functors can be realized geometrically using Beilinson-Bernstein localization, cf. \cite{beilinson1999wall}. A better approach for translations (into and out of the wall) is to consider singular localizations which send $\mathrm{Mod}(U(\frg))_{\chi_{\mu}}$ for non-regular $\chi_{\mu}$ to some $D$-modules on corresponding partial flag varieties $\GL_n/P$ \cite{backelin2015singular,bezrukavnikov2006singular}. Then translation functors after localization can be realized using pushforward and pullback along the maps between flag varieties like $\GL_n/B\rightarrow \GL_n/P$, see \cite[\S 6]{backelin2015singular} and more similarly \cite[Lem. 2.2.5]{bezrukavnikov2006singular}.

The functor $\frA_G^{\rm rig}$ in \S\ref{subsectionintroductioncategorical} is expected to be certain localization, of the form (\cite[Rem. 6.2.10]{emerton2023introduction})
\[\pi\mapsto \cL_{\infty}\widehat{\otimes}_{\cD(G)}^L\pi\]  where $\cD(G)$ is the distribution algebra of $G$ and $\cL_{\infty}$ plays the role of the sheaf of differential operators on $\frX_n$. As in Proposition \ref{propositionintroductionfiberproduct}, we fix a regular Hodge-Tate weight $\bh$ (resp. non-regular weight $0$) and let $\chi_{\lambda}$ (resp. $\chi_{\mu}$) be the associated infinitesimal character. With the map (\ref{equationintroductionfh}), we get a diagram of functors
\begin{center}
    \begin{tikzcd}
        \mathrm{D}^b_{\mathrm{f.p}}(\mathrm{an. }G)_{\chi_{\lambda}}\arrow[r,"\frA_G^{\rig}"]\arrow[d,shift left, "T_{\lambda}^{\mu}"]& D^b_{\mathrm{Coh}}((\frX_n)_{\bh}^{\wedge})\arrow[d,shift left, "Rf_{\bh,*}"]\\ 
        \mathrm{D}^b_{\mathrm{f.p}}(\mathrm{an. }G)_{\chi_{\mu}}\arrow[r,"\frA_G^{\rig}"]\arrow[u,shift left, "T_{\mu}^{\lambda}"]& D^b_{\mathrm{Coh}}((\frX_n)_{0}^{\wedge}).\arrow[u,shift left,"Lf_{\bh}^*"]
    \end{tikzcd}
\end{center}
\begin{question}\label{equationconjecture}
    In the above diagram, do we have $\frA_G^{\rig}\circ T_{\lambda}^{\mu}=Rf_{\bh,*}\circ\frA_G^{\rig}$ and $\frA_G^{\rig}\circ T_{\mu}^{\lambda}=Lf_{\bh}^*\circ\frA_G^{\rig}$?
\end{question}	
Taking account of the adjunction for translation functors, this suggests to ask whether we have isomorphisms (where $w_0$ denotes the longest element in the Weyl group)
\begin{align}\label{equationconjectureLinfty1}
     T_{-w_0\lambda}^{-w_0\mu}\cL_{\infty}|_{(\frX_{n})^{\wedge}_{\bh}}&\simeq Lf_{\bh}^*\cL_{\infty}|_{(\frX_{n})^{\wedge}_{0}},\\ \label{equationconjectureLinfty2}
     T_{-w_0 \mu}^{-w_0\lambda}\cL_{\infty}|_{(\frX_{n})^{\wedge}_{0}}&\simeq Rf_{\bh,*}\cL_{\infty}|_{(\frX_{n})^{\wedge}_{\bh}}?
\end{align}
\begin{remark}
    The sheaf $\cL_{\infty}$ should be a family version of the dual of $\Pi(D)$ where for a $(\varphi,\Gamma)$-module $D$ of rank $n$, we write $\Pi(D)$ for the conjectural locally analytic representation of $\GL_n(\Q_p)$ attached to $D$ via $p$-adic local Langlands correspondence. The expected isomorphism $T_{-w_0\lambda}^{-w_0\mu}\cL_{\infty}|_{(\frX_{n})^{\wedge}_{\bh}}\simeq Lf_{\bh}^*\cL_{\infty}|_{(\frX_{n})^{\wedge}_{0}}$ is just a family (dual) version of a conjecture of Ding \cite[Conj. 1.1, (1)]{Ding2024wall}: if $D$ has Hodge-Tate-Sen weights $\bh$, then $T_{\lambda}^{\mu}\Pi(D)=\Pi(f_{\bh}(D))$. 
\end{remark}
At present, there is no construction of the sheaf $\cL_{\infty}$ for $n\geq 2$. In the case of $\GL_2(\Q_p)$, as in the Banach case \cite[\S 7.3]{emerton2023introduction}, $\cL_{\infty}$ should be the family version of Colmez's construction $D\mapsto D^{\natural}\boxtimes \bP^1$ from $(\varphi,\Gamma)$-modules to $\cD(G)$-modules (up to a twist). An easier object to construct is $D\boxtimes \bP^1$ as only $U(\frg)$-modules, which equals to copies of $(\varphi,\Gamma)$-modules. The main theorem immediately implies the following (compared with (\ref{equationconjectureLinfty1}) and (\ref{equationconjectureLinfty2}) up to a twist).
\begin{corollary}[Corollary \ref{corollaryDboxtimesP1}]
    In the notation of Theorem \ref{theoremintroduction}, the following isomorphisms of sheaves of $U(\frg)$-modules hold locally on affinoid charts of $\frX_2$ with suitable definitions of the objects:
    \begin{align*}
        T_{\lambda}^{\mu}D_{(\frX_{2})^{\wedge}_{\bh}}\boxtimes \bP^1&\simeq Lf_{\bh}^*D_{(\frX_{2})^{\wedge}_{0}}\boxtimes \bP^1,\\
        T_{\mu}^{\lambda}D_{(\frX_{2})^{\wedge}_{0}}\boxtimes \bP^1&\simeq Rf_{\bh,*}D_{(\frX_{2})^{\wedge}_{\bh}}\boxtimes \bP^1.
    \end{align*}
\end{corollary}
\begin{remark}
    In the trianguline cases, we expect that the identification of functors in Question \ref{equationconjecture} applying for finite slope Orlik-Strauch representations is compatible with the conjectural description of $\frA_G^{\rm rig}(\pi)$ using local models and Bezrukavnikov's functor in \cite[\S 6.2.27]{emerton2023introduction}, provided the version of \textit{loc. cit.} for non-regular weights (see \cite{wu2021local} for a discussion on cycles). 
\end{remark}
\begin{remark}
    Geometric translations for real local Langlands correspondence were already discussed to some extent, cf. \cite[\S 16]{adams1992langlands}, \cite{straser2014geometric}, etc.
\end{remark}
\subsection{Outline} We review basics on Grothendieck-Springer resolutions in \S\ref{sectinGrothendieckresolution}. In \S\ref{sectionchangeofweight}, we study the change of weights maps and prove Proposition \ref{propositionintroductionfiberproduct}, and show that the local model map is flat in the $\GL_2(\Q_p)$-case. In \S\ref{sectiontranslationinfamily}, we compute translations of $(\varphi,\Gamma)$-modules in families. We prove the main theorem on geometric translations in \S\ref{sectiongeometrictranslation}. We also show in \S\ref{subsectionspecialization} how to recover some of Ding's pointwise calculation of translations from our main Theorem \ref{theoremintroduction}. In two appendices \ref{sectionappendixfamilies} and \ref{sectionappendixformalfunction}, we collect facts on families of $(\varphi,\Gamma)$-modules over the Robba rings and formal rigid geometry.

\subsection{Acknowledgments}
The idea of ``change of weights'' is due to Yiwen Ding. I would like to thank him for many inspiring conversations and questions, and for sharing with me a preliminary version of \cite{Ding2024wall}. I would also like to thank Eugen Hellmann, Claudius Heyer and Damien Junger for helpful discussions during the preparation of this work. I thank the anonymous referee for helpful comments and suggestions.
\subsection{Notation}\label{subsectionnotation}
We fix a prime number $p$. We use $L$ a finite extension of $\Q_p$ as the coefficient field.  

Suppose that $K$ is a $p$-adic local field. Let $K_{\infty}=K(\mu_{\infty})=\cup_{m}K(\mu_{p^m})$ be the extension of $K$ by adding all $p$-th power roots of unity, and $\Gamma_K:=\Gal(K_{\infty}/K)$. Also let $K_m=K(\mu_{p^m})$ for $m\geq 1$. Let $\cG_K=\Gal(\overline{\Q}_p/K)$ be the absolute Galois group of $K$. We write $\Gamma=\Gamma_K$ if $K=\Q_p$.

Let $\epsilon$ be the cyclotomic character of $\cG_K$ viewed also as the character $\mathrm{Norm}_{K/\Q_p}|\mathrm{Norm}_{K/\Q_p}|_{p}$ of $K^{\times}$  where $|p|_{p}=p^{-1}$. In our convention, the cyclotomic character has Hodge-Tate weights one.

We follow the notation of $(\varphi,\Gamma)$-modules in \cite{kedlaya2014cohomology}. If $X$ is a rigid space, we write $\cR_{X,K}$ for the Robba ring of $K$ over $X$ \cite[Def. 6.2.1]{kedlaya2014cohomology}. If $A$ is an affinoid algebra, write $\cR_{A,K}:=\cR_{\mathrm{Sp}(A),K}$. We also need the Robba rings $\cR_{A,K}^r$ and $\cR_{A,K}^{[s,r]}$ for $0<s<r$ small enough in the sense of \cite[Def. 2.2.2]{kedlaya2014cohomology} and $\cR_{A,K}=\bigcup_{0<r}\cR_{A,K}^r$. The ring $\cR_{\Q_p,K}^{[s, r]}$ is the ring of single variable rigid analytic functions converging over $\bbU_{K_0'}^{[s,r]}=\{p^{-\frac{r}{p-1}}\leq |X|\leq p^{-\frac{s}{p-1}}\}$ with coefficients in $K_0'$ the maximal unramified subfield of $K_{\infty}$ and $\cR_{\Q_p,K}^{r}=\varprojlim_{s\leq r}\cR_{\Q_p,K}^{[s,r]}$ is the functions on $\bbU^r_{K_0'}=\cup_{0<s<r}\bbU_{K_0'}^{[s,r]}$. Then $\cR_{A,K}^{[s,r]}=\cR_{\Q_p,K}^{[s,r]}\widehat{\otimes}_{\Q_p}A$ and $\cR_{A,K}^{r}=\varprojlim_{s\leq r}\cR_{A,K}^{[s,r]}$. The group $\Gamma_K$ acts on $\cR_{A,K}^r$ and $\varphi:\cR_{A,K}^r\rightarrow \cR_{A,K}^{r/p}$. 

A $(\varphi,\Gamma_K)$-module $D_A$ over $\cR_{A,K}$ is always the base change of a $(\varphi,\Gamma_K)$-module $D_A^r$ over $\cR_{A,K}^r$ for some $r$ small enough: a finite projective $\cR_{A,K}^r$-module $D_{A}^r$ equipped with an isomorphism $\varphi^*D_A^r=\cR_{A,K}^{r/p}\otimes_{\varphi,\cR_{A,K}^r}D_A^r\simeq \cR_{A,K}^{r/p}\otimes_{\cR_{A,K}^r}D_A^r$ and a commuting continuous semilinear action of $\Gamma_K$. Given a $(\varphi,\Gamma_K)$-module $D_A^r$ over $\cR_{A,K}^r$, we write $D_A^{s}:=\cR_{A,K}^{s}\otimes_{\cR_{A,K}^r}D_A^r,D_A^{[s,r]}:=\cR_{A,K}^{[s,r]}\otimes_{\cR_{A,K}^r}D_A^r,D_A:=\cR_{A,K}\otimes_{\cR_{A,K}^r}D_A^r$ for $s\leq r$. The $(\varphi,\Gamma_K)$-cohomologies $H^i_{\varphi,\gamma_K}(D_A),i=0,1,2$ for a $(\varphi,\Gamma_K)$-module is defined using Herr complex as in \cite[Def. 2.3.3]{kedlaya2014cohomology} where $\gamma_K$ is a fixed topological generator of $\Gamma_K$ modulo the torsion subgroup. And we write $\Hom_{\varphi,\gamma_K}(-,-)$ for the Hom space of $(\varphi,\Gamma_K)$-modules. We write $\cR_{A}=\cR_{A,K}$, $\gamma=\gamma_K$, etc., if $K=\Q_p$.

Let $B_{\dR}^{+},B_{\dR}=B_{\dR}^{+}[\frac{1}{t}]$ be Fontaine's de Rham period rings, where $t$ is Fontaine's $2\pi i$.

For any $r>0$, let $m(r)$ be the minimal integer such that $p^{m(r)-1}[K(\mu_{p^{\infty}}):K_0(\mu_{p^{\infty}})]r\geq 1$ where $K_0$ is the maximal unramified subfield of $K$. The integer is taken so that there are injections $\iota_m:\cR_{L,K}^r\hookrightarrow (L\otimes_{\Q_p}K_m)[[t]]$ for $m\geq m(r)$, see Appendix \ref{sectionappendixfamilies}.

For a continuous character $\delta:K^{\times}\rightarrow\Gamma(X, \cO_X)^{\times}$ over a rigid space $X$, write $\cR_{X,K}(\delta)$ for the corresponding rank one $(\varphi,\Gamma_K)$-module over $\cR_{X,K}$ in \cite[Cons. 6.2.4]{kedlaya2014cohomology}. If $K=\Q_p$, write $z:\Q_p^{\times}\hookrightarrow L^{\times}$ for the algebraic character of weight one. Recall that $\cR_L(z)=t\cR_L$.

If $G=\GL_n$ over $L$ for some $n$, we always take the Borel $B$ the subgroup of upper-triangular matrices and $T$ the diagonal torus. We use the fraktur letter $\frg$ (resp. $\frb$, resp. $\frt$) for the Lie algebra of the group $G$ (resp. $B$, resp. $T$), also viewed as an affine scheme (or its analytification) over $L$. Denote by $\Ad:G\rightarrow \mathrm{End}(\frg)$ the adjoint representation. For a Lie algebra $\frg$, denote by $U(\frg)$ the universal enveloping algebra and $Z(\frg)$ the center of $U(\frg)$. We write $\cN\subset \frg$ for the nilpotent cone. 

If $X$ is a rigid space with $x\in X$, we write $k(x)$ for the residue field at $x$. For a $(\varphi,\Gamma_K)$-module $D_X$ over $X$, we write $D_x$ or $D_{k(x)}$ for the base change to $x$. More generally for an affinoid algebra $A$ and an $A$-point $\Sp(A)\rightarrow X$, write $D_A=D_X\otimes_{\cR_{X,K}}\cR_{A,K}$, and similarly $D_A^r=D_X^r\otimes_{\cR_{X,K}^r}\cR_{A,K}^r$, etc.

Let $\cC_L$ denote the category of commutative local Artinian $L$-algebras with residue field $L$. If $A\in\cC_L$, let $\frm_A$ be its maximal ideal. 

If $Z$ is a commutative ring, $I$ is a finitely generated ideal of $Z$ and $M$ is a $Z$-module, then write $M[I]=\{m\in M\mid z.m=0,\forall z\in I\}$ and $M[I^{\infty}]:=\cup_{i=1}^{\infty}M[I^i]$. If $\chi:Z\rightarrow A$ is a surjction of rings with kernel $I$ and $M$ is an $A$-module, write $M[Z=\chi]=M[I]$ and $M\{Z=\chi\}=M[I^{\infty}]$. 
\section{The Grothendieck-Springer resolution}\label{sectinGrothendieckresolution}
In this section, we recall the basics of the Grothendieck-Springer resolution. We only consider schemes over $L$ for the moment. 

Let $G$ be a split reductive group over $L$ with a Borel subgroup $B$ and a maximal torus $T$. We will write $\frh=\frt$ for the Cartan subalgebra. We use $P$ to denote standard parabolic subgroups containing $B$ with the Lie algebra $\frp$. Write $W$ (resp. $W_P$) for the Weyl group of $G$ (resp. the Levi of $P$).
\subsection{Recollection on Grothendieck-Springer resolution}
For a parabolic subgroup $P\supset B$ with the Levi subgroup $M$ containing $T$, consider the scheme
\[\tildefrg_P:=G\times^{P}\frp\simeq\{(\nu,gP)\in\frg\times G/P\mid \Ad(g^{-1})\nu\in\frp\}\]
and the partial Grothendieck resolution
\[f_P:\tildefrg_P\rightarrow \frg, (\nu,gP)\mapsto \nu.\]
We will omit the subscript $P$ when $P=B$, namely write $f:\tildefrg=\tildefrg_B\rightarrow \frg$. There is a natural map 
\[\tildefrg_P\rightarrow \frm/\!/M\simeq \frh/W_P\]
sending $(\nu,gP)$ to the projection of $\Ad(g^{-1})\nu\in\frp$ to $\frm$. The composite $\tildefrg_P\stackrel{f_P}{\rightarrow}\frg\rightarrow \frg/\!/G\simeq \frh/W$ coincides with $\tildefrg_P\rightarrow \frh/W_P\rightarrow \frh/W$, hence the two maps induce a map 
\[g_P: \tildefrg_P \rightarrow \frg\times_{\frh/W}\frh/W_P.\]
\begin{lemma}\label{lemmaglobalsectiontildegP} Let $f_P,g_P$ be as above.
    \begin{enumerate}
        \item The morphism $f_P$ is proper and surjective, finite over $\frg^{\rm reg}$ and is finite \'etale of degree $|W/W_P|$ over $\frg^{\rm reg-ss}$. Here $\frg^{\rm reg-ss}\subset \frg^{\rm reg} \subset \frg$ denote open subschemes of regular semisimple elements and regular elements.
        \item The natural map $\cO_{\frg\times_{\frh/W}\frh/W_P}\rightarrow Rg_{P,*}\cO_{\tildefrg_P}$ is an isomorphism. Moreover, the map $\frg\times_{\frh/W}\frh/W_P\rightarrow \frg$ is finite flat of rank $|W/W_P|$. 
    \end{enumerate}
\end{lemma}
\begin{proof}
    (1) See for example \cite[Lem. 2.3]{wu2021local}.
    
    (2) The isomorphism is \cite[Lem. 3.2]{backelin2015singular}. (In \textit{loc. cit.}, the map $g_P$ is defined in terms of the linear dual $\frg^*$ of $\frg$. While there are isomorphisms $\frg^*\simeq \frg, \frh^*\simeq \frh$ and $(\frg/\frn_P)^*\simeq \frp$ where $\frn_P$ is the nilradical of $\frp$ if we choose an invariant nondegenerate symmetric bilinear form on $\frg$, for example the Killing form when $\frg$ is semisimple.) The map $\frh/W_P\rightarrow \frh/W$ is flat by miracle flatness. 
\end{proof}

We don't really need the following lemma, but it might be helpful to keep it in mind.
\begin{lemma}\label{lemmadualizingsheaf}
    The dualizing sheaf of $\tildefrg_P$ is trivial. And there is a canonical isomorphism $f^!_P\cF=Lf^{*}_P\cF$ for any coherent $\cO_{\frg}$-module $\cF$. 
\end{lemma}
\begin{proof}
    We follow the proof of \cite[Lemma. 5.1.1]{brion2007frobenius}. Let $\pi_P:\tildefrg_P\rightarrow G/P$ be the projection. The canonical bundle $\omega_{\tildefrg_P}$ of $\tildefrg_P$ is isomorphic to $\pi_P^{*}\omega_{G/P}\otimes_{\cO_{\tildefrg_P}} \omega_{\pi_P}$ where $\omega_{\pi_P}$ denotes the relative dualizing sheaf. We know $\omega_{G/P}=G\times^P\delta_P$ where $\delta_P$ is the sum of all roots in the unipotent radical $\fru$ of $\frp$. Moreover, since $\tildefrg_P=G\times^P\frp$, $\omega_{\pi_P}=G\times^P\omega_{\frp}$. Consider the fiberation $0\rightarrow\fru\rightarrow\frp\rightarrow\frm\rightarrow 0$. Choosing coordinates $m,u$ of $\frm$ and $\fru$, then $\omega_P$ is generated by $dm\wedge du$. Since $\omega_{\frm}$ is trivial as a $M$-representation, $\omega_{\frp}=-\delta_P$ as a $P$-representation. Hence $\omega_{\tildefrg_P}\simeq \cO_{\tildefrg_P}$. 

    We get $f^!_P\cO_{\frg}=f^!_P\omega_{\frg}=\cO_{\tildefrg_P}$. The map $f_P:\tildefrg_P\rightarrow \frg$ is perfect, hence $f^!_P\simeq Lf_P^*$ by \cite[\href{https://stacks.math.columbia.edu/tag/0B6U}{Tag 0B6U},\href{https://stacks.math.columbia.edu/tag/068D}{Tag 068D}]{stacks-project}.
\end{proof}
\subsection{Direct images of some line bundles} We need compute direct images of some line bundles on $\tildefrg$ in the $G=\GL_2$ case (Proposition \ref{propositiondirectimagelinebundle}). The computation will be the key for our main result on direct images of $(\varphi,\Gamma)$-modules (Proposition \ref{propositiondirectimagephigammamodule}). 

Take $\frg=\mathfrak{gl}_2=\Spec (L[a,b,c,z])$ where $a,b,c,z$ are coordinates for entries of matrices 
\[\mtwo{a+z}{b}{c}{-a+z}\in\frg.\] 
Take $\frh=\Spec(L[h,z])$ for $(z+h,z-h)\in \frh$. Recall the map $g:\widetilde{\frg}\rightarrow\frg\times_{\frh/W}\frh$ where $\frh/W=\Spec(L[z,h^2])$ and $\cO(\frh/W)\rightarrow\cO(\frg):h^2\mapsto a^2+bc$ and that $\cN\subset \mathfrak{sl}_2\subset \frg$ denotes the nilpotent cone. 
\begin{lemma}\label{lemmapropertygrothendieckresolutiongl2}
    The map $f:\tildefrg\rightarrow\frg$ is a finite \'etale rank two cover over $\frg^{\rm reg-ss}=\frg\setminus\{a^2+bc=0\}$ and is finite flat of degree two over $\frg^{\rm reg}=\frg\setminus\{a=b=c=0\}$. The fiber of $f$ over $0\in \frg$ is identified with $G/B$ and the fiber over a closed point $x\in\cN\setminus\{0\}$ is ramified of degree $2$ over the residue field $k(x)$. 
\end{lemma}
\begin{proof}
    By \cite[Prop. 2.1.1]{breuil2019local} or Lemma \ref{lemmaglobalsectiontildegP}, we know the corresponding preimage $\tildefrg^{\rm reg-ss}$ (resp. $\tildefrg^{\rm reg}$) is finite \'etale (resp. finite) over $\frg^{\rm reg-ss}$ (resp. $\frg^{\rm reg}$). Since $ f_*\cO_{\tildefrg^{\rm reg}}=\cO_{\frg\times_{\frh/W}\frh\setminus\{a=b=c=0\}}$, we see $\tildefrg^{\rm reg}=(\frg\times_{\frh/W}\frh)\setminus\{a=b=c=0\}$. The fiber over $x\in\cN\setminus\{ 0\}\subset \frg^{\rm reg}$, which has image $0$ in $\frh/W$, is the fiber of $ \frh\rightarrow\frh/W$ over $0$ whose coordinate ring is $k(x)[h]/h^2$.
\end{proof}
Let $\pi:\tildefrg\rightarrow G/B$. For $k\in \Z$, we write $\cO_{\tildefrg}(k)$ for the line bundle $\pi^*\cO_{G/B}(k)$ where our convention is the standard one that $ \cO_{G/B}(k)$ is ample for $k\geq 1$. Let $\cV=\cO_{\frg}e_1\oplus \cO_{\frg}e_2$ be the universal trivialized rank two bundle on $\frg$. The operator $\nu$ acts universally on $\cV$
\[\nu: x_1e_1+x_2e_2\mapsto ((a+z)x_1+bx_2)e_1+(cx_1+(z-a)x_2)e_2\] 
for $x_1,x_2\in \cO_{\frg}$, hence also on $f^*\cV$. By definition, $\nu$ stabilizes the following short exact sequence which gives the universal filtration on $\tildefrg$ pulled back from $G/B$ (arising from the $B$-filtration of the standard representation of $\GL_2$):
\begin{equation}\
    0\rightarrow \cO_{\tildefrg}(-1)\rightarrow f^*\cV\rightarrow \cO_{\tildefrg}(1)\rightarrow 0.
\end{equation}
\begin{proposition} \label{propositiondirectimagelinebundle}
    We write $\cU=\cO(\frg)=L[a,b,c,z]$ and $\widetilde{\cU}:=\cO(\frg\times_{\frh/W}\frh)=\cU[h]/(h^2-(a^2+cb))$ which is free of rank two over $\cU$.
\begin{enumerate}
    \item The sheaves $Rf_*\cO_{\tildefrg}(-1)$ and $Rf_*\cO_{\tildefrg}(1)$ concentrate in degree zero and are free of rank two over $\cU$.
    \item Let $\widetilde{\cV}:= \cV\otimes_{\cU}\widetilde{\cU}=f_*f^*\cV$. The sequence of $\widetilde{\cU}$-modules below
        \[\cdots \stackrel{\nu-(z-h)}{\longrightarrow}\widetilde{\cV}\stackrel{\nu-(h+z)}{\longrightarrow}\widetilde{\cV}\stackrel{\nu-(z-h)}{\longrightarrow}\widetilde{\cV}\stackrel{\nu-(h+z)}{\longrightarrow}\cdots\]
        is exact.  Moreover, the $\widetilde{\cU}$-module $\widetilde{\cV}[\nu-(z\pm h)]:=\ker(\widetilde{\cV}\stackrel{\nu-(z\pm h)}{\rightarrow}\widetilde{\cV})$ is free of rank two over $\cU$ and $(\nu-(z\pm h))\cV=(\nu-(z\pm h))\widetilde{\cV}=\widetilde{\cV}[\nu-(z\mp h)]$.
    \item The sequence of $\widetilde{\cU}$-modules
        \[0\rightarrow f_*\cO_{\tildefrg}(-1)\rightarrow \widetilde{\cV}\rightarrow f_*\cO_{\tildefrg}(1)\rightarrow 0\]
        is exact and identifies $f_*\cO_{\tildefrg}(-1)$ with $\widetilde{\cV}[\nu-(z+h)]$.
\end{enumerate}
    
\end{proposition}
\begin{proof}
    (1) The vanishing of higher direct images follows from the same proof for $\cO_{\tildefrg}$ in \cite[Prop. 3.4.1]{bezrukavnikov2008localization} which can be deduced from vanishing results for the Steinberg resolution case, cf. \cite{broer1993line} or \cite[Thm. 5.2.1]{brion2007frobenius}. In detail, by the projection formula, $R\pi_*\cO_{\widetilde{\frg}}(i)=(R\pi_*\cO_{\widetilde{\frg}})\otimes_{\cO_{G/B}}\cO_{G/B}(i)$. Consider the sequence of the bundles $0\rightarrow \widetilde{\cN}=G\times^{B}\frn\rightarrow\widetilde{\frg}\rightarrow G\times^B\frh\rightarrow 0$ on $G/B$ where $\frn$ is the nilpotent cone of $\frb$. This short exact sequence induces a filtration on $\pi_*\cO_{\widetilde{\frg}}$ with the associated graded algebra isomorphic to $\cO(\frh)\otimes \cO_{\widetilde{\cN}}$ (see \cite[Ex. II.5.16]{hartshorne2013algebraic}). Since $\cO_{\widetilde{\cN}}(\pm 1)$ has vanishing higher cohomology, so is $\pi_*\cO_{\widetilde{\frg}}(\pm 1)$. More directly, we can take an increasing filtration by grading on $\cO_{\tildefrg}$ with graded pieces the coherent sheaves $G\times^{B}\Sym^n\frb^*$ which are finite extensions of $\cO(2i),i\geq 0$ and have vanishing higher cohomology even after twisting $\cO_{G/B}(-1)$. Now we show that $Rf_*\cO_{\tildefrg}(\pm 1)$ are free over $\frg$. The dualizing complexes of $\tildefrg$ and $\frg$ are trivial by Lemma \ref{lemmadualizingsheaf} and $f$ is proper. Let $\bbD_{\rm GS}(-)$ be the Grothendieck-Serre duality. We have $\cO_{\tildefrg}(1)=\bbD_{\rm GS}(\cO_{\tildefrg}(-1))[-\dim \frg]$ and hence $Rf_*\cO_{\tildefrg}(1)= \bbD_{\rm GS}(Rf_*\cO_{\tildefrg}(-1))[-\dim \frg]$. Thus $Rf_*\cO_{\tildefrg}(\pm 1)$ are maximal Cohen-Macaulay sheaves, hence locally free, on $\frg$ (cf. \cite[\href{https://stacks.math.columbia.edu/tag/0DWZ}{Tag 0DWZ},\href{https://stacks.math.columbia.edu/tag/090U}{Tag 090U}]{stacks-project}). The (generic) rank is two by Lemma \ref{lemmapropertygrothendieckresolutiongl2}. 

    (2) The composite $(\nu-(h-z))\circ (\nu-(h+z))=0$ is due to $h^2=a^2+bc$ in $\widetilde{\cU}$. A section $x_1e_1+x_2e_2\in \widetilde{\cV},x_1,x_2\in \widetilde{\cU}$ is in the kernel of $\nu-(h+z)$ if and only if $(a-h)x_1+bx_2=cx_1-(a+h)x_2=0$. Using that $\widetilde{\cU}$ is free over $\cU$ with a basis $1,h$, we write $x_i=y_i+hz_i,z_i,y_i\in\cU,i=1,2$. The last condition is equivalent to $y_1=az_1+bz_2,y_2=cz_1-az_2, ay_1-(a^2+bc)z_1+by_2=cy_1-ay_2-(a^2+bc)z_2=0$ and to $y_1=az_1+bz_2,y_2=cz_1-az_2$. We see that there is a $\cU$-surjection $\cV=\cU^2 \twoheadrightarrow \widetilde{\cV}[\nu-(h+z)]:(z_1,z_2)\mapsto ((a+h)z_1+bz_2)e_1+(cz_1-(a-h)z_2)e_2$ which is an isomorphism (for example $(a+h)z_1+bz_2=0$ implies that $z_1=0$). Hence the $\widetilde{\cU}$-map $\nu-(z-h):\widetilde{\cV}\rightarrow \widetilde{\cV}[\nu-(z+h)]$ is surjective. The embedding $\cV\hookrightarrow \widetilde{\cV}$ induces a surjection $\cV\rightarrow \widetilde{\cV}/(\nu-(h+z))=\widetilde{\cV}/(h-(\nu-z))$. Then the map $\nu-(z-h):\widetilde{\cV}/(\nu-(z+h))\rightarrow \widetilde{\cV}[\nu-(z+h)]$ is an isomorphism.

    (3) We identify $\cO_{\tildefrg}(-1)\subset \cV$ with the subsheaf $\cO_{\tildefrg}(g.e_1)$ where $g\in \GL_2(\cO(\GL_2))$ denotes the universal element. The map $\tildefrg\rightarrow \frh$ sends $(\nu,gB)$ to the image of $\Ad(g^{-1})\nu$ in $\frh=\{(h+z,h-z)\}$ via $\frb\rightarrow \frh$. Hence $h+z\in \cO(\frh)$, pulled back to $\tildefrg$, satisfies that $\nu (g.e_1)=g.(\Ad(g^{-1})\nu) e_1=(h+z)(g.e_1)$. Thus $\cO_{\tildefrg}(-1)\subset f^*\cV[\nu-(h+z)]$. Taking direct images we get $f_*\cO_{\tildefrg}(-1)\subset (f_*f^*\cV)[\nu-(h+z)]$. We claim that $\cO_{\tildefrg}(-1)=f^*\cV[\nu-(h+z)]$ (when restricted to $\tildefrg^{\rm reg}$). Then $f_*\cO_{\tildefrg}(-1)=\widetilde{\cV}[\nu-(h+z)]$ when restricted to $\frg^{\rm reg}$ (Lemma \ref{lemmapropertygrothendieckresolutiongl2}), hence on whole $\frg$ by \cite[\href{https://stacks.math.columbia.edu/tag/0EBJ}{Tag 0EBJ}]{stacks-project} and that $\widetilde{\cV}[\nu-(h+z)]$ is a vector bundle by (2). By $G$-equivariance, we may check the identification on the open subspace $U:=\widetilde{\frg}\cap (\{\nu=\mtwo{a+z}{b}{c}{-a+z}\}\times \{\mtwo{1}{}{x}{1}B\})$. Since $\Ad(g^{-1})\nu=\mtwo{a+bx+z}{b}{c-2ax-bx^2}{-a-bx+z}$, we see $h=a+bx$ and $c-2ax-bx^2=0$. A section $x_1e_1+x_2e_2\in f^*\cV(U),x_1,x_2\in \cO_{\tildefrg}(U)$ is in the kernel of $\nu-(h+z)$ if and only if $(a-h)x_1+bx_2=cx_1-(a+h)x_2=0$ if and only if $b(x_2-xx_1)=(2a+bx)(x_2-xx_1)=0$ in $\cO_{\tildefrg}(U)$. The scheme $U$ is integral which implies that $x_2-xx_1=0$. Then $x_1e_1+x_2e_2=x_1(e_1+xe_2)=x_1(g.e_1)\in \cO_{\tildefrg}(-1)(U)$ and we finished the proof.
\end{proof}
\begin{remark}\label{remarklinebundle}
    Those rank two free sheaves in Proposition \ref{propositiondirectimagelinebundle} on $\frg$ are not flat over $\cO_{\frg\times_{\frh/W}\frh}$. The fibers of these sheaves at $a=b=c=z=0$ have dimension $2$ but are not free of rank one over $L[h]/h^2$. For example, by the proof of (2) of the proposition, we know 
    \[f_*\cO_{\tildefrg}(-1)=\widetilde{\cV}[\nu-(h+z)]\simeq (\widetilde{\cU}x_1\oplus\widetilde{\cU}x_2)/((a-h)x_1+bx_2,cx_1-(a+h)x_2).\]
\end{remark}

\section{Change of weights}\label{sectionchangeofweight}
We will construct (in \S\ref{subsectionproductformula}) the change of weights maps and prove the product formula for the completions of the stack $\frX_n$ of $(\varphi,\Gamma)$-modules of rank $n\geq 1$ along fixed Sen weights loci (Proposition \ref{propositionfiberproduct}). In \S\ref{subsectiongeneralchange}, we describe a general construction for families of $(\varphi,\Gamma)$-modules changing possibly non-fixed weights. Then we will study the flatness of the local model map in the $\GL_2(\Q_p)$ case in \S\ref{sectionflatlocalmodel}.

We fix $K$ a finite extension of $\Q_p$ and assume $\Sigma=\Hom(K,L)$ has size $|K:\Q_p|$. Take $\bh=(\bh_{\sigma})_{\sigma\in\Sigma}=(h_{\sigma,1},\cdots,h_{\sigma,n})_{\sigma
\in \Sigma}\in (\Z^n)^{\Sigma}$ such that $h_{\sigma,1}\leq \cdots\leq h_{\sigma,n}$ for all $\sigma$. Let $G=\prod_{\sigma\in\Sigma}\GL_{n/L}$ with the Weyl group $W\simeq(\mathcal{S}_n)^{\Sigma}$ where $\mathcal{S}_n$ is the $n$-th symmetric group and with the Lie algebra $\frg$. Let $P_{\bh}=\prod_{\sigma\in\Sigma}P_{\bh_{\sigma}}$ be the standard parabolic subgroup of $G$ containing $\prod_{\sigma}B$ such that the Weyl group $W_{P_{\bh}}$ of the Levi subgroup of $P_{\bh}$ is the stabilizer subgroup of $\bh$ for the action of $W=\mathcal{S}_n^{\Sigma}$ on $(\Z^n)^{\Sigma}$. We write $\tildefrg_{\bh}$ for $\tildefrg_{P_{\bh}}$ and $f_{\bh}=f_{P_{\bh}}:\tildefrg_{P_{\bh}}\rightarrow \frg$, the analytification of the map in \S\ref{sectinGrothendieckresolution}.
\subsection{Stacks of almost de Rham $(\varphi,\Gamma)$-modules} 
We recall the setting in \cite[\S 5]{emerton2023introduction} and the definition of various stacks.

Let $\Rig_L$ be the category of rigid analytic spaces over $L$ equipped with the Tate-fpqc topology defined in \cite[\S 2.1]{conrad2009non}. By a stack we mean a category fibered in groupoids over $\Rig_L$ satisfying descent for the Tate-fpqc topology. Given a stack $\frX$ over $\Rig_L$ and $\Sp(A)\in \Rig_L$, we write $\frX(A):=\frX(\Sp(A))=\Hom(\Sp(A),\frX)$ for the groupoid lying over $\Sp(A)$ (cf. Yoneda lemma \cite[\href{https://stacks.math.columbia.edu/tag/0GWI}{Tag 0GWI},\href{https://stacks.math.columbia.edu/tag/02XY}{Tag 02XY}]{stacks-project}). Sheaves with values in sets are viewed as stacks fibered in discrete categories.

\begin{example}\label{examplestacks}
    If $Y\in \Rig_L$, then $Y$ defines a sheaf over $\Rig_L$ via Yoneda embedding by \cite[Cor. 4.2.5]{conrad2006relative}. Let $Z\subset Y\in\Rig_L$ be a Zariski-closed subspace and let $\cI$ be the coherent ideal sheaf. We define a subsheaf $Y^{\wedge}\subset Y$ such that for any $X\in \Rig_L$, $Y^{\wedge}(X)$ is the set of morphisms $X\rightarrow Y$ such that the image of $X$ lies set-theoretically in $Z$. For $n\in \bN$, let $\frY_n$ be the closed subspace of $Y$ cut out by $\cI^n$. We get a directed system $\cdots\frY_n\hookrightarrow \frY_{n+1}\cdots $ of sheaves over $\Rig_L$. Then $Y^{\wedge}$ is equal to the sheaf colimit $\varinjlim_n\frY_n$: for an affinoid $\Sp(A)\in \Rig_L$, we have $Y^{\wedge}(A)= \varinjlim_n \frY_n(A)$ (cf. \cite[\href{https://stacks.math.columbia.edu/tag/0738}{Tag 0738},\href{https://stacks.math.columbia.edu/tag/0GXT}{Tag 0GXT},\href{https://stacks.math.columbia.edu/tag/0AIX}{Tag 0AIX}]{stacks-project}). 
\end{example}
For a $(\varphi,\Gamma_K)$-module $D_{L'}$ over $\cR_{L',K}$ for a finite extension $L'$ of $L$, the roots of the Sen polynomial of $D_{L'}$ in $(K\otimes_{\Q_p}L')[T]=\prod_{\sigma\in\Sigma}L'[T]$ are Sen weights of $D_{L'}$, see \cite[Def. 6.2.11]{kedlaya2014cohomology}. The $\sigma$-components of the roots for $\sigma\in \Sigma$ are called $\sigma$-Sen weights.
\begin{definition}
    \begin{enumerate}
        \item Let $L'$ be a finite extension of $L$, then a $(\varphi,\Gamma_K)$-module $D_{L'}$ of rank $n$ over $\cR_{L',K}$ is said to be almost de Rham (resp. has Sen weights $\bh$) if all the Sen weights are integers (resp. for all $\sigma\in\Sigma$, the multiset of $\sigma$-Sen weights of $D_{L'}$ is equal to $\{h_{\sigma,1},\cdots,h_{\sigma,n}\}$.)
        \item Let $X\in \Rig_L$. A $(\varphi,\Gamma_K)$-module $D_X$ of rank $n$ over $\cR_{X,K}$ is said to be almost de Rham (resp. almost de Rham of weight $\bh$) if for every point $x\in X$, the specialization $D_x$ is almost de Rham (resp. has Sen weights $\bh$).
    \end{enumerate}    
\end{definition}
If $X=\Sp(A)$, then a $(\varphi,\Gamma_K)$-module $D_{A}$ over $\cR_{A,K}$ is almost de Rham of weight $\bh$ if and only if its Sen polynomial $P_{\rm Sen}\in (K\otimes_{\Q_p}A)[T]$ is equal to $\prod_{\sigma\in \Sigma}\prod_{i=1}^n(T-h_{\sigma,i})$ modulo the nilradical of $A$.
\begin{definition}\label{defstackalmostderham}
    Let $(\frX_{n})^{\wedge}_{\bh}$ be the category fibered in groupoids over $\Rig_L$ sending  $X\in\Rig_L$ to the groupoid of almost de Rham $(\varphi,\Gamma_K)$-modules over $\cR_{X,K}$ of weight $\bh$ (and rank $n$).
\end{definition}

\begin{definition}\label{defstacknilpotent}
    Assume that $X\in \Rig_L$. Suppose that $D_{\pdR,X}$ is a finite projective rank $n$ module over $\cO_{X}\otimes_{\Q_p}K$ equipped with a decreasing filtration $\Fil^{\bullet}_{X}=(\Fil_X^{i})_{i\in\Z}$ by projective $\cO_X$-submodules of $D_{\pdR,X}$. 
    \begin{enumerate}
        \item The filtration $\Fil_X^{\bullet}$ is said to be of type $\bh$ if for any $\sigma\in\Sigma$, the filtration $\Fil^{\bullet}_{X,\sigma}:=\Fil^{\bullet}_{X}\otimes_{\cO_X\otimes_{\Q_p}K,1\otimes \sigma}\cO_X$ of $D_{\pdR,X,\sigma}:=D_{\pdR,X}\otimes_{\cO_X\otimes_{\Q_p}K,1\otimes \sigma}\cO_X$ satisfies that for all $i\in\Z$, $\Fil^{-i}_{X,\sigma}/\Fil^{-i+1}_{X,\sigma}$ is projective over $\cO_X$ of rank the multiplicity of $i$ in $\{h_{\sigma,1},\cdots,h_{\sigma,n}\}$.
        \item Let $\tildefrg_{\bh}/G$ be the category fibered in groupoids over $\Rig_L$ sending $X\in \Rig_L$ to the groupoid of triples $(D_{\rm pdR,X}, \nu_X,\Fil^{\bullet}_X)$ where $D_{\rm pdR,X}$ is a projective $\cO_X\otimes_{\Q_p}K$-module of rank $n$, $\nu_X$ is an $\cO_X\otimes_{\Q_p}K$-linear endomorphism of $D_{\rm pdR,X}$ and $\Fil^{\bullet}_X$ is a filtration of $D_{\pdR,X}$ by projective sub-$\cO_X$-modules of type $\bh$ stabilized by $\nu$.
        \item Let $\tildefrg_{\bh}$ be the category fibered in groupoids over $\Rig_L$ sending $X\in \Rig_L$ to the groupoid of $(D_{\rm pdR,X}, \nu_X,\Fil^{\bullet}_X,\alpha_X)$ where $(D_{\rm pdR,X}, \nu_X,\Fil^{\bullet}_X)\in (\tildefrg_{\bh}/G)(X)$ and $\alpha_X:D_{\rm pdR,X}\simeq (\cO_X\otimes_{\Q_p}K)^n$ is an isomorphism of $\cO_X\otimes_{\Q_p}K$-modules. 
    \end{enumerate}   
\end{definition}
\begin{lemma}
    The categories fibered in groupoids $(\frX_{n})^{\wedge}_{\bh}$ and $\tildefrg_{\bh}/G$ in Definition \ref{defstackalmostderham} and \ref{defstacknilpotent} define stacks on $\Rig_L$.
\end{lemma}
\begin{proof}
    We need to verify that $(\frX_{n})^{\wedge}_{\bh}$ and $\tildefrg_{\bh}/G$ satisfy descent for Tate-fpqc coverings. The descents for $(\varphi,\Gamma_K)$-modules and $(D_{\pdR,X},\nu_X,\Fil^{\bullet}_X)$ are effective by descents of vector bundles (cf. \cite[Thm. 4.2.8]{conrad2006relative}). The properties of being of weight/type $\bh$, etc., can be checked pointwisely and thus descend.
\end{proof}
If $\bh=0$, the information on filtrations is trivial and in this case, we write $\frg/G$ for $ \tildefrg_{\bh}/G$. Let $f_{\bh}: \tildefrg_{\bh}/G\rightarrow \frg/G$ be the natural morphism of stacks forgetting the filtrations. 
\begin{lemma} The stack $\tildefrg_{\bh}$ is represented by the rigid analytic space $\tildefrg_{\bh}$. The following diagram of stacks
    \begin{center}
        \begin{tikzcd}
            \tildefrg_{\bh}\arrow[r]\arrow[d,"f_{\bh}"]
            & \tildefrg_{\bh}/G \arrow[d,"f_{\bh}"]\\ 
            \frg\arrow[r]&\frg/G
        \end{tikzcd}
    \end{center}
is Cartesian.
\end{lemma}
\begin{proof}
    A filtration $\Fil_{X,\sigma}^{\bullet}$ of type $\bh$ of $D_{\pdR,X,\sigma}\simeq \cO_X$ on $X\in \Rig_L$ is determined by the flags $\Fil^{-h_{\sigma,n}}\supseteq \cdots \supseteq \Fil^{-h_{\sigma,1}}$ for $\sigma\in\Sigma$. With isomorphisms $D_{\pdR,X,\sigma}\simeq \cO_X^n$, one sees by definition that such flags are parametrized by the flag varieties $\GL_n/P_{\bh_{\sigma}}$. An endomorphism $\nu_{X,\sigma}$ of $\cO_X^n$ is equivalent to a map $X\rightarrow\mathfrak{gl}_n$ and the map $X\rightarrow\mathfrak{gl}_n\times \GL_n/P_{\bh_{\sigma}}$ factors through $\widetilde{\mathfrak{gl}_n}_{P_{\bh_{\sigma}}}$ if and only if $\nu_{X,\sigma}$ stabilizes the filtration. The diagram is Cartesian by definitions.
\end{proof}
\subsection{A product formula}\label{subsectionproductformula}
We recall the local model maps and will define the change of weights maps in Proposition \ref{propositionfiberproduct}. We will freely use constructions and results on families of almost de Rham $(\varphi,\Gamma_K)$-modules in Appendix \ref{sectionappendixfamilies}.

Suppose $\Sp(A)\in \Rig_L$ and $D_A\in (\frX_n)^{\wedge}_{\bh}(A)$. Apply Proposition \ref{propositionappendixequivalencealmostdeRham} for the $\Gamma_K$-representations localized from $D_A$ in \S\ref{subsectionbeauvillelaszlo}, we obtain functorially a triple 
\[(D_{\pdR}(D_A), \nu_A, \Fil^{\bullet}D_{\pdR}(D_A))\in (\tildefrg_{\bh}/G)(A).\]
Such construction glues along admissible Tate coverings and is functorial for the base change, hence we get the following statement, which already appeared in \cite[\S 5.3.23]{emerton2023introduction}.
\begin{proposition}
    The functor $D_{\pdR}:D_X\mapsto (D_{\pdR}(D_X), \nu_X, \Fil^{\bullet}D_{\pdR}(D_X))$ induces a (local model) morphism $D_{\pdR}:(\frX_{n})^{\wedge}_{\bh}\rightarrow \tildefrg_{\bh}/G$ of stacks over $\Rig_L$.
\end{proposition}
\begin{remark}
    The map $D_{\pdR}$ factors through the substack $(\widetilde{\frg}_{\bh}/G)_0^{\wedge}$ where all $\nu_X$ are required to be locally nilpotent in the sense of Lemma \ref{lemmanilpotentmatrix} below.
\end{remark}
\begin{lemma}\label{lemmanilpotentmatrix} We consider nilpotent operators.
    \begin{enumerate} 
     \item Let $A$ be a commutative Noetherian ring with the nilradical $I$. Let $\nu\in M_{n}(A)$ be an $n$-by-$n$ matrix. Then $\nu$ is nilpotent if and only if its image in $M_n(A/I)$ satisfies that $\nu^n=0$. And in this case, there is an integer $N$ depending on $A,n$ such that $\nu^N=0$ in $M_n(A)$.
     \item Let $X\in \Rig_L$ and let $\nu\in \End_{\cO_X}(\cO_X^n)$. Then $\nu$ is locally nilpotent ($\nu$ is nilpotent when restricted to any affinoid open $\Sp(A)\subset X$) if and only if for any $x\in X$, the image of $\nu$ in $\End_{k(x)}(k(x)^n)=M_n(k(x))$ is nilpotent.  
    \end{enumerate}
 \end{lemma}
 \begin{proof}
     (1) Suppose that $\nu^n\in M_n(I)$. Since $I$ is nilpotent, we see that there exists an integer $N$ such that $I^N=0$. Then $\nu^{nN}=(\nu^n)^N=0$. Conversely, suppose that $\nu$ is nilpotent. We may assume that $A$ is reduced. For any prime ideal $\frp$ of $A$, the image of $\nu$ in $M_n(A/\frp)\subset M_n(\mathrm{Frac}(A/\frp))$ is nilpotent. Hence $\nu^n\equiv 0\mod \frp$ for all prime ideal $\frp$ of $A$. This implies that $\nu^n=0$. 

     (2) Suppose that $\nu$ is pointwisely nilpotent. We prove that $\nu$ is locally nilpotent. The problem is local and we may assume $X=\Sp(A)$ is an affinoid. Let $I$ be the nilradical of $A$. Then the image of $\nu^n$ in $M_n(A/I)$ is zero since it is true pointwisely. By (1), $\nu$ is nilpotent.
 \end{proof}  
The following lemma allows change of weights and is just the family version of \cite[Lem. 2.1]{Ding2024wall}. See \S\ref{subsectiongeneralchange} for a more direct construction.
\begin{lemma}\label{lemmauniqueweight0}
    Let $\Sp(A)\in\Rig_L$ and $D_A\in (\frX_{n})^{\wedge}_{\bh}(A)$. There exists a unique $(\varphi,\Gamma_K)$-module $f_{\bh}(D_A)\in (\frX_{n})^{\wedge}_{0}(A)$ almost de Rham of weight $0$ such that $f_{\bh}(D_A)$ is a sub-$(\varphi,\Gamma_K)$-module of $D_A[\frac{1}{t}]$ and $f_{\bh}(D_A)[\frac{1}{t}]=D_A[\frac{1}{t}]$. Moreover, the formation $D_A\mapsto f_{\bh}(D_A)$ is functorial and commutes with base change.
\end{lemma}
\begin{proof}
    By definition, there exists $r>0$ such that $D_A=\cR_{A,K}\otimes_{\cR_{A,K}^r}D_A^r$ for a $(\varphi,\Gamma_K)$-module $D_A^r$ over $\cR_{A,K}^r$. We assume that $m(r)$ is large enough in the sense of Definition \ref{definitioninfinitedif}. We construct $f_{\bh}(D_A^r)$ firstly and will let $f_{\bh}(D_A)=\cR_{A,K}\otimes_{\cR_{A,K}^r}f_{\bh}(D_A^r)$. By Proposition \ref{propositionappendixglueing}, the $(\varphi,\Gamma_K)$-modules inside $D_A^r[\frac{1}{t}]$ which equal $D_A^r[\frac{1}{t}]$ after inverting $t$ is in bijection with $\Gamma_K$-invariant $(A\otimes_{\Q_p}K_m)[[t]]$-lattices in $D_{\dif}^m(D_{A}^r)$ for $m=m(r)$. By (1) of Proposition \ref{propositionappendixequivalencealmostdeRham} and also the Step 1 of its proof, there exists a $\Gamma_K$-invariant lattice inside $D_{\dif}^m(D_{A}^r)$ of weight $0$ corresponding to the filtrations $\Fil^{\bullet}$ of type $0$ of $D_{\pdR}(D_A)$, which can only be the trivial filtration:
    \[\Fil^{0}=D_{\pdR}(D_A)\supset \Fil^{1}=\{0\}.\] 
    Thus there exists a sub $(\varphi,\Gamma_K)$-module $f_{\bh}(D_A^r)$ of rank $n$ inside $D_A^r[\frac{1}{t}]$ of weight $0$ such that $f_{\bh}(D_A^r)[\frac{1}{t}]=D_A^r[\frac{1}{t}]$. To show the uniqueness, suppose that $\Delta_A,\Delta_A'\subset D_{A}[\frac{1}{r}]$ are two required $(\varphi,\Gamma_K)$-submodules of weight $0$. Then we obtain a map in
    \[\Hom_{\varphi,\gamma_K}(\Delta_A,\Delta_{A}'[\frac{1}{t}])=\varinjlim_{i\geq 0}\Hom_{\varphi,\gamma_K}(\Delta_A,t^{-i}\Delta_{A}')=\Hom_{\varphi,\gamma_K}(\Delta_A,\Delta_A').\] 
    The last equality follows from that $H^0_{\varphi,\gamma_K}(t^{-i}\Delta_A^{\vee}\otimes_{\cR_{A,K}} \Delta_A'/t^{-i+1}(\Delta_A^{\vee}\otimes_{\cR_{A,K}} \Delta_A'))=0$ for all $i\geq 1$ by Lemma \ref{lemmavanishingH0} below since $\Delta_A^{\vee}\otimes \Delta_A'$ has weight $0$. Hence the identity $\Delta_A[\frac{1}{t}]=\Delta_A'[\frac{1}{t}]$ is induced by a map $\Delta_A\rightarrow \Delta_A'$ which is necessarily an inclusion. We conclude that $\Delta_A=\Delta_A'$. The functoriality follows from the functorialities of the constructions in Appendix \ref{sectionappendixfamilies} or the uniqueness of $f_{\bh}(D_A)$.
\end{proof}
\begin{lemma}\label{lemmavanishingH0}
    Suppose that $D_A$ is a $(\varphi,\Gamma_K)$-module over $\cR_{A,K}$ for an affinoid algebra $A$ over $L$ and is almost de Rham of weight $0$. Then $H^0_{\varphi,\gamma_K}(t^iD_A/t^{i+1}D_A)=0$ for all $i\neq 0$.
\end{lemma}
\begin{proof}
    We know $t^iD_A$ has weights all equal to $i$. Suppose that $D_A=\cR_{A,K}\otimes_{\cR_{A,K}}D_A^r$ for some $(\varphi,\Gamma_K)$-module over $\cR_{A,K}^r$ and $r>0$. Then $t^iD_A/t^{i+1}D_A=\varinjlim_{r'\geq r}t^iD_A^{r'}/t^{i+1}D_A^{r'}=\varinjlim_{r'\geq r}\prod_{m\geq m(r')}D_{\Sen}^{m}(t^iD_A)$, cf. Appendix \ref{sectionappendixfamilies} and \cite[Prop. 2.15]{liu2015triangulation}. Taking $\varphi$-invariants we have $(t^iD_A/t^{i+1}D_A)^{\varphi=1}=D_{\Sen}^{\infty}(t^iD_A)=K_{\infty}\otimes_{K_m}D_{\Sen}^m(t^iD_A)$ (cf. \cite[Prop. 3.2.4, Def. 6.2.11]{kedlaya2014cohomology}). By definition $H^0_{\varphi,\gamma_K}(t^iD_A/t^{i+1}D_A)\subset D_{\Sen}^{\infty}(t^iD_A)^{\gamma_K=1}$. The differential of the $\Gamma_K$-action gives the Sen operator $\nabla_{\Sen}$. Since the Sen weights of $t^iD_A$ are pointwisely $i\neq 0$, $\nabla_{\Sen}-i$ acts locally nilpotently on $D_{\Sen}^{\infty}(D_A)$. One gets that $D_{\Sen}^{\infty}(t^iD_A)^{\gamma_K=1}\subset D_{\Sen}^{\infty}(t^iD_A)^{\nabla_{\rm Sen}=0}=0$.
\end{proof}
Glueing the construction in Lemma \ref{lemmauniqueweight0}, we get a map 
\[f_{\bh}:(\frX_{n})^{\wedge}_{\bh}\rightarrow (\frX_{n})^{\wedge}_{0}\] 
and a commutative diagram of stacks over $\Rig_L$
\begin{center}
    \begin{tikzcd}
        (\frX_{n})^{\wedge}_{\bh}\arrow[r,"D_{\rm pdR}"]\arrow[d,"f_{\bh}"]
        & \tildefrg_{\bh}/G \arrow[d,"f_{\bh}"]&\tildefrg_{\bh}\arrow[l]\arrow[d,"f_{\bh}"]\\ 
        (\frX_n)^{\wedge}_{0}\arrow[r,"D_{\rm pdR}"]&  \frg/G&\frg\arrow[l]
    \end{tikzcd}
\end{center}
where all squares are Cartesian by the following proposition.
\begin{proposition}\label{propositionfiberproduct}
    The functor $\Psi=(f_{\bh},D_{\pdR})$ induces an equivalence
    \[\Psi:(\frX_{n})^{\wedge}_{\bh}\simeq (\frX_n)^{\wedge}_{0}\times_{\frg/G}\tildefrg_{\bh}/G\]
    of stacks over $\Rig_L$.
\end{proposition}
\begin{proof}
    The $2$-fiber product $(\frX_n)^{\wedge}_{0}\times_{\frg/G}\tildefrg_{\bh}/G$ is a stack \cite[\href{https://stacks.math.columbia.edu/tag/026G}{Tag 026G}]{stacks-project} and for an affinoid $\Sp(A)\in \Rig_L$, the groupoid $((\frX_n)^{\wedge}_{0}\times_{\frg/G}\tildefrg_{\bh}/G)(A)$ is the category of tuples 
    \[(\Delta_A, D_{\pdR,A}, \nu_A,\Fil^{\bullet}D_{\pdR,A},\alpha_A)\] 
    where $\Delta_A$ is a $(\varphi,\Gamma_K)$-module of rank $n$ over $\cR_{A,K}$ almost de Rham of weight $0$, a triple 
    \[(D_{\pdR,A}, \nu_A,\Fil^{\bullet}D_{\pdR,A})\in \tildefrg_{\bh}/G(A)\] and $\alpha_A:D_{\pdR,A}\simeq D_{\pdR}(\Delta_A)$ is an isomorphism of $A\otimes_{\Q_p}K$-modules compatible with the nilpotent operators. Suppose that $g_1,g_2\in \mathrm{Isom}(D_{A},D_A')$ for $D_A,D_A'\in (\frX_n)^{\wedge}_{\bh}$. If $\Psi(g_1)=\Psi(g_2)$, then $g_1=g_2:D_A[\frac{1}{t}]=f_{\bh}(D_A)[\frac{1}{t}]\rightarrow D_A'[\frac{1}{t}]=f_{\bh}(D_A')[\frac{1}{t}]$. Since $D_A$ is $t$-torsion free, we see $g_1=g_2$. Now suppose $g$ is an isomorphism $(f_{\bh}(D_A),\Fil^{\bullet}D_{\pdR}(D_A),\alpha_A)\simeq (f_{\bh}(D_A'),\Fil^{\bullet}D_{\pdR}(D_A'),\alpha_A')$ where $\alpha_A,\alpha_A'$ are the natural identifications, e.g., $D_{\pdR}(D_A)=D_{\pdR}(D_A[\frac{1}{t}])=D_{\pdR}(f_{\bh}(D_A))$. By \cite[Lem. 2.2.9]{kedlaya2014cohomology}, the morphism $g:f_{\bh}(D_A)\rightarrow f_{\bh}(D_A')$ is induced by some map $f_{\bh}(D_A)^r\rightarrow f_{\bh}(D_A')^r$ uniquely determined by $f_{\bh}(D_A)$ and $f_{\bh}(D_A')$ for some $r>0$ such that $m(r)$ is large enough. By the uniqueness in the proof of Lemma \ref{lemmauniqueweight0}, $f_{\bh}(D_A)^r=f_{\bh}(D_A^r)$ and $f_{\bh}(D_A')^r=f_{\bh}((D_A')^r)$. Then Proposition \ref{propositionappendixequivalencealmostdeRham} and Proposition \ref{propositionappendixglueing} implies that there exists $g':D_A^r\simeq (D_A')^r$ which induces $g$. This shows that $\Psi$ is fully faithful by \cite[\href{https://stacks.math.columbia.edu/tag/04WQ}{Tag 04WQ}]{stacks-project}. Also given $(\Delta_A, D_{\pdR,A}, \nu_A,\Fil^{\bullet}D_{\pdR,A},\alpha_A)$, the triple $(D_{\pdR,A}, \nu_A,\Fil^{\bullet}D_{\pdR,A})$ and $\alpha_A$ define a lattice $D_{\dif}^{m,+}$ of $D_{\dif}^{m}(\Delta_A)$ by Proposition \ref{propositionappendixequivalencealmostdeRham} for $m=m(r)$ if $\Delta_A$ is the base change from a $(\varphi,\Gamma_K)$-module over $\cR_{A,K}^r$ for some $r$ such that $m(r)$ is large enough, which gives a modification $D_A$ of $\Delta_A$ by Proposition \ref{propositionappendixglueing}. This shows that the tuple lies in the essential image of $\Psi_A$. Then $\Psi$ is an equivalence by \cite[\href{https://stacks.math.columbia.edu/tag/046N}{Tag 046N}]{stacks-project}.  
\end{proof} 
\begin{corollary}
    The map $f_{\bh}:(\frX_{n})^{\wedge}_{\bh}\rightarrow (\frX_n)^{\wedge}_{0}$ is projective, i.e., for any $\Sp(A)\in \Rig_L$ with $\Sp(A)\rightarrow (\frX_n)^{\wedge}_{0}$, the fiber product $f_{\bh}^{-1}(\Sp(A)):=\Sp(A)\times_{(\frX_n)^{\wedge}_{0}}(\frX_n)^{\wedge}_{\bh}$ is isomorphic to a rigid analytic space projective over $\Sp(A)$. 
\end{corollary}
\begin{proof}
    By Proposition \ref{propositionfiberproduct}, $\Sp(A)\times_{(\frX_n)^{\wedge}_{0}}(\frX_n)^{\wedge}_{\bh}=\Sp(A)\times_{\frg/G}\tildefrg_{\bh}/G$. Let $(D_{\pdR,A},\nu_A)$ be the universal $A\otimes_{\Q_p}K$-module over $A$ induced from $\Sp(A)\rightarrow \frg/G$. Then $f_{\bh}^{-1}(\Sp(A))$ is the stack over $\Rig_L/\Sp(A)$ of $\nu_A$-stable filtrations on $D_{\pdR,A}$ of type $\bh$. This stack is representable by the representability of Grassmannians and that being $\nu_A$-stable is a Zariski-closed condition (essentially given by vanishing of matrix coefficients for morphisms between vector bundles).
\end{proof}
\begin{remark}
    Locally, we can choose a trivialization $D_{\pdR,A}\simeq (A\otimes_{\Q_p}K)^n$ in the above proof, equivalently choose a section $\Sp(A)\rightarrow \frg$ for $\Sp(A)\rightarrow \frg/G$. The map $ \Sp(A)\rightarrow \frg$ is defined by $\nu_A$ with the set-theoretical image contained in the nilpotent cone $\cN$, and $f_{\bh}^{-1}(\Sp(A))=\Sp(A)\times_{\frg}\tildefrg_{\bh}$.
\end{remark}

\subsection{Change of weights for general families}\label{subsectiongeneralchange}
We point out that change of weights for $(\varphi,\Gamma_K)$-modules may work for more general families, without pointwisely fixed Sen weights. Let $D_A$ be a $(\varphi,\Gamma_K)$-module over $\cR_{A,K}$ for an affinoid $\Sp(A)\in\mathrm{Rig}_L$. We fix $\sigma\in \Sigma$ and write $P_{\rm Sen}(T)\in (A\otimes_{\Q_p}K)[T]$ for the Sen polynomial of $D_A$, and $P_{\Sen,\sigma}(T)$ for its $\sigma$-component via $A\otimes_{\Q_p}K\simeq \prod_{\sigma}A$. We will call $P_{\Sen,\sigma}(T)$ the $\sigma$-Sen polynomial in the following. 

\begin{lemma}\label{lemmacoprimepolynomials}
    Two polynomials $Q(T)$ and $S(T)$ in $A[T]$ are coprime to each other ($(Q(T),S(T))=(1)$) if and only if for any $x\in \Sp(A)$, the sets of roots of $Q(T)\otimes_Ak(x)$ and of $S(T)\otimes_Ak(x)$ in $\overline{k(x)}$ have empty intersection.
\end{lemma}
\begin{proof}
    The condition that $(Q(T),S(T))=(1)$ is equivalent to that there is no maximal ideal $\frm$ of $A[T]$ containing both $S(T)$ and $P(T)$. Any maximal ideal $\frm$ of $A[T]$ lies over a point $x\in \Sp(A)$ \cite[\href{https://stacks.math.columbia.edu/tag/00GB}{Tag 00GB}]{stacks-project}. The result follows.
\end{proof}
\begin{proposition}
    Suppose that the $\sigma$-Sen polynomial $P_{\Sen,\sigma}(T)$ of $D_A$ admits a decomposition $P_{\Sen,\sigma}(T)=Q(T)S(T)$ in $A[T]$ by monic polynomials such that $(Q(T),S(T))=(1)$. Then there exists a unique $(\varphi,\Gamma_K)$-module $D_A'$ over $\cR_A$ contained in $D_A$ and containing $tD_A$ such that the Sen polynomial of $D_A'$ is equal to $Q(T-1)S(T)\prod_{\sigma'\neq\sigma}P_{\Sen,\sigma}(T)\in \prod_{\sigma\in\Sigma}A[T]$.
\end{proposition}
\begin{proof}
    We suppose that $D_A=D_A^r\otimes_{\cR_{A,K}^r}\cR_{A,K}$ for a $(\varphi,\Gamma_K)$-module $D_A^r$ over $\cR_{A,K}^r$. By Proposition \ref{propositionappendixglueing}, it is enough to prove the following statement for a semilinear $\Gamma_K$-representation $D_{\dif,A,\sigma}^{m,+}$ over $(A\otimes_{\sigma,K}K_m)[[t]]$: suppose that the characteristic polynomial for the Sen operator $\nabla=\nabla_{\Sen}$ on $D_{\dif,A,\sigma}^{m,+}/t$ is equal to $P_{\Sen,\sigma}(T)=Q(T)S(T)$ such that $(Q(T),S(T))=(1)$, then there exists a unique sub-$\Gamma_K$-representation $M$ contained in $D_{\dif,A,\sigma}^{m,+}$ and containing $tD_{\dif,A,\sigma}^{m,+}$ such that the Sen polynomial of $M/t$ is equal to $Q(T-1)S(T)$. 
    
    Since $(Q(T),S(T))=1$, $A[T]/P_{\Sen,\sigma}(T)=(A[T]/Q(T))\times (A[T]/S(T))$ by the Chinese remainder theorem. As $P_{\Sen,\sigma}(\nabla)$ annihilates $D_{\dif,A,\sigma}^{m,+}/t$, this leads to a canonical decomposition $D_{\dif,A,\sigma}^{m,+}/t=M_{Q}\oplus M_{S}$ where $Q(\nabla)$ kills $M_{Q}$ and $S(\nabla)$ is invertible on $M_{Q}$. Since the actions of $\Gamma_K, A$ and $K_m$ commute with $\nabla$, $M_{Q}$ and $M_S$ are $\Gamma_K$-stable projective $A\otimes_{\sigma,K}K_m$-modules. 
    
    We claim that $M_{Q}$ is projective over $A\otimes_{\sigma,K}K_m$ of rank $\deg(Q)$ with characteristic polynomial of $\nabla$ equaling to $Q(T)$. We can check the rank at points $\Sp(L')\rightarrow \Spec(A)$ for a finite extension $L'$ over $L$ and reduce to the case when $A=L'$ such that $\Hom_K(K_m,L')=[K_m:K]$. Then $D_{\dif,L',\sigma}^{m,+}/t=\prod_{\sigma'\in \Hom_K(K_m,L')} D_{\Sen,L',\sigma'}^{m}$ where $D_{\Sen,L',\sigma'}^{m}=D_{\Sen,L'}^{m}\otimes_{L'\otimes_{\sigma,K}K_m, 1\otimes \sigma'}L'$. And $\Gamma_K$ permutes and induces $\nabla$-equivariant isomorphisms between different $\sigma'$-factors. We have decompositions $D_{\Sen,L',\sigma'}^{m}=M_{\sigma',Q}\oplus M_{\sigma',S}$ for all $\sigma'$. Up to enlarging $L'$, the decomposition refines to a decomposition by generalized eigenspaces for $\nabla$ whose dimensions are given by multiplicities of roots of $P_{\Sen,\sigma}$. Then we see the rank over $L'$ of each $M_{\sigma',Q}$ is equal to the degree of $Q$ (and with the characteristic polynomial of $\nabla$ equaling $Q(T)$). Hence $M_{Q}=\prod_{\sigma'}M_{\sigma',Q}$ is projective over $L'\otimes_{\sigma,K}K_m$ with the expected rank. Return to general $A$, let $Q'(T), S'(T)$ be the characteristic polynomials of $\nabla$ on $M_{Q},M_S$. Then $Q'(T)S'(T)=Q(S)S(T)$ and $(Q'(T),S'(T))=(Q(T),S'(T))=(Q'(T),S(T))=1$ by Lemma \ref{lemmacoprimepolynomials}. We conclude that $A[T]/Q'(T)=A[T]/Q(T)$, hence $Q'(T)=Q(T)$ since both are monic polynomials of the same degree.
    
    We take $M:=\ker(D_{\dif,A,\sigma}^{m,+}\rightarrow D_{\dif,A,\sigma}^{m,+}/t\twoheadrightarrow M_Q)$. There is a $\Gamma_K$-filtration of $(A\otimes_{\sigma,K}K_m)[[t]]$-submodules  
    \[t^2D_{\dif,A,\sigma}^{m,+}\subset tM\subset tD_{\dif,A,\sigma}^{m,+}\subset M\subset D_{\dif,A,\sigma}^{m,+}\]
    with graded pieces $tM_{S},tM_{Q},M_{S},M_{Q}$. Then $M/tM$ admits a $\Gamma_K$-filtration with graded pieces $tM_Q$ and $M_S$. Since $\nabla(tx)=t(\nabla+1)x$ for $x\in M_Q$, the characteristic polynomial of $\nabla$ on $M/tM$ is $Q(T-1)S(T)$. And $M$ is finite projective over $(A\otimes_{\sigma,K}K_m)[[t]]$ by Lemma \ref{lemmaprojectivelattice} below. The uniqueness comes from the uniqueness of the decomposition $M=M_{Q}\oplus M_S$.
\end{proof}
\begin{lemma}\label{lemmaprojectivelattice}
    Let $B$ be a Noetherian ring. Let $M$ be a submodule of a finite projective $B[[t]]$-module $D$ containing $t^kD$ such that $D/M$ is finite flat over $B[[t]]/t^k$. Then $M$ is a finite projective $B[[t]]$-module of the same rank as $D$. 
\end{lemma} 
\begin{proof} 
    Certainly $M$ is finite over $B[[t]]$. Use the sequence $0\rightarrow M\rightarrow D\rightarrow D/M\rightarrow 0$ and that $D$ is flat over $B[[t]]$, we have $\Tor^{B[[t]]}_{i}(-,M)=\Tor^{B[[t]]}_{i+1}(-,D/M)$ for $i\geq 1$. For any $B[[t]]$-module $N$, there is a spectral sequence \cite[\href{https://stacks.math.columbia.edu/tag/061Y}{Tag 061Y}]{stacks-project} 
    \[\Tor_n^{B[[t]]/t^k}(\Tor_m^{B[[t]]}(N,B[[t]]/t^k),D/M)\Rightarrow \Tor^{B[[t]]}_{n+m}(N,D/M).\] 
    The flatness of $D/M$ over $B[[t]]/t^k$ implies that $\Tor_n^{B[[t]]/t^k}(\Tor_m^{B[[t]]}(N,B[[t]]/t^k),D/M)=0$ for $n\geq 1$. Thus the spectral sequence degenerates at the $E_2$-page and 
    \[\Tor^{B[[t]]}_i(-,D/M)=\Tor^{B[[t]]}_i(-,B[[t]]/t^k)\otimes_{B[[t]]/t^k}D/M\] 
    for all $i\geq 0$. Since $B[[t]]/t^k$ admits a flat resolution $0\rightarrow t^kB[[t]]\rightarrow B[[t]]\rightarrow B[[t]]/t^k\rightarrow 0 $, $\Tor_i^{B[[t]]}(-,B[[t]]/t^k)=0$ for all $i\geq 2$. Hence $\Tor^{B[[t]]}_{i}(-,M)=0$ for all $i\geq 1$. This implies that $M$ is a finite flat $B[[t]]$-module. 
\end{proof}
\begin{example}
    Suppose that $D_A$ has weights $\bh\in (\Z^n)^{\Sigma}$ as in Lemma \ref{lemmauniqueweight0}. Pick $\sigma\in\Sigma$ and assume that $\{h_{\sigma,1},\cdots,h_{\sigma,n}\}=\{-k_{1},\cdots,-k_s\}$ as sets where $-k_1<\cdots<-k_s$ and each $-k_i$ appears $m_i$ times in $\bh_{\sigma}$. Let $I$ be the nilradical of $A$. Then $P_{\Sen,\sigma}(T)\equiv \prod_{i=1}^s(T+k_i)^{m_i}\mod I$. By Hensel's lemma \cite[\href{https://stacks.math.columbia.edu/tag/0ALI}{Tag 0ALI}]{stacks-project}, there exist coprime monic polynomials $Q(T),S(T)$ such that $P_{\Sen,\sigma}(T)=Q(T)S(T)$ and $Q(T)\equiv (T+k_1)^{m_1}\mod I$. The above proposition gives a $(\varphi,\Gamma_K)$-module $D_A'\subset D_A$ such that $D_A'[\frac{1}{t}]=D_A[\frac{1}{t}]$ of $\sigma$-weights $h_{\sigma,1}+1=\cdots=h_{\sigma,m_{1}}+1\leq h_{\sigma,m_1+1}\leq \cdots$. Repeating such procedures for all $\sigma$ and multiplying suitable powers of $t$, we can find in the end $D_A'$ such that $D_A'=f_{\bh}(D_A)$ is almost de Rham of weight $0$ and $D_A'[\frac{1}{t}]=D_A[\frac{1}{t}]$.
\end{example}
\subsection{Flatness of the local model map}\label{sectionflatlocalmodel}
We prove that the local model map $D_{\pdR}: (\frX_n)^{\wedge}_{0}\rightarrow \frg/G$ is flat in the sense of Corollary \ref{corollaryflatlocalrings} and in the case $n=2$, $K=\Q_p$. This part is to explain Hypothesis \ref{hypothesisflat} that will appear in our main theorem (Theorem \ref{theorem} and also Proposition \ref{propositiondirectimagephigammamodule}). From now on we assume $n=2$ and $K=\Q_p$. The major tool will be miracle flatness. Most proofs in this section work for more general situations, except for the very last part of the proof of Proposition \ref{propositionflatversalmaps}.

We will define and study flatness via morphisms between complete local rings. We first recall the definition of versal rings of stacks assuming existence.
\begin{definition}
    Let $\frX$ be a stack over $\mathrm{Rig}_L$ with a morphism $x:\Sp(L')\rightarrow \frX$ where $L'$ is a finite extension of $L$. Then $x$ corresponds to an object in $\frX(L')$ denoted by $D_{L'}$. 
    \begin{enumerate}
        \item We define $\cF_{\frX,x}$ to be the groupoid fibered in $\cC_{L'}$ sending $A'\in\cC_{L'}$ to pairs $(D_{A'}\in \frX(A'),\iota_{A'}:D_{A'}\otimes_{A'}L'\simeq D_{L'})$ where we write $D_{A'}\otimes_{A'}L'$ for the pullback of $D_{A'}$ in $\frX$ along $\Sp(A'/\frm_{A'})\rightarrow \Sp(A')$. A morphism $(A',D_{A'},\iota_{A'})\rightarrow (A'',D_{A''},\iota_{A''})$ in $\cF_{\frX,x}$ is a map $A'\rightarrow A''$ in $\cC_{L'}$ together with an isomorphism $D_{L'}\otimes_{A'}A''\simeq D_{A''}$ compatible with $\iota_{A'}$ and $\iota_{A''}$, cf. \cite[\href{https://stacks.math.columbia.edu/tag/07XD}{Tag 07XD}]{stacks-project}. 
        \item A formal object of $\frX$ is a complete Noetherian local ring $(R,\frm_R)$ with $L'=R/\frm_R$ finite over $L$ and objects $D_{R/\frm_R^n}\in\frX(R/\frm_R^n)$ together with isomorphisms $D_{R/\frm_R^n}\otimes_{R/\frm_R^n}R/\frm_R^{n-1}\simeq D_{R/\frm_R^{n-1}}$ for all $n$ \cite[\href{https://stacks.math.columbia.edu/tag/07X3}{Tag 07X3}]{stacks-project}. This formal object is versal at the map $x:\Sp(L')\rightarrow \frX$ corresponding to $D_{R/\frm_R}\in D(L')$ if the induced map $\Spf(R)\rightarrow \cF_{\frX,x}$ is formally smooth. In this case we say $R$ a versal ring of $\frX$ at $x$.  
    \end{enumerate}   
\end{definition}

\begin{remark}
    Let $A'$ be an Artin local $L$-algebra with residue field $L'$ finite over $L$, then $A'$ is an $L'$-algebra in a unique way such that $L'\rightarrow A'\rightarrow A'/\frm_{A'}=L'$ is the identity map (i.e., $A'\in \cC_{L'}$) since $L'$ is formally \'etale over $L$ \cite[\href{https://stacks.math.columbia.edu/tag/04G3}{Tag 04G3}]{stacks-project}. If $A$ is an affinoid algebra and $\frm$ is a maximal ideal of $A$ with residue field $L'$, then the completion $\widehat{A}_{\frm}$ is an $L'$-algebra and is a versal ring of $A$ at the point $x:\Sp(L')\rightarrow \Sp(A)$ pro-representing $\cF_{\Sp(A),x}$. If $L''$ is a finite extension of $L'$ and we let $x'':\Sp(L'')\rightarrow \Sp(L')\rightarrow \Sp(A)$, then $\cF_{\Sp(A),x''}$ is pro-represented by the base change $\widehat{A}_{\frm}\otimes_{L'}L''$.  Furthermore, if $\frY^{\wedge}$ is the completion of $\frY=\Sp(A)$ with respect to an ideal $I\subset A$ in the way of Example \ref{examplestacks} with an $L'$-point $x:\Sp(L')\rightarrow \frY_1\rightarrow \frY^{\wedge}$, then by definition $\cF_{\frY,x}=\cF_{\frY^{\wedge},x}$ are both pro-represented by $\widehat{A}_{\frm}=\widehat{A}^{\wedge}_{\frm}$.
\end{remark}

We choose an $L$-point $x\in (\frX_{2})^{\wedge}_0(L)\subset \frX_{2}(L)$ with Sen weights $(0,0)$ corresponding to a $(\varphi,\Gamma)$-module over $\cR_L$. Let $X_{D_L}=\cF_{\frX_2,x}$ be the deformation problem over $\cC_L$ sending $A\in\cC_L$ to the groupoid of  pairs $(D_A,\iota_A)$ where $D_{A}$ is a $(\varphi,\Gamma)$-module over $\cR_A$ and $\iota_{A}:D_A/\frm_A\simeq D_L$. Let $x_{\pdR}=D_{\pdR}(D_L)$ be the image of $x$ in $(\frg/G)(L)$ given by $(D_{\pdR}(D_L),\nu_L)$ and write $X_{x_{\pdR}}$ be the deformation problem of $(D_{\pdR}(D_L),\nu_L)$ over $\cC_L$ \cite[\S 3.1]{breuil2019local}. 
\begin{proposition}\label{propositionformallysmooth}
    If $D_L$ is not a twist by a character of an extension of $t^{-1}\cR_L(\epsilon)$ by $\cR_L$ (written as $[\cR_L-t^{-1}\cR_{L}(\epsilon)]$), then $D_{\pdR}:X_{D_L}\rightarrow X_{x_{\pdR}}$ is formally smooth.
\end{proposition}
\begin{proof}
    To show formally smoothness, we need to show that for any surjction $A'\rightarrow A=A'/I$ in $\cC_L$ such that $\frm_{A'}I=0$, any deformation $(D_{A},\iota)$ of $D_L$ and deformation $(D_{\pdR,A'},\nu_{A'})$ with an isomorphism $(D_{\pdR,A'},\nu_{A'})\otimes_{A'}A\simeq (D_{\pdR}(D_{A}),\nu_{A})$, there exists $(D_{A'},\iota_{A'})\in X_{D_L}(A')$ such that there exists an isomorphism $D_{\pdR}(D_{A'})\simeq D_{\pdR,A'}$ compatible with $\nu_{A'}$ and induces the corresponding isomorphism modulo $I$, see \cite[\href{https://stacks.math.columbia.edu/tag/06HF}{Tag 06HF}]{stacks-project}. 

    It's more convenient for us to use the language of $B$-pairs: the equivalence between $(\varphi,\Gamma)$-modules and $B$-pairs \cite{berger2008construction} and the equivalence between $X_{x_\pdR}$ and deformations of almost de Rham $B_{\rm dR}$-representations \cite[Lem. 3.1.4]{breuil2019local}. We follow the proof and notation of \cite[Prop. 2.30]{nakamura2014deformations}. Write $W=(W_e, W_{\rm dR}^+)=(W_e(D_L),W_{\rm dR}^+(D_L))$ and $W_{A}=(W_{e,A}, W_{\rm dR,A}^+)=(W_e(D_{A}),W_{\rm dR}^+(D_{A}))$. Write $\End(W)=W^{\vee}\otimes W$ where the tensor is in the category of $B$-pairs. Choose basis of $W_{e,A}$ and $W_{{\rm dR},A}^+$. Then $W_A$ gives us $1$-cocycles $\rho_e:\cG_{\Q_p}\rightarrow \GL_2(B_e\otimes_{\Q_p}A)$, $\rho_{\rm dR}:\cG_{\Q_p}\rightarrow \GL_2(B_{\rm dR}^+\otimes_{\Q_p}A)$ and a matrix $P\in\GL_2(B_{\rm dR}\otimes_{\Q_p}A)$ such that $P\rho_e(g)g(P)^{-1}=\rho_{\rm dR}(g)$ for any $g\in \cG_{\Q_p}$. Choose an $L$-linear section $s:A\rightarrow A'$ of $A'\rightarrow A$ which gives us lifts $\widetilde{\rho}_e:=s\circ \rho_{e},\widetilde{\rho}_{\rm dR}:=s\circ \rho_{\rm dR}$ and $\widetilde{P}$. These elements define $2$-cocyles. For example, $c_{\rm dR}^2\in I\otimes_L Z^2(\cG_{\Q_p},\End_{B_{\rm dR}^+\otimes_{\Q_p}L}(W_{\rm dR}^+))$ is defined such that (use $\End_{B_{\rm dR}^+\otimes_{\Q_p}A'}(W_{\rm dR,A}^+\otimes_{A}A')\otimes_{A'}I=I\otimes_L\End_{B_{\rm dR}^+\otimes_{\Q_p}L}(W_{\rm dR}^+)$)
    \[c_{\rm dR}^2(g_1,g_2)=\widetilde{\rho}_{\rm dR}(g_1g_2)g_1(\widetilde{\rho}_{\rm dR}(g_2))^{-1}\widetilde{\rho}_{\rm dR}(g_1)^{-1}-1, \forall g_1,g_2\in\cG_{\Q_p}.\] 
    The vanishing of $H^2(\cG_{\Q_p}, \End(W))$ (Lemma \ref{lemmaH2weight0} below) implies that in the class of $(\widetilde{\rho}_{e},\widetilde{\rho}_{\rm dR},\widetilde{P})$ there exists always a lift $(\rho_{e,A'},\rho_{\rm dR,A'},P_{A'})$ which defines a $B$-pair $W_{A'}=(W_{e,A'},W_{{\rm dR},A'}^+)$ over $A'$ deforming $W_A$. Moreover, standard arguments show that the set of deformations is an affine space under $I\otimes_LH^1(\cG_{\Q_p},\End(W))$ (see the proof of \cite[Lem. 2.28]{nakamura2014deformations}). For example, another lift $\rho_{{\rm dR},A'}'$ of $\rho_{\rm dR}$ defines a $1$-cocycle in $I\otimes_LZ^{1}(\cG_{\Q_p},\End_{B_{\rm dR}^+\otimes_{\Q_p}L}(W_{\rm dR}^+))$:
    \[c^1_{\rm dR}(g)=\rho_{{\rm dR},A'}'(g)\rho_{{\rm dR},A'}(g)^{-1}-1,\forall g\in\cG_{\Q_p}.\]
    Now we consider lifts of $W_{{\rm dR},A}$ to $W_{{\rm dR},A'}$. An easier argument shows that the set of deformations is parametrized by $I\otimes_{L}H^1(\cG_{\Q_p},\End_{B_{\rm dR}\otimes_{\Q_p}L}(W_{\rm dR}))$. The map between lifts induced by $D_{\pdR}$ corresponds to the natural map $I\otimes_LH^1(\cG_{\Q_p},\End(W))\rightarrow I\otimes_{L}H^1(\cG_{\Q_p},\End_{B_{\rm dR}\otimes_{\Q_p}L}(W_{\rm dR}))$. We conclude that the existence of deformations to $A' $ of $W_{A}$ with given image under $D_{\pdR}$ is equivalent to that the map $H^1(\cG_{\Q_p},\End(W))\rightarrow H^1(\cG_{\Q_p},\End_{B_{\rm dR}\otimes_{\Q_p}L}(W_{\rm dR}))$ between tangent spaces is surjective (cf. \cite[\href{https://stacks.math.columbia.edu/tag/0E3R}{Tag 0E3R}]{stacks-project}).
    
    To see when $H^1(\cG_{\Q_p},\End(W))\rightarrow H^1(\cG_{\Q_p},W_{\rm dR})$ is surjective, we go back to the language of $(\varphi,\Gamma)$-modules (we can also use $B$-quotients as in \cite{kedlaya2009some}). Consider the long exact sequence (see Lemma \ref{lemmaH2weight0} below)
    \[\cdots \rightarrow H^1_{\varphi,\gamma}(\End_{\cR_L}(D_L))\rightarrow H^1_{\varphi,\gamma}(\End_{\cR_L}(D_L)/t\End_{\cR_L}(D_L))\rightarrow H^2_{\varphi,\gamma}(t\End_{\cR_L}(D_L))\rightarrow 0. \]
    Using Lemma \ref{lemmaH1weight0} below, the map $H^{1}_{\varphi,\Gamma}(\End_{\cR_L}(D_L))\rightarrow H^1(\cG_{\Q_p}, W_{\rm dR}(\End_{\cR_L}(D_L)))$ is surjective if and only if $ H^2_{\varphi,\gamma}(t\End_{\cR_L}(D_L))=0$. By local Tate duality, $H^2_{\varphi,\gamma}(t\End_{\cR_L}(D_L))^{\vee}=H^0_{\varphi,\gamma}(t^{-1}\End_{\cR_L}(D_L)^{\vee}(\epsilon))=H^0_{\varphi,\gamma}(t^{-1}D_L\otimes_{\cR_L}D_L^{\vee}(\epsilon))$ which is non zero if and only if there exists a non-zero morphism $f:D_L\rightarrow D_L(\epsilon z^{-1})$ of $(\varphi,\Gamma)$-modules. Let $f\neq 0$ be such a map. The kernel and image of $f$ are $(\varphi,\Gamma)$-modules \cite[Prop. 1.1.1]{berger2008construction}. If $D_L$ is irreducible, then we get an injection $D_L\hookrightarrow D_L(\epsilon z^{-1})$ which must be an isomorphism as both modules have weights zero (cf. Lemma \ref{lemmauniqueweight0}). This is not possible considering $\varphi$-slopes (as $\epsilon z^{-1}(p)=p^{-1}$). Hence we may suppose that $D_L$ is split trianguline, namely there exist smooth characters $\delta_1,\delta_2:\Q_p^{\times}\rightarrow L^{\times}$ and a short exact sequence of $(\varphi,\Gamma)$-modules 
    \[0\rightarrow \cR_L(\delta_1)\rightarrow D_L\rightarrow \cR_L(\delta_2)\rightarrow 0.\]
    We may also suppose that $\cR_L(\delta_1)$ is the rank one kernel of $f$. Then we get an injection $\cR_L(\delta_2)\hookrightarrow D_L(\epsilon z^{-1})$. Hence $\delta_1\epsilon z^{-1}=\delta_2$ since $\Hom_{\varphi,\gamma}(\cR_L(\delta),\cR_L(\delta'))=H^0_{\varphi,\gamma}(\cR_L(\delta'\delta^{-1}))\neq 0$ for two smooth characters $\delta,\delta'$ if and only if $\delta=\delta'$. Then under the assumption that $ D_L$ is not of this form, $D_{\pdR}$ is formally smooth at $D_L$.
\end{proof}
\begin{remark}
    In the case that $D_{\pdR}$ is not smooth at $D_L$, $f_{\bh}^{-1}(D_L)$ may contain non-smooth points, see Lemma \ref{lemmaH2weight0} below.
\end{remark}
\begin{lemma}\label{lemmaH1weight0}
    Suppose that $D_L$ is a $(\varphi,\Gamma)$-module of Hodge-Tate-Sen weights all $0$. Then the map $H^1_{\varphi,\gamma}(D_L)\rightarrow H^1(\cG_{\Q_p},W_{\rm dR}(D_L))$ factors through $H^1_{\varphi,\gamma}(D_L/tD_L)$ and induces an isomorphism $H^1_{\varphi,\gamma}(D_L/tD_L)\simeq H^1(\cG_{\Q_p},W_{\rm dR}^+(D_L))$.
\end{lemma}
\begin{proof}
    Since $W_{\rm dR}^+(D_L)$ has weights $0$, the map $H^1(\cG_K, W_{\rm dR}^+(D_L))\rightarrow H^1(\cG_{\Q_p}, W_{\rm dR}(D_L))$ is an isomorphism. Moreover $H^1(\cG_{\Q_p}, tW_{\rm dR}^+(D_L))=0$ is $0$ (cf. \cite[Cor. 5.6]{nakamura2014deformations}). We get isomorphisms 
    \[H^1(\cG_{\Q_p}, W_{\rm dR}(D_L))=H^1(\cG_{\Q_p}, W_{\rm dR}^+(D_L))=H^1(\cG_{\Q_p}, W_{\rm dR}^+(D_L)/t).\] 
    There is a factorization $H^1_{\varphi,\gamma}(D_L)\rightarrow H^1(\cG_{\Q_p}, W_{\rm dR}^+(D_L))\rightarrow H^1(\cG_{\Q_p}, W_{\rm dR}(D_L))$, see \cite[\S 2.1]{nakamura2009classification}. The cohomology of $D_L$ and $tD_{L}$ can be computed as $\cG_{\Q_p}$-cohomology of complexes in the first two columns of the following short exact sequence of complexes of $\cG_{\Q_p}$-modules (see \textit{loc. cit.}):
    \begin{center}
        \begin{tikzcd}
            W_e(tD_L)\oplus W_{\rm dR}^+(tD_L)\arrow[r]\arrow[d]
            & W_e(D_L)\oplus W_{\rm dR}^+(D_L)\arrow[d]\arrow[r]
            & W_{\rm dR}^+(D_L)/t\arrow[d]\\ 
            W_{\rm dR}(tD_L)\arrow[r]
            &  W_{\rm dR}(D_L)\arrow[r]
            &0.
        \end{tikzcd}
    \end{center}
    By comparing the long exact sequence for the cohomology of $0\rightarrow tD_L\rightarrow D_L\rightarrow D_L/tD_L\rightarrow 0$ and using five lemma, we see $H^i_{\varphi,\gamma}(D_L/t)\simeq H^i(\cG_{\Q_p},W_{\rm dR}^+(D_L)/t)$ for $i=0,1$. 
\end{proof}
\begin{lemma}\label{lemmaH2weight0}
    If $D_L$ has Hodge-Tate weights all $0$, then $H^2_{\varphi,\gamma}(\End_{\cR_L}(D_L))=0$ and $X_{D_L}$ is formally smooth.
\end{lemma}
\begin{proof}
    By the local Tate duality,
    \[H^2_{\varphi,\gamma}(\End_{\cR_L}(D_L))=H^0_{\varphi,\gamma}((D_L^{\vee}\otimes_{\cR_L} D_L)^{\vee}(\epsilon))^{\vee}=\Hom_{\varphi,\gamma}(\cR_L,D_L\otimes_{\cR_L} D_L^{\vee}(\epsilon))^{\vee}.\]
    Suppose that there is a non-zero map $g:\cR_L\rightarrow D_L\otimes_{\cR_L} D_L^{\vee}(\epsilon)$, which must be an injection and induces $B_{\dR}^+\hookrightarrow W_{\rm dR}^+(D_L\otimes_{\cR_L} D_L^{\vee}(\epsilon))$. However, all Hodge-Tate-Sen weights of the $(\varphi,\Gamma)$-module $D_L\otimes_{\cR_L} D_L^{\vee}(\epsilon)$ are equal to $1$,  
    \[H^0(\cG_{\Q_p}, W_{\rm dR}^+(D_L\otimes_{\cR_L} D_L^{\vee}(\epsilon)))=\Fil^0D_{\rm dR}(D_L\otimes_{\cR_L} D_L^{\vee}(\epsilon))=0.\]
    Hence $H^2_{\varphi,\gamma}(\End_{\cR_L}(D_L))=0$. The formally smooth statement is \cite[Prop. 3.6]{Chenevier2011fern}.
\end{proof}
We choose a trivialization $\alpha_L:D_{\pdR}(D_L)\simeq L^2$ and let $X_{x_{\pdR}}^{\square}$ be the completion of $\frg$ at the corresponding matrix $\nu_{L}\in\frg$. Let $X_{D_L}^{\square}:=X_{D_L}\times_{X_{x_{\pdR}}}X_{x_{\pdR}}^{\square}$, defined in \cite[\S 3.5]{breuil2019local}. Below, for a groupoid $X$ fibered over $\cC_L$, write $|X|$ for the corresponding functor taking isomorphism classes.
\begin{proposition}\label{propositionflatversalmaps}
    The morphism $X_{D_L}\rightarrow X_{x_{\pdR}}$ is relatively representable and is flat for all $D_L$ in the sense of maps between versal rings.
\end{proposition}
\begin{proof}
    Let $\cM_L=D_L[\frac{1}{t}]$ and let $X_{\cM_L}$ be the groupoid of deformations of $\cM_L$ in \cite[\S 3.3]{breuil2019local}. By \cite[Lem. 3.5.3]{breuil2019local}, the map $X_{\cM_L}\rightarrow X_{x_{\pdR}}$ is relatively representable. Let $W=(W_e,W_{\rm dR}^+):=(W_e(D_L),W_{\rm dR}^+(D_L))$ be the $B$-pair of $D_L$. We only need to show that the map $X_{D_L}\simeq X_{W_e}\times_{X_{W_{\rm dR}}}X_{W_{\rm dR}^+}\rightarrow X_{\cM_L}\simeq X_{W_e}$ is relatively representable (we used \cite[Prop. 3.5.1]{breuil2019local}). We reduce to show that $X_{W_{\rm dR}^+}\rightarrow X_{W_{\rm dR}}$ is relatively representable. This is true and in our case $X_{W_{\rm dR}^+}\simeq X_{W_{\rm dR}}$ when $W_{\rm dR}^+$ has Hodge-Tate weight $0$ (see Lemma \ref{lemmauniqueweight0} or \cite[Prop. 3.1]{wu2021local}). Note that we get $X_{D_L}\simeq X_{W_e}$.

    We may suppose that the map $X_{D_L}^{\square}=\Spf(R_{D_L}^{\square})\rightarrow X_{x_{\pdR}}^{\square}=\Spf(S^{\square})$ is induced by a continuous local morphism $S^{\square}\rightarrow R_{D_L}^{\square}$ of complete Noetherian local rings. We prove that this map is flat. By Proposition \ref{propositionformallysmooth}, this map is formally smooth, hence flat, if $D_L$ is not an extension of $t^{-1}\cR_L(\epsilon\delta)$ by $\cR_L(\delta)$ for a character $\delta$ of $\Q_p^{\times}$. Otherwise, by miracle flatness \cite[\href{https://stacks.math.columbia.edu/tag/00R3}{Tag 00R3}]{stacks-project} and Lemma \ref{lemmaH2weight0}, it's enough to show that the fiber $R^{\square}_{D_L}/\frm_{S^{\square}}$ has codimension $\dim S^{\square}=4$ in $R^{\square}_{D_L}$.
    
    We first calculate the dimension of $R_{D_L}^{\square}$. Let $A=L[\epsilon]=L[\epsilon]/\epsilon^2$. The fibers of $|X_{D_L}^{\square}|(A)\rightarrow |X_{D_L}(A)|$ over given $(D_{A},\iota_A)\in |X_{D_L}(A)|$ are isomorphisms $\alpha_A:D_{\pdR}(D_{A})\simeq A^2$ parametrized by $\End_{L}(L^2)$. Two deformations given by $\alpha_A,\alpha'_A$ are equivalent (give the same object in $|X_{D_L}^{\square}|(A)$) if and only if there exists an isomorphism $W_{e,A}\simeq W_{e,A}$ (in bijection with $H^0(\cG_{\Q_p}, \End_{B_e\otimes_{\Q_p}L}(W_{e}))$) inducing $(\alpha')^{-1}\alpha$. The dimension of $|X_{D_L}(A)|$ is the dimension of $H^1(\cG_{\Q_p},\End_{B_e\otimes_{\Q_p}L}(W_{e}))$. The composite 
    \begin{align*}
        H^0(\cG_{\Q_p},\End_{B_e\otimes_{\Q_p}L}(W_e))\rightarrow H^0(\cG_{\Q_p},\End_{B_{\dR}\otimes_{\Q_p}L}(W_{\dR}))\\\simeq \End_{\mathrm{Rep}_L(\bG_a)}(D_{\pdR}(W_{\dR}))\hookrightarrow \End_{L}(D_{\pdR}(W_{\dR}))
    \end{align*} 
    is injective. Hence
    \[\dim_L X_{D_L}^{\square}(A)=\dim_L H^1(\cG_{\Q_p}, \End_{B_e\otimes_{\Q_p}L}(W_{e}))+4-\dim_L H^0(\cG_{\Q_p}, \End_{B_e\otimes_{\Q_p}L}(W_{e}))=8\]
    by Euler characteristic formula and vanishing of $H^2$ (Lemma \ref{lemmaH2weight0}).
    
    If $D_L=\cR_L\oplus \cR_L(t^{-1}\epsilon)$, $\nu_L=0$, the quotient $R^{\square}/\frm_{S^{\square}}$ pro-represents the functor sending $A\in\cC_L$ to the groupoid of $(D_{A},\iota_A,\alpha_A)$ where $\alpha_A:D_{\pdR}(D_A)\simeq A^n$ and $D_A$ is de Rham. Since $ D_L$ is semi-stable, all its de Rham deformations are semi-stable by \cite[Thm. 0.9]{berger2002representations} and we have $D_{\pdR}(D_A)=D_{\mathrm{st}}(D_A)$ for such deformations. By the equivalence in \cite{berger2008equations}, the category of semi-stable $(\varphi,\Gamma)$-modules of Hodge-Tate weights $(0,0)$ is equivalent to the category of $(\varphi,N)$-modules of rank $2$. Hence the de Rham locus is the deformation space of matrices $(\varphi,N)\in \GL_2\times\frg$ such that $N\varphi=p\varphi N$. This space has dimension $4$ by \cite[Prop. 2.1]{hellmann2023derived}.

    If $D_L$ is a non-split extension of $t^{-1}\cR_L(\epsilon)$ by $\cR_L$, then $\End_{\varphi,\gamma}(D_L)=L$ and the sheaf $|X_{D_L}|$ is pro-represented by a deformation ring $R_{D_L}$ of dimension $\dim_LH^1_{\varphi,\gamma}(\End_{\cR_L}(D_L))=5$ \cite[Prop. 3.4]{Chenevier2011fern} (even though $X_{D_L}$ is not equivalent to $|X_{D_L}|$). The fiber product $X_{D_L,\nu_L}^{\square}:=\Spf(R^{\square}_{D_L})\times_{\Spf(S^{\square})}\Spf(S^{\square}/\frm_{S^{\square}})$ pro-represents the groupoid fibered over $\cC_L$ sending $A\in\cC_L$ to the groupoid of $(D_{A},\iota_A,\alpha_A)$ where $\alpha_A:D_{\pdR}(D_A)\simeq A^n$ such that $\nu_A=\nu_L$ under the trivialization $\alpha_A$. Let $X_{D_L,0}$ be the groupoid sending $A$ to $(D_A,\iota_A)\in X_{D_L}(A)$ such that coefficients of Sen polynomials (given by $\mathrm{tr}(\nu_A)$ and $\det(\nu_A)$) of $D_A$ vanish. Then $|X_{D_L,0}|$ is pro-represented by a quotient $R_{D_L,0}$ of $R_{D_L}$. Consider the map $|X_{D_L,\nu_L}^{\square}|\rightarrow |X_{D_L,0}|$. This map is formally smooth and of relative dimension $1$ by Lemma \ref{lemmaforchangenu} below. Thus $\dim_LR_{D_L,\nu_L}^{\square}=\dim_LR_{D_L,0}+1$. We show that $\dim_LR_{D_L,0}=3$, or has codimension $2$ in $R_{D_L}$.  Let $X_{D_L}^{z^{-1}\epsilon}$ be the deformation problem parametrizing deformations of $D_L$ with fixed determinant $z^{-1}\epsilon$ (or $\cR_{A}(z^{-1}\epsilon)$). Then $|X_{D_L}^{z^{-1}\epsilon}|$ is pro-represented by a complete Noetherian local ring $R_{D_L}^{z^{-1}\epsilon}$. Let $R_{\cR_L}$ be the dimension $2$ universal deformation ring of the trivial rank one $(\varphi,\Gamma)$-module. We have $R_{D_L}=R_{D_L}^{z^{-1}\epsilon}\widehat{\otimes}_LR_{\cR_L}$ (see Lemma \ref{lemmasquarerootcharacter} below). Since fixed weight deformation of the trivial character has dimension one, it's enough to show that the fixed determinant deformation ring is flat over $L[[h]]$ where $h$ is sent to the element in $R_{D_L}^{z^{-1}\epsilon}$ given by the trace of the universal nilpotent operator $\nu$. By Krull's principal ideal theorem, a minimal prime of $R^{z^{-1}\epsilon}_{D_L}$ containing $h$ has either height one or height zero. If all minimal primes containing $h$ has height one, then $\dim R^{z^{-1}\epsilon}_{D_L}/h=\dim R^{z^{-1}\epsilon}_{D_L}-1=2$ as desired. If the minimal prime containing $h$ has height zero, since $R^{z^{-1}\epsilon}_{D_L}$ is integral (even regular, as is for $R_{D_L}$), we see $h=0$ in $R^{z^{-1}\epsilon}_{D_L}$. This is not possible since we can construct trianguline deformations of the form $D_A=[\cR_A(\delta_A^{-1})-t^{-1}\cR_A(\delta_A\epsilon)]$ for some $\delta_A:\Q_p^{\times}\rightarrow A^{\times}$ and $A=L[\epsilon]$ such that the weight of $\delta_A$ is not zero and $D_A$ deforms $D_L$ (the cokernel of $H^1_{\varphi,\gamma}(\cR_{A}(z\delta_A^{-2}\epsilon^{-1}))\rightarrow H^1_{\varphi,\gamma}(\cR_L(z\epsilon^{-1}))$ maps injectively into $H^2_{\varphi,\gamma}(\cR_L(z\epsilon^{-1}))=0$).
\end{proof}

\begin{lemma}\label{lemmaforchangenu}
    The map $|X_{D_L,\nu_L}^{\square}|\rightarrow |X_{D_L,0}|$ in the above proof is formally smooth of relative dimension $1$.
\end{lemma}
\begin{proof}
    Let $A'\rightarrow A=A'/I$ be a surjection in $\cC_L$ such that $I^2=0$. Let $\nu_{A'}\in \End_{A'}({A'}^2)$ such that $\nu_{A'}\equiv \nu_L=\mtwo{0}{1}{}{0}\mod I$ and $\det(\nu_{A'})=\mathrm{tr}(\nu_{A'})=0$. We claim that there exists $M\in I_2+IM_2(A')$ (where $I_2$ is the identity matrix) such that $M\nu_{A'}M^{-1}=\nu_L$. To show the claim, write $\nu_{A'}=\mtwo{a}{1+b}{c}{d}$ for $a,b,c,d\in IA'$ and suppose that $M=I_2+\mtwo{x}{y}{z}{w}$ for $x,y,z,w\in IA'$. Since $I^2=0$, one can calculate that $M\nu_{A'}=\nu_LM$ if and only if $\mtwo{z}{w-x}{}{-z}=\mtwo{a}{b}{c}{d}$. As $\det(\nu_{A'})=\mathrm{tr}(\nu_{A'})=0$ implies that $d=-a$ and $c=0$, the solutions exist.

    The above discussion shows that the map is formally smooth. To see the relative dimension, we only need to calculate the difference between tangent spaces. Let $A=L[\epsilon]$. The fiber in $|X_{D_L,\nu_L}^{\square}|(A)$ over given $(D_{A},\iota_{A})\in |X_{D_L,0}|(A)$ (by the above discussion we may suppose that $\nu_A=\mtwo{0}{1}{}{0}$) consists of matrices $M$ of the form $I_2+\epsilon\mtwo{x}{y}{0}{x}$ for $x,y\in L$ which span a space of dimension two. Two such matrices give the same object in $|X_{D_L,\nu_L}^{\square}|(A)$ if and only if there exists an automorphism $D_{A}\rightarrow D_{A}$ which reduces to identity modulo $\epsilon$ (determined by an element in $L=H^0_{\varphi,\gamma}(\End_{\cR_L}(D_L))$) and induces the automorphism of $D_{\pdR}(D_A)$. Hence the fibers between tangent spaces have dimension one. 
\end{proof}
\begin{lemma}\label{lemmasquarerootcharacter}
    Let $A\in \cC_L$ and let $\delta_{A}:\Q_p^{\times} \rightarrow A^{\times}$ be a continuous character such that $\delta_A\equiv 1\mod \frm_A$. Then $\delta_{A}^{\frac{1}{2}}:\Q_p^{\times} \rightarrow A^{\times},x\mapsto \delta_A(x)^{\frac{1}{2}}:=\sum_{i\geq 0}\binom{\frac{1}{2}}{i}(\delta_A(x)-1)^i$ defines the unique character such that $\delta_A^{\frac{1}{2}}\equiv 1\mod \frm_A$ and $\delta_A=(\delta_A^{\frac{1}{2}})^2$.
\end{lemma}

We will use the flatness in Proposition \ref{propositionflatversalmaps} in the following situation. Let $\frY=\Sp(A)$ be an affinoid over $L$ together with a $(\varphi,\Gamma)$-module $D_A$. Let $I\subset A$ be the ideal generated by coefficients of the Sen polynomial cutting out the locus where $D_A$ has Sen weights $(0,0)$ and $\frY^{\wedge}$ the completion along this locus, $A^{\wedge}$ the completion with respect to $I$ (cf. Appendix \ref{sectionappendixformalfunction}). Let $h:\Sp(A)\rightarrow \frX_2$ and $h^{\wedge}:\frY^{\wedge}\rightarrow (\frX_{2})_0^{\wedge}$ be the morphism of stacks induced by $D_A$ and $(D_{A/I^n})_n$. Let $ D_{\pdR}:\frY^{\wedge}\rightarrow (\frX_2)_0^{\wedge}\stackrel{D_{\pdR}}{\rightarrow}\frg/G$ be the composite and let $D_{\pdR}^{\square}:\frY^{\wedge,\square}:=\frY^{\wedge}\times_{\frg/G}\frg\rightarrow \frg$ be the base change. 
\begin{corollary}\label{corollaryflatlocalrings}
    In the above situation, suppose that $h$ is smooth in the sense of versal maps: for any $L'$-point $y:\Sp(L')\rightarrow \Sp(A)$ and $x:\Sp(L')\stackrel{y}{\rightarrow}\Sp(A)\stackrel{h}{\rightarrow}\frX_2$, the induced map $\cF_{\Sp(A),y}\rightarrow \cF_{\frX_2,x}=X_{D_{L'}}$ is formally smooth. Then the map  $D_{\pdR}^{\square}:\frY^{\wedge,\square} \rightarrow \frg$ is flat in the sense that maps between versal rings at points with residue fields finite over $L$ are flat. 
\end{corollary}
\begin{proof}
    An $L'$-point $y^{\square}:\Sp(L')\rightarrow \frY^{\wedge,\square}$ corresponds to a map $\Sp(L')\rightarrow A$ such that the pullback $D_{L'}$ has weights $(0,0)$ together with a framing $\alpha_{L'}:D_{\pdR}(D_{L'})\simeq (L')^2$. Let $\nu_{L'}\in \frg(L')$ be the nilpotent element. By the difinition of $\frY^{\wedge,\square}$, the deformation problem $\cF_{\frY^{\wedge,\square},y^{\square}}$ sends $A'\in\cC_{L'}$ to pairs $(\alpha_{A'}:D_{\pdR}(D_A\otimes_AA')\simeq (A')^2,A\rightarrow A')$  where $A\rightarrow A'$ is a morphism such that $A\rightarrow L'$ factors through $A'$ and $\alpha_{A'}\equiv \alpha_{L'}\mod \frm_{A'}$. Then $(D_{A'}=D_A\otimes_AA',\iota_{A'}:D_{A'}\otimes_{A'}L'=D_{L'},\alpha_{A'})\in X_{D_{L'}}^{\square}(A')$ deforming $x^{\square}=(D_{L'},\iota_{L'},\alpha_{L'})\in X_{D_{L'}}^{\square}(L')$. By our assumption $\cF_{\frY^{\wedge,\square},y^{\square}}=\cF_{\frY^{\wedge},y}\times_{\cF_{\frg/G,\nu_{L'}}}\cF_{\frg,\nu_{L'}}=\cF_{\frY^{\wedge},y}\times_{X_{D_{L'}}}X_{D_{L'}^{\square}}$ is formally smooth over $X_{D_{L'}^{\square}}$. By Proposition \ref{propositionflatversalmaps}, the versal ring map which induces $X_{D_{L'}^{\square}}\rightarrow \cF_{\frg,\nu_{L'}}$ is flat. Hence the composite map $\cF_{\frY^{\wedge,\square},y^{\square}}\rightarrow \cF_{\frg,\nu_{L'}}$ is also flat.
\end{proof}

\begin{remark}\label{remarktrivialtorsor}
    Locally on $\frY_1$, one can choose a trivialization $D_{\pdR}(D_{A/I})$ on $\frY_1$ where $\frY_n=\Sp(A/I^n)$. Then $\frY_1^{\square}=\frY_1\times \GL_2$ and $D_{\pdR}(D_{\frY_1})$ is free with a basis on $\frY_1$. We can lift this basis by Nakayama lemma to a basis of $D_{\pdR}(D_{A^{\wedge}}):=\varprojlim_nD_{\pdR}(D_{A/I^n})$ which is a finite free $A^{\wedge}$-module by Lemma \ref{lemmaGAGAaffinoid}. We get a trivialization $\frY^{\wedge,\square}=\frY^{\wedge}\times \GL_2$.

    Cover $\frY^{\wedge,\square}$ by affinoids of the form $\Sp(B/I)\times U$ where $\Sp(B)\subset \Sp(A)$ and $U\subset \GL_2$ are affinoid opens. Let $C=B\widehat{\otimes}_L\cO(U)$ and $C^{\wedge}$ be its $I$-adic completion. The morphism $\Sp(C)^{\wedge}=\Sp(B)^{\wedge}\times U\rightarrow \frg$ of ringed sites factors through an affinoid $V\subset \frg$. In fact, write $\frg=\cup_{s\in\bbN}\frg_{\leq p^s}:=\cup_{s\in\bbN}\Sp(L\langle p^sa,p^sb,p^sc,p^sd\rangle)$. The map $\Sp(B)^{\wedge}\times U\rightarrow \frg$ is determined by sending $a,b,c,d$ to the matrix coefficients of $\nu_{C^{\wedge}}\in \End_{C^{\wedge}}(D_{\pdR}(D_{C^{\wedge}}))$ under the trivialization $ D_{\pdR}(D_{C^{\wedge}})\simeq (C^{\wedge})^2$ on $\Sp(C)^{\wedge}$. Take $s$ such that $p^sa,\cdots, p^sd$ are topologically nilpotent in $C/I$. Then for any $n\geq 1$, $p^sa,\cdots, p^sd$ are also topologically nilpotent in $C/I^n$ and induces $\varinjlim_n \Sp(C/I^n)\rightarrow \frg_{\leq p^s}$. The ring maps $\cO(\frg_{\leq p^s})\rightarrow C^{\wedge}$ and $L[a,b,c,d]\rightarrow C^{\wedge}$ are flat by Corollary \ref{corollaryflatlocalrings} and Lemma \ref{lemmaflatlocalrings} below.
\end{remark}    
\begin{lemma}\label{lemmaflatlocalrings}
    Let $A$ be an Noetherian ring with an ideal $I$ and $I$-adic completion $A^{\wedge}$. Let $B$ be another Noetherian ring and $g:B\rightarrow A^{\wedge}$ be a morphism of rings. Suppose that for any maximal ideal $\frm$ of $A^{\wedge}$ and $\frn=g^{-1}(\frm)$, the homomorphism $\widehat{B}_{\frn}\rightarrow \widehat{A}^{\wedge}_{\frm}$ of complete local rings is flat. Then the ring map $g$ is flat itself.
\end{lemma}
\begin{proof}
    Since $I$ is in the Jacobson radical of $A^{\wedge}$, maximal ideals of $A^{\wedge}$ is in bijection with maximal ideals of $A/I=A^{\wedge}/I$. By \cite[\href{https://stacks.math.columbia.edu/tag/00HT}{Tag 00HT}]{stacks-project}, its enough to show that for all $\frm,\frn$ as above, the map $B_{\frn}\rightarrow (A^{\wedge})_{\frm}$ is flat. Since both $A^{\wedge}$ and $B$ are Noetherian, the maps $(A^{\wedge})_{\frm}\rightarrow \widehat{A}^{\wedge}_{\frm}$ and $B_{\frn}\rightarrow \widehat{B}_{\frn}$ are faithfully flat, the flatness of the map between Zariski local rings follows (by definition). 
\end{proof}
\section{Translations in family for $\GL_2(\Q_p)$}\label{sectiontranslationinfamily}
We will recall the definition of translations for $(\varphi,\Gamma)$-modules in \cite{ding2023change} in the $\GL_2(\Q_p)$ case. We will need some explicit calculations for translations in families. It will be shown that the translation from regular weights to non-regular weights is the same as the change of weights of $(\varphi,\Gamma)$-modules. From now on, $K=\Q_p$.

The Lie algebra $\frg=\mathfrak{gl}_2$ is spanned by 
\[a^+=\mtwo{1}{}{}{0},a^-=\mtwo{0}{}{}{1}, u^+=\mtwo{0}{1}{}{0},u^-=\mtwo{0}{}{1}{0}.\] 
We write
\[\frz=a^++a^-=\mtwo{1}{}{}{1},\frh=a^+-a^-=\mtwo{1}{}{}{-1}\]
The Casimir operator $\frc=\frh^2-2\frh+4u^+u^-\in U(\frg)$. 

Let $A$ be an affinoid algebra over $L$. A $(\varphi, \Gamma)$-module $D_A$ over $\cR_A$ is a $P^+=\mtwo{\Z_p\setminus 0}{\Z_p}{}{1}$-module and is a $\frp^+$-module where $\frp^+=L[a^+,u^+]$ via the identifications 
\[\varphi=\mtwo{p}{}{}{1},\Gamma=\mtwo{\Z_p^{\times}}{}{}{1}, L[[X]]=L[[\mtwo{1}{\Z_p}{}{1}]].\] 
The actions of $u^+$ on $D_A$ is given by $v\mapsto \frac{d}{dz}(1+X)^zv|_{z=0}=\log(1+X)v=tv$ and $a^+$ by $\nabla=\nabla_{\rm Sen}$ by the lemma below.
\begin{lemma}
    Suppose that $D_A^r$ is a $(\varphi,\Gamma)$-module over $\cR_A^r$, then the map $v\mapsto \frac{d}{d\gamma}|_{\gamma=1}\gamma.v, \gamma\in\Z_p^{\times}\simeq \Gamma$ defines the Sen operator $\nabla_{\Sen}$ acting on $D_A^r$.
\end{lemma}
\begin{proof}
    The proof is the same as the case when $A$ is a field using \cite[Lem. 5.2]{berger2002representations}. It is proved in \cite[Prop. 2.2.14]{kedlaya2014cohomology} that the action of $\Gamma$ on $D_A^r$ extends to an action of the distribution algebra of $\Gamma$ which contains the element $a^+$.
\end{proof}
\begin{definition}\label{definitionphgammagmodule}
    A $(\varphi, \Gamma,\frg)$-module over $\cR_A^r$ is a $(\varphi,\Gamma)$-module $D_A^r$ (in the sense of \cite[Def. 2.2.12]{kedlaya2014cohomology}) over $\cR_A^r$ with an $A$-linear continuous action of $U(\frg)$ extending the action of $U(\frp^+)$ such that the $U(\frg)$-action extends continuously to $D_A^{s}$ for all $0<s<r$ and $\varphi u^-=p^{-1}u^-\varphi, \varphi a^-=a^{-}\varphi$ under $\varphi:D_A^r\rightarrow D_A^{r/p}$. A $(\varphi, \Gamma,\frg)$-module over $\cR_A$ is a $(\varphi,\Gamma)$-module over $\cR_A$ with an action of $U(\frg)$ that is the base change from $\cR_A^r$ to $\cR_A$ for a $(\varphi, \Gamma,\frg)$-module over $\cR_A^r$ and some $r>0$.
\end{definition}
A $(\varphi,\Gamma,\frg)$-module is naturally a $(P^+,\frg)$-module defined in a similar way where $P^+$ acts continuously. If $D_A$ is a $(\varphi, \Gamma,\frg)$-module, then the action of $Z(\frg)$ commutes with $\varphi,\Gamma, \cR_A^r$. 

\begin{lemma}\label{lemmastandardgstructure}
    Let $D_A^r$ be a $(\varphi,\Gamma)$-module over $\cR_A^r$ of rank two. Write $P_{\rm Sen}(T)=T^2-\gamma_1T+\gamma_0\in A[T]$ for the Sen polynomial of $D_A$. There exists a unique $A$-linear $\frg$-module structure on $D_A^r$ such that $\frc$ acts on $D_A^r$ by $(\gamma_1^2-4\gamma_0-1)$ and $\frz$ acts by $\gamma_1-1$ making $D_A^r$ (resp. $D_A$) a $(\varphi,\Gamma,\frg)$-module. 
\end{lemma}
\begin{proof}
    This is just \cite[Prop. 2.2]{colmez2018poids}. We declare the (unique) action of $a^-$ by $\frz-a^+$, $\frh=a^+-a^{-}=2\nabla-\frz=2\nabla-\gamma_1+1$ and $u^{-}$ by $\frac{\frc-\frh^2+2\frh}{4u^+}=-\frac{P_{\rm Sen}(\nabla)}{t}$. The last one is possible because $ P_{\rm Sen}(\nabla)D_A\subset tD_A$ by the same reason for \cite[Lem. 1.6]{colmez2018poids} using Proposition \ref{propositionappendixglueing}. We can check this formally defines an action of $\frg$. For example for $v\in D_A$, we have $[u^+,u^-]v=(u^+u^--u^{-}u^+)v=-P_{\rm Sen}(\nabla)+t^{-1}P_{\rm Sen}(\nabla)tv=-(\nabla^2-\gamma_1\nabla+\gamma_0)v+t^{-1}(\nabla^2-\gamma_1\nabla+\gamma_0)tv=(2\nabla-\gamma_1+1)v=\frh.v$ using that $\nabla(tv)=tv+t\nabla(v)$. We also check that $\varphi u^-=p^{-1}u^-\varphi$ using that $\varphi$ commutes with $\frz,\frc$ and $\varphi u^+=p u^+\varphi$. 
\end{proof}

\begin{definition}\label{definitionstandardgmodulestructure}
    We say the $\frg$-module structure in Lemma \ref{lemmastandardgstructure} the standard $\frg$-module structure for a rank $2$ $(\varphi,\Gamma)$-module over $\cR_A$ or $\cR_A^r$.
\end{definition}
\begin{remark}
    It is possible to equip a $(\varphi,\Gamma)$-module with different $\frg$-structures. See \cite[Rem. 2.14]{ding2023change} for more discussions.
\end{remark}
Let $V_k=\Sym^k L^2=\cR_L^+/X^{k+1}$ be the irreducible representation of $\frg$ of highest weight $(k,0)$ (for the Borel subalgebra the algebra of upper-triangular matrices). Here $\cR_L^+=D(\Z_p,L)\subset \cR_L$ is the distribution algebra of $\Z_p$ identified with rigid analytic functions in variable $X$ on the open unit disc. If $D_A^r$ is a $(\varphi,\Gamma,\frg)$-module, then $D_A^r\otimes_L V_k$ is also a $(P^+,\frg)$-module via the diagonal action of $P^+$. The following observation is due to Ding.
\begin{proposition}[{\cite[Prop. 2.1]{ding2023change}}]\label{proptensoralgebraic}
    Suppose that $D_A$ is a $(\varphi,\Gamma,\frg)$-module over $\cR_A$. The diagonal action of $\cR_A^+=D(\Z_p,L)\widehat{\otimes}_LA$ extends to an action of $\cR_A$ making $ D_A\otimes_L V_k$ a $(\varphi,\Gamma,\frg)$-module over $\cR_A$. And there is a filtration 
    \[0\subset D_A\otimes_L X^{k}\cR_L^+/X^{k+1}\subset \cdots\subset D_A\otimes_LX^i\cR_L^+/X^{k+1}\subset\cdots \subset D_A\otimes_L \cR_L^+/X^{k+1}\]
    of $(\varphi,\Gamma)$-modules (as well as $\frp^+$-modules, but not as $\frg$-modules) with graded pieces 
    \[D_A\otimes_L (X^i\cR_L^+/X^{i+1})\simeq t^i D_A\] 
    for $0\leq i\leq k.$
\end{proposition}
\begin{proof}
    By definition, there exists $r$ such that $D_A$ is the base change of $D_A^r$ from a finite projective module over $\cR_A^r$. We will prove (and will use) the statement for $D_A^r$.
    
    The same proof of \textit{loc. cit.} shows that the action of $\cR_A^{+}$ extends to an action of $\cR_A^r$ on $D_A^r\otimes_LV_k$. We give a direct proof here. Notice that for $X=[1]-1\in L[[\Z_p]]\subset D(\Z_p,L)$, $g\in D_A^r$ and $v\in V_k$, we have $X(X^{-1}g\otimes v)=g\otimes v+X^{-1}g\otimes Xv+g\otimes Xv$. We then get $ \frac{1}{X}(g\otimes v)=X^{-1}g\otimes v-\frac{1}{X}(\frac{X+1}{X}g\otimes Xv)=X^{-1}g\otimes v-\frac{X+1}{X^2}g\otimes Xv+\frac{1}{X}((\frac{X+1}{X})^2g\otimes X^2v)=\sum_{i=0}^k(-\frac{X+1}{X})^iX^{-1}g\otimes X^iv$. Hence $X^{-m}(g\otimes v)=\sum_{i=0}^k\binom{m+i-1}{m-1}(-\frac{X+1}{X})^{i}X^{-m}g\otimes X^iv$. And we conclude that $f(X)(g\otimes v)=\sum_{i=0}^k\frac{1}{i!}(X+1)^{i}f^{(i)}(X)g\otimes X^iv$ where $f^{(i)}$ denotes the $i$-th derivative of $f$. The extension of the action follows since if $f\in \cR_A^r=\cO(\bbU^r\times \Sp(A))$, so is its derivative for $X$. Then $D_A^r\otimes_LV_k$ becomes a projective $\cR_A^r$-module with a $\cR_A^r$-filtration. The projectivity is due to that extensions of projective modules are projective. 
    
    We extend the $\varphi,\Gamma$ actions diagonally. For example $\varphi: D_A^{r}\rightarrow \cR^{r/p}_A\otimes_{\cR^r}D_A^r$ gives a map $D_A^r\otimes_L \cR_L^+/X^{k+1}\rightarrow (\cR^{r/p}_A\otimes_{\cR^r_A}D_A^r)\otimes_L \cR_L^+/X^{k+1}$. There is an $\cR_{A}^{r/p}$-isomorphism $\Psi:\cR^{r/p}_A\otimes_{\cR_A^r}(D_{A}^r\otimes_L \cR_L^+/X^{k+1})\simeq (\cR^{r/p}_A\otimes_{\cR^r_A}D_A^r)\otimes_L \cR_L^+/X^{k+1}: f(X)\otimes (g\otimes v)\mapsto \sum_{i=0}^k(\frac{1}{i!}(X+1)^if^{(i)}(X) \otimes g)\otimes X^iv$ for $f(X)\in \cR_A^{r/p},g\in D_A^r, v\in \cR_L^+/X^{k+1}$ which is well-defined and extends the $\cR_L^{+}$-actions (given by the diagonal action of $\Z_p$). We get the desired $\varphi$-action map: $D_A^r\otimes_L \cR_L^+/X^{k+1}\rightarrow \cR^{r/p}_A\otimes_{\cR_A^r}(D_{A}^r\otimes_L \cR_L^+/X^{k+1}):g\otimes m\rightarrow \Psi^{-1}(\varphi(g)\otimes \varphi(m))$. The equality $\Psi^{-1}(\varphi (f(X+1).(g\otimes m)))=f((X+1)^p).(\Psi^{-1}(\varphi(g)\otimes \varphi(m)))$ is equivalent to that $\varphi (f(X+1).(g\otimes m))=f((X+1)^p).(\varphi(g)\otimes \varphi(m))$ which can be checked formally.

    See Lemma \ref{lemmatwist} below for the last isomorphism.
\end{proof}

\begin{lemma}\label{lemmabasechangetensor}
    If $A\rightarrow B$ is a morphism of affinoid algebras and $D_A$ is an $(\varphi,\Gamma,\frg)$-module, then $D_A\otimes_{\cR_A}\cR_B$ is also a $(\varphi,\Gamma,\frg)$-module and there is an isomorphism $(D_A\otimes_L\cR_L^+/X^{k+1})\otimes_{\cR_A}\cR_B\simeq ((D_A\otimes_{\cR_A}\cR_B)\otimes_L\cR_L^+/X^{k+1})$ as $(\varphi,\Gamma,\frg)$-modules.
\end{lemma}
\begin{proof}
    We let $Z(\frg)\subset U(\frg)$ acts on the $(\varphi,\Gamma)$-module $D_B:=D_A\otimes_{\cR_A}\cR_B$ by extending the scalars. To show that $D_B$ is a $(\varphi,\Gamma,\frg)$-module, we need to show that the action of $u^{-}=-\frac{1}{t}(\frc^2-\frh^2+2\frh)$ is defined, or that $(\frc^2-(2\nabla-\frz)^2+2(2\nabla-\frz))D_B\subset t D_B$. Using Proposition \ref{propositionappendixglueing}, that $D_{\Sen}^{m}(D_B)=D_{\Sen}^{m}(D_A)\otimes_{A}B$, the extension of the $u^-$ action follows. The map $(D_A\otimes_L\cR_L^+/X^{k+1})\otimes_{\cR_A}\cR_B\rightarrow ((D_A\otimes_{\cR_A}\cR_B)\otimes_L\cR_L^+/X^{k+1})$ is given by $(g\otimes v)\otimes f\mapsto \sum_{i=0}^k\frac{1}{i!}(X+1)^{i}f^{(i)}(X)g\otimes X^iv$ for $g\in D_A,v\in \cR^+_L/X^{k+1}$. The map induces isomorphisms on graded pieces of the filtration in Proposition \ref{proptensoralgebraic}, hence is an isomorphism of $(\varphi,\Gamma)$-modules over $\cR_B$ itself. It is a $\frg$-isomorphism since it is moreover compatible with the $\cR_B$-linear actions of $Z(\frg)$.  
\end{proof}
We need recall translations of $\frg$-modules.

\begin{lemma}\label{lemmageneralizedeigenspace}
    Let $R$ be a ring.
    \begin{enumerate}
        \item Suppose that $R$ is commutative with a nilpotent ideal $I$ and $V_R$ is an $R$-module equipped with an $R$-linear endomorphism $C$. Let $a,b\in R$ such that $a-b\in I$. Then $V_R\{C=a\}=V_R\{C=b\}$ where $\{C=-\}$ denotes the generalized eigenspace.
        \item Let $R$ be an $L$-algebra. Suppose that $V_R$ is an $R$-module and $Z$ is a commutative algebra finitely generated over $L$ with an $L$-morphism $Z\rightarrow \End_R(V_R)$. Suppose that $V_R$ is locally $Z$-finite (for any $v\in V_R$, $Z.v\subset V_R$ is finite-dimensional over $L$). Then the $R$-morphism $\oplus_{\frm\in\mathrm{SpecMax}(Z)}V_R[\frm^{\infty}]\rightarrow V_R$ is an isomorphism. 
        \item Let $R\rightarrow S$ be a morphism of $L$-algebras. Let $V_S=V_R\otimes_RS$ and let $Z$ act on $V_S$ by extending the scalar. Suppose that $V_R$ is locally $Z$-finite. Then for any $\frm\in \mathrm{SpecMax}(Z)$, the natural map $V_R[\frm^{\infty}]\otimes_RS\rightarrow V_S[\frm^{\infty}]$ is an isomorphism.
    \end{enumerate}
\end{lemma}
\begin{proof}
    (1) Let $x=b-a\in I$. Let $T=C-a$. Then $(T-x)^s=\sum_{i=0}^s\binom{s}{i}(-1)^{n-i}T^{i}x^{n-i}$. If $v\in V$ is killed by some power of $T$, since it is killed by some power of $x$, it is killed by some power $(T-x)^s=(C-b)^s$. 

    (2) The decomposition (as $L$-spaces) holds for all finite-dimensional $Z$-modules.

    (3) This follows from the decompositions in (2) for $V_R$ and $V_S$ and that $V_R[\frm^{\infty}]\otimes_{R}S$ is mapped into $ (V_R\otimes_RS)[\frm^{\infty}]$.
\end{proof}
For two weights $\lambda,\mu\in \frt^*=\Hom_L(\frt,L)$ such that $\nu=\lambda-\mu\in\Z^2$ is integral, write $\overline{\nu}$ for the dominant weight in the Weyl group orbits of $\nu$ for the linear Weyl group action. We recall the following definitions (see \cite[\S 2.3, \S 2.4.1]{jena2021translation}). For $\lambda\in\frt^*$, let $\frm_{\lambda}\subset Z(\frg)$ be the kernel of the infinitesimal character $\chi_{\lambda}$ attached to $\lambda$ via the Harish-Chandra isomorphism.
\begin{definition}
    Let $M$ be a $U(\frg)$-module.
    \begin{enumerate}
        \item We say $M$ is locally $Z(\frg)$-finite if for any $v\in M$, the subspace $Z(\frg).v$ is finite-dimensional over $L$. 
        \item If $M$ is locally $Z(\frg)$-finite, write $\pr_{|\lambda|}M:=M[\frm_{\lambda}^{\infty}]$.
        \item If $M$ is locally $Z(\frg)$-finite, $T_{\mu}^{\lambda}M:=\pr_{|\lambda|}(\pr_{|\mu|}M\otimes_L L(\overline{\nu}))$.
    \end{enumerate}
\end{definition}

Remark that if $M$ is $Z(\frg)$-finite, then $M\otimes_L L(\overline{\nu})$ is also locally $Z(\frg)$-finite (see \cite[Cor. 2.6 (ii)]{bernstein1980tensor}). Moreover, there is a direct sum decomposition $M=\oplus_{\frm\in \mathrm{SpecMax}(Z(\frg))}M[\frm^{\infty}]$. If $M=M[\frm^{\infty}]$ for some $\frm=\frm_{\lambda}$ and $\lambda+\rho$ is dominant, then
\[ M\otimes V=\oplus_{\mu\in \wt(V)} \mathrm{pr}_{|\lambda+\mu|}(M\otimes V)\]
for any finite dimensional $\frg$-module $V$ where $\mu$ runs over all $\frt$-weights appeared in $V$ (see the end of \cite[\S 2]{bernstein1980tensor}). Moreover, if $0\rightarrow M_1\rightarrow M\rightarrow M_2\rightarrow 0$ is a short exact sequence of $U(\frg)$-modules such that $M_1,M_2$ are locally $Z(\frg)$-finite, then so is $M$.

\begin{lemma}\label{lemmatranslationliealgebra}
    Let $\lambda_1, \cdots ,\lambda_n\in \frt^*$ such that $\lambda_{i+1}-\lambda_{i}\in X^*(\frt)_+$ are dominant integral weights for all $n-1\geq i\geq 1$ and suppose that $\lambda_1+(1,0)$ is dominant. Then for a locally $Z(\frg)$-finite $\frg$-module $M$, there are natural isomorphisms 
    \[T_{\lambda_{n-1}}^{\lambda_n}\cdots T_{\lambda_1}^{\lambda_2}M\simeq T_{\lambda_1}^{\lambda_n}M\simeq \pr_{|\lambda_n|}(\pr_{|\lambda_1|}M\otimes L(\lambda_1-\lambda_2)\otimes\cdots\otimes L(\lambda_{n}-\lambda_{n-1})).\]
\end{lemma}
\begin{proof}
    The tensor product $L(\lambda_1-\lambda_2)\otimes\cdots \otimes L(\lambda_n-\lambda_{n-1})$ has a direct summand $L(\lambda_n-\lambda_1)$ of multiplicity one and all the other summands are irreducible representations with highest weight $\mu<\lambda_n-\lambda_1$ by induction using Steinberg's theorem \cite{steinberg1961general}.
\end{proof}

\begin{lemma}\label{lemmageneralizedeigenspacephiGammamodules}
    Let $D_A$ be a $(\varphi,\Gamma,\frg)$-module over $\cR_A$ that is locally $Z(\frg)$-finite.  Let $I$ be the nilradical of $A$.
    \begin{enumerate}
        \item Suppose that $\chi_{1,A},\chi_{2,A}:Z(\frg)\rightarrow A$ are two characters such that $\chi_1\equiv \chi_2 \mod I$. Then $D_A\{Z(\frg)=\chi_{1,A}\}=D_A\{Z(\frg)=\chi_{2,A}\}$.
        \item Let $\lambda\in\frt^*$ and $\chi_A=\chi_{\lambda}$ corresponding to $\frm_{\lambda}\in \mathrm{SpecMax}(Z(\frg))$. Then the generalized eigenspace $D_A\{Z(\frg)=\chi_A\}=D_A[\frm_{\lambda}^{\infty}]$ is a $(\varphi,\Gamma)$-module. 
        \item In the situation of (2). The functor $D_A\{Z(\frg)=\chi_{\lambda}\}$ is exact on locally $Z(\frg)$-finite $(\varphi,\Gamma,\frg)$-modules and its formation commutes with arbitrary base change, i.e., for any map $A\rightarrow B$ of $L$-affinoid algebras, we have $D_A[\frm_{\lambda}^{\infty}]\otimes_{\cR_A}\cR_B\simeq D_B[\frm_{\lambda}^{\infty}]$ where $D_B=D_A\otimes_{\cR_A}\cR_B$. 
    \end{enumerate}
\end{lemma}
\begin{proof}
     (1) An affinoid algebra $A$ over $L$ is Jacobson and Noetherian. The nilradical $I$ is nilpotent. The statement follows from that $Z(\frg)$ is finitely generated and Lemma \ref{lemmageneralizedeigenspace}.

     (2) Since the action of $Z(\frg)$ commutes with the $(\varphi,\Gamma,\frg)$-module structure on $D_A$, we see $D_A\{Z(\frg)=\chi_A\}$ is a $\cR_A$-module with compatible actions of $\varphi,\Gamma,\frg$. It suffices to show that $D_A\{Z(\frg)=\chi_A\}$ is a $(\varphi,\Gamma,\frg)$-module. By (2) of Lemma \ref{lemmageneralizedeigenspace}, $D_A^r[\frm_{\lambda}^{\infty}]$ is a direct summand of $D_A^r$. Then $D_A^r[\frm_{\lambda}^{\infty}]$ is finite projective over $\cR_A^r$ and we have $D_A\{Z(\frg)=\chi_A\}=\cR_A\otimes_{\cR^r_A}D_A^r\{Z(\frg)=\chi_A\}$. 

     (3) This also follows from Lemma \ref{lemmageneralizedeigenspace}.
\end{proof}
For Sen weights $\bh=(h_1,h_2)\in\Z^2,h_1\leq h_2$, we write 
$\lambda=\lambda_{\bh}=(h_2-1,h_1)$ for the corresponding character of $\frb$ where $\lambda+(1,0)$ is dominant. Then $\chi_{\lambda}$ is the infinitesimal character $Z(\frg)\rightarrow L:\frz\mapsto h_1+h_2-1,\frc\mapsto (h_1-h_2)^2-1$.
\begin{proposition}\label{proptranslationphigammamodule}
    Let $D_A$ be an almost de Rham $(\varphi,\Gamma)$-module of rank two over $\cR_A$ with Sen weights $(h_1,h_2),h_1\leq h_2$, equipped with the standard $U(\frg)$-module structure by Lemma \ref{lemmastandardgstructure}. Then $D_A=D_A\{Z(\frg)=\chi_{\lambda}\}$. And for any $\mu\in X^*(\frt)$, $T_{\lambda}^{\mu}D_A$ is an almost de Rham $(\varphi,\Gamma)$-module over $\cR_A$ provided that $T_{\lambda}^{\mu}D_A\neq 0$. Moreover, the formation of $T_{\lambda}^{\mu}D_A$ commutes with base change. 
\end{proposition}
\begin{proof}
    In Lemma \ref{lemmastandardgstructure}, $Z(\frg)$ act on $D_A$ by a character $\chi_A(\frz)=\gamma_1-1$ and $\chi_A(\frc)=(\gamma_1^2-4\gamma_0-1)$ if the Sen polynomial equals to $T^2-\gamma_1T+\gamma_0$. Since $D_A$ pointwisely has Sen weights $h_1,h_2$, the Sen polynomial equals to $(T-h_1)(T-h_2)$ at all closed points of $\Spec(A)$. An element $f\in A$ is in the nilradical $I$ of $A$ if $f\in\frm$ for every maximal ideal $\frm$ of $A$. We see $X^2-\gamma_1X+\gamma_0\equiv X^2-(h_2+h_1)X+h_1h_2 \mod I$. Thus $\chi_A\equiv \chi_{\lambda}\mod I$. Hence $D_A=D_A[Z(\frg)=\chi_A]=D_A\{Z(\frg)=\chi_{\lambda}\}$ by (1) of Lemma \ref{lemmageneralizedeigenspacephiGammamodules}. The other statements follow from the same lemma, Lemma \ref{lemmabasechangetensor} and Proposition \ref{proptensoralgebraic}.
\end{proof}
\begin{corollary}\label{corollarypointwisesDing}
    We equip rank two $(\varphi,\Gamma)$-modules $D_A$ with the standard $\frg$-structures.
    \begin{enumerate}
        \item Suppose that $D_A$ is almost de Rham with pointwisely regular Sen weights $(h_1,h_2)\in\Z^2, h_1<h_2$. Then for any integral weight $\mu=(h_2'-1,h_1')$ such that $h_1'<h_2'$, $T_{\lambda}^{\mu}D_A$ is almost de Rham of rank two pointwisely with Sen weights $(h_1',h_2')$.
        \item Suppose that $D_A$ is almost de Rham with pointwisely non-regular Sen weights $(h_1',h_2')\in\Z^2, \mu=(h_2'-1,h_1'), h_1'=h_2'$. Then for any integral weight $\lambda=(h_2-1,h_1)$ such that $h_1<h_2$, $T_{\mu}^{\lambda}D_A$ is almost de Rham of rank $4$.
    \end{enumerate}
\end{corollary}
\begin{proof}
    In any case $T_{\lambda}^{\mu}D_A$ is a $(\varphi,\Gamma)$-module by Proposition \ref{proptranslationphigammamodule} and its rank, being almost de Rham and Sen weights can be checked at points. The statements follow from the case when $A$ is a field which was studied in \cite[Prop. 2.19]{ding2023change}. 
\end{proof}
\begin{remark}
    In the case when $\lambda-\mu=(1,0)$ and $D_A$ has non-regular Sen weights, $T_{\mu}^{\lambda}D_A=D_A\otimes_L V_1$.
\end{remark}

We start calculations of translations. The easiest case is a twist by an algebraic character.
\begin{lemma}\label{lemmatwist}
    For $i\in\Z$, let $Lt^i$ be the algebraic character $\det^i$ of $\GL_2$ which is also a $\frg$-module. Then for a rank $2$ $(\varphi,\Gamma)$-module $D_A$ over $\cR_A$, $t^iD_A$ with the standard $\frg$-module structure is equal to $D_A\otimes_L Lt^i$.
\end{lemma}
\begin{proof}
    On $\det^i=Lt^i$ we have $u^{\pm}.t^i=0$ and $a^{\pm}.t^i=it^i$. Thus for $x\otimes t^i,x\in D_A$ we get $u^+.(x\otimes t^i)=tx\otimes t^i,a^{\pm}.(x\otimes t^i)=(a^{\pm}+i)x\otimes t^i$ and $u^{-}.(x\otimes t^i)=u^{-}.x\otimes t^i=-\frac{P_{\rm Sen}(\nabla)}{t}x\otimes t^i$. For $t^ix\in t^iD_A$ we have $u^+.(t^ix)=t^{i+1}x,a^{+}.(t^ix)=t^i(a^{+}+i)x, \frz. t^ix=(\gamma_1-1+2i)x$ in the notation of Lemma \ref{lemmastandardgstructure}. Hence $a^-. t^ix=t^i(\gamma_1-1+2i-a^{+}-i)x=t^i((a^{-}+i)x)$. And $u^{-}.t^ix=-\frac{P'_{\rm Sen}(\nabla)}{t}t^ix=-t^i\frac{P_{\rm Sen}(\nabla)}{t}x$ where $P'_{\rm Sen}$ denotes the Sen polynomial of $t^iD_{A}$. The last equality comes from $\nabla.t^ix=t^i(\nabla+i)x$ and $P'_{\rm Sen}(T)=P_{\rm Sen}(T-i)=T^2-(\gamma_1+2i)T+\gamma_0+i\gamma_1+i^2$. 
\end{proof}
By a twist, we only need to discuss the case when at least one of the weights of an almost de Rham $(\varphi,\Gamma)$-module is zero.
\begin{lemma}\label{lemmaregularweights}
    Let $D_A$ be an almost de Rham $(\varphi,\Gamma)$-module over $\cR_A$ with pointwisely regular Sen weights $h_1<h_2$. Then there exists $\alpha,\beta\in A$ such that $P_{\rm Sen}(T)=(T-\alpha)(T-\beta)$ and $\alpha-h_1,\beta-h_2$ are in the nilradical $I$ of $A$.
\end{lemma}
\begin{proof}
    We have $P_{\rm Sen}(T)\equiv (T-h_1)(T-h_2)\mod I$. Hence $P_{\rm Sen}(h_1)\in I$. Moreover $P^{(1)}_{\rm Sen}(h_1)\equiv h_1-h_2\notin I$. By Hensel's lemma,  we can find $\alpha\in A$ such that $\alpha\equiv h_1\mod I$ and $P_{\rm Sen}(\alpha)=0$. Then $P_{\rm Sen}$ must be of the form $(T-\alpha)(T-\beta)$ such that $\beta\equiv h_2\mod I$.
\end{proof}
\begin{proposition}\label{propositionregulartononregular}
    Let $D_A$ be an almost de Rham $(\varphi,\Gamma)$-module pointwisely with regular Sens weight $(0,k)$ and let $\lambda=(k-1,0)$ where $k\geq 1$. We assume that the Sen polynomial is equal to $(T-(k+z-h))(T-(z+h))$ for $z,h$ in the nilradical of $A$ (which is possible by Lemma \ref{lemmaregularweights}).  
    \begin{enumerate}
        \item Let $\mu= (k-1,k)$. The natural map $T_{\lambda}^{\mu}D_A\hookrightarrow D_A\otimes_{L}V_{k}\twoheadrightarrow D_A$ identifies $T_{\lambda}^{\mu}D_A$ with the unique almost de Rham $(\varphi,\Gamma)$-submodule inside $D_A$ of weights $(k,k)$ such that $T_{\lambda}^{\mu}D_A[\frac{1}{t}]=D_A[\frac{1}{t}]$, which contains $t^kD_{A}$ and is the preimage of $(D_A/t^k)[\prod_{i=0}^{k-1}(\nabla-(k+z-h)-i)=0]$. 
        \item Let $\lambda'=(k'-1,0)$ for $k'\geq k$. The natural map $T_{\lambda}^{\lambda'}D_A\hookrightarrow D_A\otimes_{L}V_{l}\twoheadrightarrow D_A$ for $l=k'-k$ identifies $T_{\lambda}^{\mu}D_A$ with the unique almost de Rham $(\varphi,\Gamma)$-submodule inside $D_A$ such that $T_{\lambda}^{\lambda'}D_A[\frac{1}{t}]=D_A[\frac{1}{t}]$ of weights $(0,k')$ and consists of $v\in D_A$ such that $\prod_{i=0}^{j}(\nabla-(z+h)-i)v\in t^{j+1}D_A$ for all $j\leq l-1$.  
    \end{enumerate}
    In both cases, the output of the translation is a rank two $(\varphi,\Gamma)$-module and admits an infinitesimal character.
\end{proposition}
\begin{proof}
    The last assertion follows from Corollary \ref{corollarypointwisesDing}, or by the case when $l=1$ using the calculation below and by an induction as for \cite[Prop. 2.19]{ding2023change}.

    Let $l\geq 1$. The element $\frz$ acts on $D_A\otimes_LV_l$ by $2z+k+l-1$. Let $e$ be the lowest weight vector of $\cR_L^+/X^{l+1}\simeq V_l:\alpha\mapsto \alpha e$. Then $\frh=2a^+-\frz$ acts by $2\nabla-2z-k+1$ on $D_A$ and by $2i-l$ on $t^ie$; $u^+$ by $t$ on $D_A$ and $t=\log(1+X)=\sum_{i\geq 0}\frac{(-1)^i}{i}X^i$ on $V_l$; $u^-$ by $-\frac{1}{t}P_{\rm Sen}(\nabla)$ on $D_A$ and $u^-.t^ie=i(l-i+1)t^{i-1}e$ so that $\frc=\frh^2-2\frh+4u^+u^-$ acts on $V_l$ by $l^2+2l$. Hence the Casimir acts by (for $v\in D_A, 0\leq i\leq l$ and with the convention that $t^{i-1}e=0$ if $i=0$)
    \begin{align}\label{equationactioncasimir}
        \frc. (v\otimes t^ie)=&((2\nabla+1-2z-k+2i-l)^2-2(2\nabla+1-2z-k+2i-l))v\otimes t^ie\\\nonumber
        &+4i(l-i+1)(v\otimes t^ie+tv\otimes t^{i-1}e)
        -4P_{\rm Sen}(\nabla)v\otimes t^ie-4t^{-1}P_{\rm Sen}(\nabla)v\otimes t^{i+1}e\\\nonumber
        =&(4(2i-l)\nabla+(2z+k-2i+l)^2-1-4(z+h)(k+z-h)+4i(l-i+1))v\otimes t^ie\\\nonumber
        &+4i(l-i+1)tv\otimes t^{i-1}e- 4t^{-1}P_{\rm Sen}(\nabla)v\otimes t^{i+1}e\\\nonumber
        =&(4(2i-l)\nabla+(\gamma_1-2i+l)^2-1-4\gamma_0+4i(l-i+1))v\otimes t^ie\\\nonumber
        &+4i(l-i+1)tv\otimes t^{i-1}e- 4t^{-1}P_{\rm Sen}(\nabla)v\otimes t^{i+1}e\nonumber
    \end{align}
    where $\gamma_1=2z+k$ and $\gamma_0=(z+h)(k+z-h)$.

    (1) Now $l=k$. 
    \begin{align*}
        \frc. (v\otimes t^ie)=(4(2i-k)\nabla+4(i(1-2z-k)+kz+k^2-hk+h^2)-1)v\otimes t^ie\\
        +4i(k-i+1)tv\otimes t^{i-1}e- 4t^{-1}P_{\rm Sen}(\nabla)v\otimes t^{i+1}e.
    \end{align*}
    Suppose that $\sum_{i=0}^kv_i\otimes t^ie$ is an eigenvalue for $4h^2-1$. Consider the coefficient in $D_A$ of $1\otimes e$ we get $ 4ktv_1+(-4k\nabla+4(kz+k^2-hk+h^2)-1)v_0=(4h^2-1)v_0$. We see $v_1=t^{-1}(\nabla-(k+z-h))v_0$. We prove by induction that $v_{i}=\frac{1}{it}(\nabla-(k+z-h))v_{i-1}$ for $i\geq 1$. Suppose the statement holds for $i-1\geq 0$. Then $t^{-1}P_{\rm Sen}(\nabla)v_{i-1}=i(\nabla-(z+h)+1)v_i$. Consider the coefficient of $1\otimes t^ie$ we get 
    \[ -4i(\nabla-(z+h)+1)v_i+(4(2i-k)\nabla+4(i(1-2z-k)+kz+k^2-hk+h^2)-1)v_i+4(i+1)(k-i)tv_{i+1}=(4h^2-1)v_i\]
    which is equivalent to that $v_{i+1}=\frac{1}{(i+1)t}(\nabla-(k+z-h))v_i$ unless $i\geq k$. And if $i=k$ the equality holds under the induction assumption.
    
    Thus 
    \[\prod_{i=0}^{k-1}\frac{1}{t}(\nabla-(k+z-h))v_0=\frac{1}{t^k}(\nabla-(k+z-h)-(k-1))\cdots(\nabla-(k+z-h))v_0\in D_A.\]
    
    By the discussion in \S\ref{subsectiongeneralchange}, $D_A'=\{v\in D_A\mid \frac{1}{t}(\nabla-(k+z-h))v\in D_A\}$ defines the unique sub-$(\varphi,\Gamma)$-module of $D_A$ such that $D_A'[\frac{1}{t}]=D_A[\frac{1}{t}]$ and $D_A'$ is almost de Rham of weight $(1,k)$. By Proposition \ref{propositionappendixglueing}, the description of $T_{\lambda}^{\mu}D_A$ follows from the following statement:
    \begin{lemma}
        Let $M$ be a continuous $\Gamma$-representation over $(A\otimes_{\Q_p}K_m)[[t]]$ with the connection $\nabla$ as in Appendix \ref{sectionappendixfamilies} such that the characteristic polynomial of $\nabla$ on $M/tM$ is $(X-(k+z-h))(X-(z+h))$ for $k\in \Z_{\geq 0}$ and $h,z$ nilpotent. Then there is a direct sum decomposition
        \[M/t^kM=\oplus_{i=0}^{k-1}M[\nabla=(k+z-h+i)]\oplus \oplus_{i=0}^{k-1}M[\nabla=(z+h+i)]\]
        where all the direct summands are projective of rank one over $A\otimes_{\Q_p}K_m$.
    \end{lemma}
    (2) Suppose that $\sum_{i=0}^lv_i\otimes t^ie$ is an eigenvector of $\frc$ for the eigenvalue $(k+l-2h)^2-1$. Consider the coefficient of $1\otimes e$ we get $ 4ltv_1+(-4l\nabla+(2z+k+l)^2-4(k+z-h)(z+h)-1)v_0=((k+l-2h)^2-1)v_0$. We see $v_1=\frac{1}{t}(\nabla-(h+z))v_0$. We prove by induction that we must have $v_{i}=\frac{1}{it}(\nabla-(h+z))v_{i-1}$ for $i\geq 1$. Assume that the statement holds for $i-1\geq 0$. Then $t^{-1}P_{\Sen}(\nabla)v_{i-1}=i(\nabla-(k+z-h)+1)v_i$. Consider the coefficient of $1\otimes t^ie$ we get 
    \begin{align*}
        -4i(\nabla-(k+z-h)+1)v_i+(4(2i-l)\nabla+(2z+k-2i+l)^2-1-4(z+h)(k+z-h)+4i(l-i+1))v_i\\+4(i+1)(l-i)tv_{i+1}=((k+l-2h)^2-1)v_i
    \end{align*}
    which is equivalent to that $v_{i+1}=\frac{1}{(i+1)t}(\nabla-(h+z))v_i$ unless $i\geq l$. And if $i=l$ the equality holds under the statement. 
    
    To see the uniqueness. Assume that $D_{A}'\subset D_A$ with Sen weights $(0,k')$ and such that $D_{A}'[\frac{1}{t}]=D_{A}[\frac{1}{t}]$. Then $D_{\pdR}(D_A')=D_{\pdR}(D_A)$ and $\Fil^0D_{\pdR}(D_A')\subset \Fil^0D_{\pdR}(D_A)$ by definition. Since $\Fil^0D_{\pdR}(D_A'),\Fil^0D_{\pdR}(D_A)$ are both direct summand of $D_{\pdR}(D_A)$ of rank one, $\Fil^0D_{\pdR}(D_A')=\Fil^0D_{\pdR}(D_A')$ which determines $D_{\dif}^{m,+}(D_A')$ for some $m$ and also $D_A'$ by results in Appendix \ref{sectionappendixfamilies}. The description holds as for (1) and by an induction on $l$.
\end{proof}
From non-regular weights to regular weights we only treat the easiest case.
\begin{proposition}\label{propositionnonregulartoregular}
    Let $\Delta_A$ be an almost de Rham $(\varphi,\Gamma)$-module pointwisely with non-regular Sen weights $(0,0)$, $\mu=(-1,0)$ and $\lambda=(0,0)$. Then the natural map $T_{\mu}^{\lambda}\Delta_A\hookrightarrow \Delta_A\otimes_{L}V_{1}$ is an isomorphism. If the Sen polynomial of $\Delta_A$ is $T^2-\gamma_1T+\gamma_0\in A[T]$, then $(\frc-\gamma_1^2+4\gamma_0)^2=4(\gamma_1^2-4\gamma_0)$ on $\Delta_A\otimes_LV_1$. 
\end{proposition}
\begin{proof}
    We follow the notation in the proof of Proposition \ref{propositionregulartononregular}. By (\ref{equationactioncasimir}) (replacing $k+2z$ by $\gamma_1$ and $(k+z-h)(z+h)$ by $\gamma_0$), we see 
    \[(\frc-\gamma_1^2+4\gamma_0). (v\otimes t^ie)=(4(2i-1)\nabla+4i(1-\gamma_1)+2\gamma_1)v\otimes t^ie+4i(2-i)tv\otimes t^{i-1}e- 4t^{-1}P_{\rm Sen}(\nabla)v\otimes t^{i+1}e.\]
    For $i=0$ we have 
    \[(\frc-\gamma_1^2+4\gamma_0). (v\otimes e)=(-4\nabla+2\gamma_1)v\otimes e- 4t^{-1}(\nabla^2-\gamma_1\nabla+\gamma_0)v\otimes te\]
    and for $i=1$ 
    \[(\frc-\gamma_1^2+4\gamma_0). (v\otimes te)=(4\nabla+4-2\gamma_1)v\otimes te+4tv\otimes e.\]
    Then $(\frc-\gamma_1^2+4\gamma_0)^2-4(\gamma_1^2-4\gamma_0)=0$ on $\Delta_A\otimes_LV_1$ by a direct check.
\end{proof}

The following proposition describes the counit map of the adjunction of the translations in the case $k=1$.
\begin{proposition}\label{propositionunitcounit}
    Let $\lambda=(0,0)$ and $\mu=(-1,0)$. Let $D_{A}$ be an almost de Rham $(\varphi,\Gamma)$-module of weight $(1,0)$ over $A$ with Sen polynomial $(T-(z-h+1))(T-(z+h))$ where $z,h$ are nilpotent. Let $\Delta_{A}\subset t^{-1}D_{A}$ be the sub $(\varphi,\Gamma)$-module of Sen weights $(0,0)$ in Proposition \ref{propositionregulartononregular} which provides the isomorphism $\Delta_{A}=T_{\lambda}^{\mu}D_{A}$. Let $c=\frc-4h^2$ which acts on $D_{A}$ by $-4h$.
    
    Then the composite $T_{\mu}^{\lambda}\Delta_{A}=T_{\mu}^{\lambda}T_{\lambda}^{\mu}D_{A}\rightarrow D_{A} \hookrightarrow \Delta_{A}$ induced by the counit map $T_{\mu}^{\lambda}T_{\lambda}^{\mu}D_{A}\rightarrow D_{A}$ is equal to $T_{\mu}^{\lambda}\Delta_{A}(=\Delta_{A}\otimes_L\cR_L^+/X^2)\stackrel{\frac{1}{4}(c-4h)}{\rightarrow }T_{\mu}^{\lambda}\Delta_{A}[c+4h]\subset \Delta_{A}\otimes_L\cR_L^+/X^2\twoheadrightarrow \Delta_{A}$.
\end{proposition}
\begin{proof}
Before the proof, note that Proposition \ref{propositionregulartononregular} shows that $\Delta_{A}$ contains $D_{A}$ and identifies $\Delta_{A}$ with the preimage of $(t^{-1}D_{A}/D_{A})[\nabla=z-h]$. Also by Proposition \ref{propositionnonregulartoregular}, $(c-4h)(c+4h)$ acts as zero on $T_{\mu}^{\lambda}\Delta_A$.

Recall we identify $\cR_L^+/X^2\simeq \Sym^1L^2=Lx\oplus Ly$ where $1=y$ is the lowest weight vector, $t=X=x$. Let $(\Sym^1L^2)^{\vee}=Lx^*\oplus Ly^*$ be the dual where $x^*,y^*$ are the dual basis. We fix an isomorphism $Lt^{-1}\otimes_L\cR_L^+/X^2\simeq \det^{-1}\otimes_L\Sym^1L^2=(\Sym^1L^2)^{\vee}$ by $t^{-1}\otimes 1=x^*,t^{-1}\otimes t=-y^*$ (since $u^+.x^*=-y^*$).
The unit and counit map is induced by
\begin{align*}
        L&\rightarrow (Lx^*\oplus Ly^*)\otimes_L (Lx\oplus Ly)\rightarrow L\\ 
        1&\mapsto x^*\otimes x+y^*\otimes y
\end{align*}
where the second map is the evaluation map.

The map $T_{\mu}^{\lambda}\Delta_{A}=T_{\mu}^{\lambda}T_{\lambda}^{\mu}D_{A}\rightarrow D_{A} \hookrightarrow \Delta_{A}$ factors through (by Lemma \ref{lemmatwist})
\[ T_{\mu}^{\lambda}T_{\lambda}^{\mu}D_{A}=\Delta_{A}\otimes_L \cR_L^+/X^2\hookrightarrow t^{-1}D_{A}\otimes_L \cR_L^+/X^2\otimes_L \cR_L^+/X^2\rightarrow D_{A}\hookrightarrow \Delta_{A}\]
where $v\in \Delta_{A}$ is sent to $v\otimes e+t^{-1}(\nabla-(z-h))v\otimes te$ in $t^{-1}D_{A}\otimes_L \cR_L^+/X^2$ by the proof of Proposition \ref{propositionregulartononregular}. The image of $(v\otimes e+t^{-1}(\nabla-(z-h))v\otimes te)\otimes e\in (t^{-1}D_{A}\otimes_L \cR_L^+/X^2)\otimes_L \cR_L^+/X^2$ via the evaluation map in $D_{A}$ is $-(\nabla-(z-h))v$ and the image of $(v\otimes e+t^{-1}(\nabla-(z-h))v\otimes te)\otimes te$ is $tv$. 

On the other hand, for $v\in \Delta_{A}$, using (\ref{equationactioncasimir}) again for $\Delta_{A}\otimes_L\cR_L^+/X^2$,
\[(c\pm 4h). (v\otimes t^ie)=(4(2i-1)\nabla+4i(1-2z)+4z\pm 4h)v\otimes t^ie+4i(2-i)tv\otimes t^{i-1}e- 4t^{-1}P_{\rm Sen}(\nabla)v\otimes t^{i+1}e.\]
For $i=0$ we have 
\[(c-4h). (v\otimes e)=4(-\nabla+(z-h))v\otimes e- 4t^{-1}P_{\rm Sen}(\nabla)v\otimes te.\]
and 
\[(c- 4h). (v\otimes te)=4(\nabla+(1-z-h))v\otimes te+4tv\otimes e.\]
Hence the image of $\frac{1}{4}(c-4h)(v\otimes e)$ in $\Delta_{A}$ is $(-\nabla+z-h)v$ and the image of $\frac{1}{4}(c-4h)(v\otimes te)$ in $\Delta_{A}$ is $tv$.
\end{proof}

\section{Geometric translations for $\GL_2(\Q_p)$}\label{sectiongeometrictranslation}
We will prove our main results on the geometric translations in the $\GL_2(\Q_p)$ case (Theorem \ref{theorem}). We first make necessary preparations in \S\ref{subsectionformalcompletion} for the study of the direct image of $(\varphi,\Gamma)$-modules over formal rigid spaces in Construction \ref{construction}. We prove our main result in \S\ref{subcectiondireftimageandtranslation} which compares the direct image and the translation. The interested reader may consult the pointwise cases in \S\ref{subsectionspecialization} first or the description in Corollary \ref{corollarydirectimagephigammamodule} for the basic ideas on computations. We take $G=\GL_{2/L}$, $K=\Q_p$ and $f:\tildefrg\rightarrow \frg$.
\subsection{Formal completion of $(\varphi,\Gamma)$-modules}\label{subsectionformalcompletion}
A $(\varphi,\Gamma)$-module $D_{A}^r$ over $\cR_{A}^r$ for an affinoid algebra $A$ over $L$ is equivalently a $(\varphi,\Gamma)$-bundle on $\bbU^{r}_{\Sp(A)}=\Sp(A)\times_L \bbU^{r}$, namely a $\Gamma$-equivariant vector bundle $\cD^r_{\Sp(A)}$ on $\bbU^{r}_{\Sp(A)}$ equipped with an isomorphism $\varphi^*\cD^r_{\Sp(A)}\simeq \cD^r_{\Sp(A)}|_{\bbU^{r/p}_{\Sp(A)}}$ commuting with $\Gamma$-actions. This point of view will be more convenient for the consideration of cohomologies (but will not work for translations), and we need a similar description for formal completions of $(\varphi,\Gamma)$-modules. A basic discussion on coherent modules over formal rigid spaces can be found in Appendix \ref{sectionappendixformalfunction}.

Let $\Sp(A)\in \Rig_L$ with an ideal $I$ of $A$. Let $\frY_n=\Sp(A/I^n)$ and $\frY^{\wedge}:=\varinjlim_n \frY_n$. The latter is a ringed site with the structure sheaf $\cO_{\frY^{\wedge}}=\varprojlim_n\cO_{\frY_n}$ whose global section is $A^{\wedge}:=\varprojlim_nA/I^n$. There is a sheaf of $\cO_{\frY^{\wedge}}$-module $\cR_{\frY^{\wedge}}^r:=\varprojlim_n\cR_{\frY_n}^r$ for small enough $r>0$ whose global section is $\cR_{A^{\wedge}}^r:=\varprojlim_n\cR_{A/I^n}^r$. We define $\cR_{A^{\wedge}}:=\varinjlim_r\cR_{A^{\wedge}}^r$.

Recall that for $r$ small enough and $s<r$, $\varphi:\cR^{[s,r]}_A\rightarrow \cR^{[s/p,r/p]}_A,\cR^{r}_A\rightarrow \cR^{r/p}_A$ make the targets finite free of rank $p$ over the sources and we have $\cR_{A}^{r/p}=\oplus_{i=0}^{p-1}(1+X)^i\varphi(\cR_{A}^r)$. Taking $I$-adic completion, we get $\cR_{A^{\wedge}}^{r/p}=\oplus_{i=0}^{p-1}(1+X)^i\varphi(\cR_{A^{\wedge}}^r)$ and taking direct limit $\cR_{A^{\wedge}}=\oplus_{i=0}^{p-1}(1+X)^i\varphi(\cR_{A^{\wedge}})$.

\begin{definition}\label{definitionphigammamodulecomplete}
    Let $\frX^{\wedge}$ be the formal completion of a rigid space $\frX$ along a Zariski closed subspace defined by a coherent ideal sheaf $\cI$ as in Definition \ref{definitionformalcompletion}. 
    \begin{enumerate}
        \item A $(\varphi,\Gamma)$-module (resp. $(\varphi,\Gamma,\frg)$-module) $D_{\frX^{\wedge}}^r$ over $\cR_{\frX^{\wedge}}^r:=\varprojlim_{n}\cR_{\frX_n}^r$ is a locally finite projective $\cR^r_{\frX^{\wedge}}$-module with a $(\varphi,\Gamma)$-structure, i.e., an isomorphism 
        \[\varphi:\varphi^*D_{\frX^{\wedge}}^r:=\cR^{r/p}_{\frX^{\wedge}}\otimes_{\varphi,\cR^r_{\frX^{\wedge}}}D_{\frX^{\wedge}}^r\simeq \cR^{r/p}_{\frX^{\wedge}}\otimes_{\cR^r_{\frX^{\wedge}}}D_{\frX^{\wedge}}^r\] 
        commuting with a semilinear action of $\Gamma$ (resp. $(\varphi,\Gamma)$-structure and an action of $\frg$) such that there exist $(\varphi,\Gamma)$-modules (resp. $(\varphi,\Gamma, \frg)$-modules as in Definition \ref{definitionphgammagmodule}) $D^r_{\frX_n}$ over $\cR^r_{\frX_n}$ satisfying $D^r_{\frX_{n+1}}/\cI^{n}=D^r_{\frX_{n}}$ for all $n\geq 1$ and there exists an isomorphism $D^r_{\frX^{\wedge}}\simeq \varprojlim_n D^r_{\frX_n}$.
        \item A $(\varphi,\Gamma)$-module $D_{\frX^{\wedge}}$ over $\cR_{\frX^{\wedge}}:=\varinjlim_r\cR_{\frX^{\wedge}}^r$ is a $\cR_{\frX^{\wedge}}$-module with $(\varphi,\Gamma)$-structure such that there exists a $(\varphi,\Gamma)$-module $D^r_{\frX^{\wedge}}$ over $\cR_{\frX^{\wedge}}^r$ and $D_{\frX^{\wedge}}=\cR_{\frX^{\wedge}}\otimes_{\cR^r_{\frX^{\wedge}}}D^r_{\frX^{\wedge}}=\varinjlim_{r'<r}\cR_{\frX^{\wedge}}^{r'}\otimes_{\cR^r_{\frX^{\wedge}}}D^r_{\frX^{\wedge}}$.
    \end{enumerate}  
\end{definition}
The underlying $\cR_{\frY}^r$-module of a $(\varphi,\Gamma)$-module $D_{\frY}^r$ over $\cR_{\frY}^r$ is the $\cR_{\frY}^r$-module associated to its global section $D_{A}^r$ which is finite projective over $\cR_{A}^r$: for any affinoid open $\Sp(B)\subset \frY$, $\Gamma(\Sp(B),D_{\frY}^r)=D_{A}^r\otimes_{\cR_{A}^r}\cR_{B}^r$. A vector bundle over $\bbU_{\frY}^r$ is equivalently a compatible family of finite projective modules over $\cR^{[s,r]}_{\frY}$ or $\cR^{[s,r]}_{A}$ for $s<r$. We can similarly define the notion of vector bundles or $\varphi$-bundles over $\bbU^r_{\frY^{\wedge}}=\varinjlim_{s,n}\bbU^{[s,r]}_{\frY_n}$. 
\begin{lemma}\label{lemmacompletionphigammamodule}
    Let $g:\bbU^{r}_{\frY^{\wedge}}\rightarrow \frY^{\wedge}$ be the projection.
    \begin{enumerate}
        \item Suppose that $D_{A^{\wedge}}^r$ is a finite projective $\cR_{A^{\wedge}}^r$-module, then $D_{A^{\wedge}}^r$ is $I$-adically complete and is the global section of a vector bundle $(\cR^{[s,r]}_{A/I^n}\otimes_{\cR_{A^{\wedge}}^{r}}D_{A^{\wedge}}^r)_{s,n}$ over $\bbU^r_{\frY^{\wedge}}$. Moreover for $r'\leq r$, $\cR^{r'}_{A^{\wedge}}\otimes_{\cR^{r}_{A^{\wedge}}}D^r_{A^{\wedge}}=\varprojlim_{n}D_{A/I_n}^{r'}$. 
        \item Taking global sections on $\frY^{\wedge}$ induces an equivalence between the category of $\cR^r_{\frY^{\wedge}}$-modules of the form $D^r_{\frY^{\wedge}}=\varprojlim_nD^r_{\frY_n}$ where $D^r_{\frY_n}$ are $\cR_{\frY_n}^r$-modules associated to projective $\cR_{A/I^n}^r$-modules and $D^r_{\frY_{n+1}}/I^n=D^r_{\frY_{n}}$ for all $n$ and the category of finite projective $\cR_{A^{\wedge}}^r$-modules. 
        \item The direct image functor $g_*$ induces an equivalence of categories between $\varphi$-bundles over $\bbU^r_{\frY^{\wedge}}$ and finite projective $\varphi$-modules over $\cR_{\frY^{\wedge}}^r$. 
    \end{enumerate} 
\end{lemma}
\begin{proof}
    (1) If $D_{A^{\wedge}}^r$ is a finite projective $\cR_{A^{\wedge}}^r$-module, we may find another finite module $D'$ such that $D_{A^{\wedge}}^r\oplus D'\simeq (\cR^r_{A^{\wedge}})^n$ for some $n$. Since $(\cR^r_{A^{\wedge}})^n$ is $I$-adically complete \cite[\href{https://stacks.math.columbia.edu/tag/05GG}{Tag 05GG}]{stacks-project}, so is its direct summand. Hence $D_{A^{\wedge}}^r=\varprojlim_n D_{A/I^n}^r=\varprojlim_{n,s}D_{A/I^n}^{[s,r]}=\varprojlim_{n,s}D_{A/I^n}^{r}\otimes_{\cR_{A/I^n}^r}\cR^{[s,r]}_{A/I^n}=\varprojlim_s(D_{A^{\wedge}}^r\otimes_{\cR_{A^{\wedge}}^r}\cR_{A^{\wedge}}^{[s,r]})$ (for the last equality we use that $D_{A^{\wedge}}^r\otimes_{\cR_{A^{\wedge}}^r}\cR_{A^{\wedge}}^{[s,r]}$ is $I$-adically complete being finite projective over $\cR_{A^{\wedge}}^{[s,r]}$). The statement for $D_{A^{\wedge}}^{r'}$ follows similarly.

    (2) Suppose that $(D_{\frY_n}^r)_n$ is a collection of finite projective $\cR_{\frY_n}^r$-modules (giving vector bundles on $\bbU^r_{\frY_n}$) such that $D_{\frY^{\wedge}}^r:=\varprojlim_nD_{\frY_n}^r$. Write $D_{A/I^n}^r:=\Gamma(\frY_n,D_{\frY_n}^r)$. By the definition of inverse limit of sheaves, for any affinoid open $\Sp(B)\subset \frY$, $\Gamma(\Sp(B/I),D_{\frY^{\wedge}}^r)=\varprojlim_n\Gamma(\Sp(B/I^n),D_{\frY_n}^r)=D_{B^{\wedge}}^r:=\varprojlim_n\cR^r_{B/I^n}\otimes_{\cR^r_{A/I^n}}D_{A/I^n}^r$. Hence the section is finite projective over $\cR_{B^{\wedge}}^r$ and $D_{B^{\wedge}}^r/I^n=\cR^r_{B/I^n}\otimes_{\cR^r_{A/I^n}}D_{A/I^n}^r$ by (2) of Lemma \ref{lemmaflatcompletion}. Then $D_{B^{\wedge}}^r=D_{A^{\wedge}}^r\otimes_{\cR_{A^{\wedge}}^r}\cR_{B^{\wedge}}^r$ (both sides are $I$-adically complete by (1)). Take $B=A$ we see the global section is finite projective. The essential surjectivity follows from (1): given a $(\varphi,\Gamma)$-module $D_{A^{\wedge}}^r$ finite projective over $\cR_{A^{\wedge}}^r$, $D_{A^{\wedge}}^r/I^n$ defines sheaves $D_{\frY_n}^r$ and $\Gamma(\frY^{\wedge},\varprojlim_nD_{\frY_n}^r)=D_{A^{\wedge}}^r$. The fully faithfulness follows as for (2) of Lemma \ref{lemmaGAGAaffinoid}.

    (3) Let $\cD^r_{\frY^{\wedge}}=(\cD^{[s,r]}_{\frY_n})_{s,n}$ be a $\varphi$-bundle on $\bbU^r_{\frY^{\wedge}}$. By \cite[Prop. 2.2.7]{kedlaya2014cohomology}, the global section of $(\cD^{[s,r]}_{\frY_n})_{s}$ for a fix $n$ defines via $g_*$ a finite projective $\varphi$-module denoted by $D_{\frY_n}^r$ over $\cR_{\frY_n}^r$. By the equivalence of \textit{loc. cit.}, we have $D_{\frY_n}^r/I^{n-1}=D^r_{\frY_{n-1}}$. Hence the $\cR_{\frY^{\wedge}}^r$-module $D_{\frY^{\wedge}}^r:=\varprojlim_n D_{\frY_{n}}^r$ is finite projective over $\cR^r_{\frY^{\wedge}}$ by (2). Thus $g_*$ sends $\varphi$-bundles to fintie projective $\varphi$-modules (see also Lemma \ref{lemmadirectimagephibundles} below). The fully faithfulness of $g_*$ follows from mod $I^n$-cases in \cite[Prop. 2.2.7]{kedlaya2014cohomology} and that for two $\varphi$-modules $D_1,D_2$ (which are $I$-adically complete by (1))
    \[\Hom_{\cR^r_{\frY^{\wedge}}}(D^r_1,D^r_2)=\varprojlim_n\Hom_{\cR^r_{\frY^{\wedge}}}(D^r_1,D^r_2/I^n)=\varprojlim_n\Hom_{\cR^r_{\frY_n}}(D^r_1/I^n,D^r_2/I^n).\]
    The essential surjectivity follows from (1) and (2). 
\end{proof}

\begin{remark}
    Since we will only deal with finite projective modules, we ignore if we can define the notion of coadmissible modules over the topological ring $\cR_{A^{\wedge}}^r=\varprojlim_{s,n}\cR_{A/I^n}^{[s,r]}$ as for modules over $\cR_{A}^r$ in \cite[\S 3]{schneider2003algebras}.
\end{remark}

We construct below our major players: some $(\varphi,\Gamma)$-module $D_{\tildefrY^{\wedge}}^r$ over a formal rigid space $\tildefrY^{\wedge}$ projective over $\frY^{\wedge}$ and we will study its direct image along the map $\tildefrY^{\wedge}\rightarrow \frY^{\wedge}$.
\begin{construction}\label{construction}
    Fix $\bh=(h_1,h_2)\in\Z^2, h_1< h_2$. Let $\frY=\Sp(A)\in \Rig_L$ be an affinoid with a $(\varphi,\Gamma)$-module $\Delta_A$ of rank $2$ over $\frY$ base changed from a $(\varphi,\Gamma)$-module $\Delta_A^r$ over $\cR_A^r$ for some $r>0$. We assume that $r$ is taken such that the number $m(r)$ is large enough for $D_{\dif}^{m(r),+}(\Delta_A^r)$ in the sense of Definition \ref{definitioninfinitedif}. 
    
    Let $I$ be the ideal of $A$ generated by zeroth and the first coefficients of the Sen polynomial of $\Delta_A$. Let $\frY_n=\Sp(A/I^n)$ and $\frY^{\wedge}=\varinjlim\frY_n$. By definition, $\frY^{\wedge}=\frY\times_{\frX_2}(\frX_2)^{\wedge}_{0}$. According to Proposition \ref{propositionfiberproduct}, $\tildefrY^{\wedge}:=f_{\bh}^{-1}(\frY^{\wedge})=\frY^{\wedge}\times_{(\frX_2)_0^{\wedge}}(\frX_2)_{\bh}^{\wedge}=\frY^{\wedge}\times_{\frg/G}\tildefrg/G$. In the following, we construct explicitly these spaces and the universal $(\varphi,\Gamma)$-module $D_{\tildefrY^{\wedge}}$ over $\cR_{\tildefrY^{\wedge}}$.
    \begin{center}
        \begin{tikzcd}
            \tildefrY^{\wedge}\arrow[r]\arrow[d,"f_{\bh}"]&(\frX_{2})^{\wedge}_{\bh}\arrow[r,"D_{\rm pdR}"]\arrow[d,"f_{\bh}"]& \tildefrg/G \arrow[d,"f"]\\ 
            \frY^{\wedge}\arrow[r]&(\frX_2)^{\wedge}_{0}\arrow[r,"D_{\rm pdR}"]&  \frg/G.
        \end{tikzcd}
    \end{center}

    Write $\Delta^r_{\frY_n}$ for the base change of $\Delta^r_A$. Proposition \ref{propositionappendixequivalencealmostdeRham} gives a vector bundle $D_{\rm pdR}(\Delta_{\frY_n})$ on $\frY_n$ together with a nilpotent operator $\nu_{\frY_n}$ such that $D_{\rm pdR}(\Delta_{\frY_n})\otimes_{\cO_{\frY_n}}\cO_{\frY_{n-1}}=D_{\rm pdR}(\Delta_{\frY_{n-1}})$ for $n\geq 2$. Let $\frY_n^{\square}$ be the $\GL_2$-torsor over $\frY_n$ trivializing $D_{\rm pdR}(D_{\frY_n})$. Then $\frY_n^{\square}\times_{\frY_n}\frY_{n-1}=\frY_{n-1}^{\square}$. The nilpotent operator $\nu_{\frY_n}$ induces $\frY_n^{\square}\rightarrow \frg$. Let $\widetilde{\frY}_n^{\square}=\frY_n^{\square}\times_{\frg}\tildefrg$ and $\widetilde{\frY}_n=[\widetilde{\frY}_n^{\square}/\GL_2]$. Then $\widetilde{\frY}_n$ is a rigid analytic space projective over $\frY_n$ and $\widetilde{\frY}_n\times_{\frY_{n}}\frY_{n-1}=\widetilde{\frY}_{n-1}$. Let $\widetilde{\frY}^{\wedge}=\varinjlim \widetilde{\frY}^{\wedge}_n$ (see Remark \ref{remarktrivialtorsor1} below). Let $\Delta_{\tildefrY_n}^r$ be the pullback and $\Delta_{\tildefrY^{\wedge}}^r:=\varprojlim_n\Delta_{\tildefrY_n}^r$. 

    On each $\widetilde{\frY}_n$, the universal $\nu$-stable filtration of $D_{\pdR}(\Delta_{\tildefrY_n})$ provided from $\widetilde{\frg}$ gives a $\Gamma$-invariant $(K_{m}\otimes_{\Q_p}\cO_{\widetilde{\frY}_n})[[t]]$-lattice of weight $\bh$ inside $D_{\dif}^m(\Delta_{\widetilde{\frY}_n})$ by Proposition \ref{propositionappendixequivalencealmostdeRham}, and by Proposition \ref{propositionappendixglueing} we obtain a modification $D_{\widetilde{\frY}_n}^r$ of $\Delta_{\widetilde{\frY}_n}^r$ on $\widetilde{\frY}_n$ which is a $(\varphi,\Gamma)$-module over $\cR_{\tildefrY_n}^r$. Let $D_{\widetilde{\frY}^{\wedge}}^r=\varprojlim_n D_{\widetilde{\frY}_n}^r$.  We write $\cD_{\widetilde{\frY}_n^{\wedge}}^r$ for $D^r_{\widetilde{\frY}_n^{\wedge}}$ viewed as coherent sheaves on $\bbU^r_{\widetilde{\frY}_n}$ and similarly $\cD_{\widetilde{\frY}^{\wedge}}^r=\varprojlim_n \cD_{\widetilde{\frY}_n}^r$.  

    Finally, let $D_{\tildefrY^{\wedge}}=\varinjlim_rD_{\tildefrY^{\wedge}}^r$, a $(\varphi,\Gamma)$-module in the sense of Definition \ref{definitionphigammamodulecomplete}.
\end{construction}

\begin{remark}\label{remarktrivialtorsor1}
    Locally on $\frY$ we can choose a trivialization of $D_{\rm pdR}(D_{A^{\wedge}})$ as in Remark \ref{remarktrivialtorsor}. With the trivialization the operator $\nu$ on $D_{\rm pdR}(D_{A^{\wedge}})\simeq (A^{\wedge})^2$ induces maps $\frY_n\rightarrow \frg$ for all $n$. Then we can describe locally the formal rigid space in the above construction via $\widetilde{\frY}_n\simeq \frY_n\times_{\frg}\widetilde{\frg}$ and $\widetilde{\frY}^{\wedge}=\mathrm{colim}_n\widetilde{\frY}_n$. We also note that in this (local) case, $A^{\wedge}$-linear operator on $(A^{\wedge})^2=D_{\pdR}(D_{A^{\wedge}})$ induces a scheme map $ \Spec(A^{\wedge})\rightarrow \frg^{\alg}$. The proper formal rigid space $\widetilde{\frY}^{\wedge}$ over $\frY^{\wedge}$ is the relative analytification of the formal completion of a projective scheme $\Spec(A^{\wedge})\times_{\frg^{\alg}}\widetilde{\frg}^{\alg}\subset\bP^1_{\Spec(A^{\wedge})}$ over $\Spec(A^{\wedge})$ in the way discussed before Corollary \ref{corollaryGAGA}.
\end{remark}

By discussions in Appendix \ref{sectionappendixformalfunction}, $\cD_{\widetilde{\frY}^{\wedge}}^r=\varprojlim_n \cD_{\widetilde{\frY}_n}^r$ is a coherent $\cO_{\bbU^r_{\widetilde{\frY}^{\wedge}}}$-module and is locally free of rank two (Lemma \ref{lemmacompletionphigammamodule}).

We still write $f_{\bh}:\widetilde{\frY}^{\wedge}\rightarrow\frY^{\wedge}$ for the morphism of ringed spaces with the sheaves of abstract rings $\cR_{\widetilde{\frY}^{\wedge}}^r$ or $\cO_{\widetilde{\frY}^{\wedge}}$, and $\cR_{\frY^{\wedge}}^r$ or $\cO_{\frY^{\wedge}}$, via the map $f_{\bh}^{-1}\cR_{\frY^{\wedge}}^r\rightarrow \cR_{\widetilde{\frY}^{\wedge}}^r$ or $f_{\bh}^{-1}\cO_{\frY^{\wedge}}\rightarrow \cO_{\widetilde{\frY}^{\wedge}}$.

\begin{definition}\label{definitiondirectimage}
    We write $Rf_{\bh,n,*}D^r_{\widetilde{\frY}_n}$, resp. $Rf_{\bh,*}D^r_{\widetilde{\frY}^{\wedge}}$, resp. $Rf_{\bh,*}D_{\widetilde{\frY}^{\wedge}}$ to be the derived direct image of the sheaf of $\cO_{\widetilde{\frY}_n}$-module $D^r_{\widetilde{\frY}_n}$, resp. $\cO_{\tildefrY^{\wedge}}$-module $D^r_{\widetilde{\frY}^{\wedge}}$, resp. $\cO_{\tildefrY^{\wedge}}$-module $D_{\widetilde{\frY}^{\wedge}}$ along the map $f_{\bh}$ \cite[\href{https://stacks.math.columbia.edu/tag/071J}{Tag 071J}]{stacks-project}. 
\end{definition}
\begin{remark}
    The direct image will only work well after Lemma \ref{lemmadirectimagephibundles} below. The sheaves $f_{\bh,n,*}D^r_{\widetilde{\frY}_n}$ and $f_{\bh,*}D^r_{\widetilde{\frY}^{\wedge}}$ remain to have the $\Gamma$-actions but may not be $(\varphi,\Gamma)$-modules as in Definition \ref{definitionphigammamodulecomplete}, namely they may not be projective over $\cR_{\frY_n}^r$ or $\cR^r_{\frY^{\wedge}}$ and the map 
    \[\varphi:R^if_{\bh,*}D_{\tildefrY^{\wedge}}^r\rightarrow R^if_{\bh,*}(\cR_{\tildefrY^{\wedge}}^{r/p}\otimes_{\cR_{\tildefrY^{\wedge}}^r}D_{\tildefrY^{\wedge}}^r)\] 
    induced by $\varphi:D_{\tildefrY^{\wedge}}^r\rightarrow \cR_{\tildefrY^{\wedge}}^{r/p}\otimes_{\cR_{\tildefrY^{\wedge}}^r}D_{\tildefrY^{\wedge}}^r$ a priori may not factor through $\cR_{\frY^{\wedge}}^{r/p}\otimes_{\cR_{\frY^{\wedge}}^r}R^if_{\bh,*}D_{\tildefrY^{\wedge}}^r$. 
\end{remark}
We use $f_{\bh,n}^r: \bbU^r_{\widetilde{\frY}_n}\rightarrow \bbU^r_{\frY_n}$ to denote the base change of $f_{\bh}$ and similarly $f_{\bh}^r: \bbU^r_{\widetilde{\frY}^{\wedge}}\rightarrow \bbU^r_{\frY^{\wedge}}$. By Corollary \ref{corollaryGAGA}, $Rf_{\bh,*}^r$ sends coherent $\cO_{\bbU^r_{\widetilde{\frY}^{\wedge}}}$-modules to coherent $\cO_{\bbU^r_{\frY^{\wedge}}}$-modules. Consider the diagram
\begin{center}
    \begin{tikzcd}
        \bbU^r_{\widetilde{\frY}^{\wedge}}\arrow[d,"f_{\bh}^r"]\arrow[r,"\widetilde{g}"]
        & \widetilde{\frY}^{\wedge}\arrow[d,"f_{\bh}"]\\ 
        \bbU^r_{\frY^{\wedge}}\arrow[r, "g"]
        & \frY^{\wedge}
    \end{tikzcd}
\end{center}
of ringed sites with structure sheaves $\cO_{\bbU^r_{\widetilde{\frY}^{\wedge}}},\cO_{\tildefrY^{\wedge}}$, etc.

\begin{lemma}\label{lemmadirectimagephibundles}
    Let $D_{\widetilde{\frY}^{\wedge}},\cD_{\widetilde{\frY}^{\wedge}}$ be as in Construction \ref{construction}.
    \begin{enumerate}
        \item For each $n\geq 1$, $R\widetilde{g}_{*}\cD^r_{\widetilde{\frY}_n}=D^r_{\widetilde{\frY}_n}$ as modules over $\cR^r_{\widetilde{\frY}_n}$. Similarly $R\widetilde{g}_{*}\cD^r_{\widetilde{\frY}^{\wedge}}=D^r_{\widetilde{\frY}^{\wedge}}$. Hence $Rf_{\bh,*}D_{\widetilde{\frY}^{\wedge}}^r=Rg_*Rf^r_{\bh,*}\cD^r_{\tildefrY^{\wedge}}$ and $Rf_{\bh,*}D_{\widetilde{\frY}_n}^r=Rg_*Rf^r_{\bh,*}\cD^r_{\tildefrY_n}$ for all $n\geq 1$. 
        \item  As sheaves of $\cR_{\frY^{\wedge}}^r$-modules with $\Gamma$-actions, $R^if_{\bh,*}D^r_{\widetilde{\frY}^{\wedge}}=g_*R^if_{\bh,*}^r\cD^r_{\widetilde{\frY}^{\wedge}}$ and are isomorphic to the inverse limit $\varprojlim_ng_*R^if_{\bh,*}^r\cD^r_{\widetilde{\frY}_n}$.
        \item For $r'\leq r$ and $i\geq 0$, the natural map $\cR^{r'}_{\frY^{\wedge}}\otimes_{\cR^{r}_{\frY^{\wedge}}}R^if_{\bh,*}D_{\tildefrY^{\wedge}}^r\rightarrow R^if_{\bh,*}(\cR^{r'}_{\tildefrY^{\wedge}}\otimes_{\cR^{r}_{\tildefrY^{\wedge}}}D_{\tildefrY^{\wedge}}^r)$ is an isomorphism provided that $R^if_{\bh,*}^r\cD^r_{\tildefrY^{\wedge}}$ is a locally finite free $\cO_{\bbU^r_{\frY^{\wedge}}}$-module. In this case, $R^if_{\bh,*}D^r_{\tildefrY^{\wedge}}$ is a $(\varphi,\Gamma)$-module over $\cR_{\frY^{\wedge}}^r$.
        \item For each $i$, $R^if_{\bh,*}D_{\widetilde{\frY}^{\wedge}}=\varinjlim_{r}R^if_{\bh,*}D^r_{\widetilde{\frY}^{\wedge}}$. Under the assumption in (3), the direct image $R^if_{\bh,*}D_{\widetilde{\frY}^{\wedge}}$ is a $(\varphi,\Gamma)$-module and $R^if_{\bh,*}D_{\widetilde{\frY}^{\wedge}}=\cR_{\frY^{\wedge}}\otimes_{\cR^{r}_{\frY^{\wedge}}}R^if_{\bh,*}D_{\tildefrY^{\wedge}}^r$.
    \end{enumerate}
\end{lemma}
\begin{proof}
    (1) The first statement is classical and follows from the vanishing of higher coherent cohomologies of quasi-Stein spaces $\bbU^r$ \cite{kiehl1967theorem}. The derived direct image commutes with derived inverse limit \cite[\href{https://stacks.math.columbia.edu/tag/0BKP}{Tag 0BKP}]{stacks-project}. The inverse systems $(\cD^r_{\widetilde{\frY}_n})_n$ and $(D^r_{\widetilde{\frY}_n})_n$ are Mittag-Leffler. We get that $R\widetilde{g}_*\cD_{\widetilde{\frY}^{\wedge}}^r=R\widetilde{g}_*R\varprojlim_n\cD_{\widetilde{\frY}_n}^r=R\varprojlim_nR\widetilde{g}_*\cD_{\widetilde{\frY}_n}^r=D_{\widetilde{\frY}^{\wedge}}^r$. 

    (2) By Corollary \ref{corollaryGAGA} (and its proof), for each $i\geq 0$, $R^if_{\bh,*}^r\cD^r_{\widetilde{\frY}^{\wedge}}$ is a coherent $\cO_{\bbU^r_{\frY^{\wedge}}}$-module and $R^if_{\bh,*}^r\cD^r_{\widetilde{\frY}^{\wedge}}=\varprojlim_nR^if_{\bh,*}^r\cD^r_{\widetilde{\frY}_n}$ is an inverse limit of a Mittag-Leffler inverse system. We see 
    \begin{align*}
        Rg_*R^if_{\bh,*}^r\cD^r_{\widetilde{\frY}^{\wedge}}=Rg_*R\varprojlim_nR^if_{\bh,*}^r\cD^r_{\widetilde{\frY}_n}=R\varprojlim_nRg_*R^if_{\bh,*}^r\cD^r_{\widetilde{\frY}_n}\\=R\varprojlim_ng_*R^if_{\bh,*}^r\cD^r_{\widetilde{\frY}_n}=\varprojlim_ng_*R^if_{\bh,*}^r\cD^r_{\widetilde{\frY}_n}.
    \end{align*}
    The last equality follows from that the exact functor $g_*$ sends a Mittag-Leffler system to a Mittag-Leffler system.

    (3) Suppose that $R^if_{\bh,*}^r\cD^r_{\tildefrY^{\wedge}}$ is a vector bundle. The isomorphism $\varphi:\varphi^*\cD_{\tildefrY^{\wedge}}^r\simeq \cD_{\tildefrY^{\wedge}}^r|_{\bbU_{\tildefrY^{\wedge}}^{r/p}}$ induces via $Rf_{\bh,*}^{r/p}$ isomorphisms $\varphi^*Rf_{\bh,*}^r\cD_{\tildefrY^{\wedge}}^r\simeq Rf_{\bh,*}^{r/p}(\cD_{\tildefrY^{\wedge}}^r|_{\bbU_{\tildefrY^{\wedge}}^{r/p}})=(Rf_{\bh,*}^{r}\cD_{\tildefrY^{\wedge}}^r)|_{\bbU_{\frY^{\wedge}}^{r/p}}$ (by flat base changes). By (3) of Lemma \ref{lemmacompletionphigammamodule} and (2), $R^if_{\bh,*}D^r_{\tildefrY^{\wedge}}$ is a finite projective $D^r_{\frY^{\wedge}}$-module which also ensures that $\cR^{r'}_{\frY^{\wedge}}\otimes_{\cR^{r}_{\frY^{\wedge}}}R^if_{\bh,*}D_{\tildefrY^{\wedge}}^r$ is finite projective corresponding to the restriction of $R^if_{\bh,*}^r\cD_{\tildefrY^{\wedge}}^r$ to $\bbU^{r'}_{\frY^{\wedge}}$. The isomorphism follows from that $\cR^{r'}_{\tildefrY^{\wedge}}\otimes_{\cR^{r}_{\tildefrY^{\wedge}}}D_{\tildefrY^{\wedge}}^r$ corresponds to the restriction of $\cD_{\tildefrY^{\wedge}}^r$ to $\bbU_{\tildefrY^{\wedge}}^{r'}$ via $(\widetilde{g}|_{\bbU_{\tildefrY^{\wedge}}^{r'}})_*$ and the $\varphi$-structure on $R^if_{\bh,*}D^r_{\tildefrY^{\wedge}}$ comes from the $\varphi$-structure on $R^if_{\bh,*}^r\cD^r_{\tildefrY^{\wedge}}$ via $(g|_{\bbU_{\frY^{\wedge}}^{r/p}})_*$.

    (4) The map $\widetilde{\frY}_1\rightarrow \frY_1$ is a proper map between quasi-compact spaces, the direct image of $\cO_{\tildefrY^{\wedge}}$-modules can be computed as $R^if_{\bh,*}D_{\tildefrY^{\wedge}}=R^if_{\bh,*}\varinjlim_rD^r_{\tildefrY^{\wedge}}=\varinjlim_rR^if_{\bh,*}D^r_{\tildefrY^{\wedge}}$ by \cite[\href{https://stacks.math.columbia.edu/tag/0739}{Tag 0739}, \href{https://stacks.math.columbia.edu/tag/07TA}{Tag 07TA}]{stacks-project} and \cite[Exposé V, Prop. 5.1]{grothendieck2006SGA4} (and \cite[Prop. 6.3/4]{bosch2014formal}). The last equality follows from that colimits commute with tensor product.
\end{proof}
\subsection{Direct image and translation}\label{subcectiondireftimageandtranslation}
In this subsection we prove our main results on geometric translations of $(\varphi,\Gamma)$-modules in the Construction \ref{construction}. Discussions in \S\ref{sectionflatlocalmodel} are served for the following hypothesis.
\begin{hypothesis}\label{hypothesisflat}
    We assume that the map $\frY^{\wedge,\square}:=\varinjlim_n\frY_n^{\square}\rightarrow \frg$ in Construction \ref{construction} is flat, in the sense that it satisfies the conclusion of Corollary \ref{corollaryflatlocalrings} (and hence Remark \ref{remarktrivialtorsor}).
\end{hypothesis}
The hypothesis above is satisfied for any smooth map $\frY\rightarrow \frX_2$ by Corollary \ref{corollaryflatlocalrings}. Such smooth maps exist and we know that for any point in $\frX_2(L')$ where $L'$ is a field finite over $L$, there exists an open neighborhood of the point in $\frX_2$ which receives a smooth surjection from a rigid analytic space by \cite[Thm. 5.1.20]{emerton2023introduction}.
\begin{proposition}\label{propositiondirectimagephigammamodule}
    Assume Hypothesis \ref{hypothesisflat} and let $\bh=(h_1,h_2)$ such that $ k=h_2-h_1\in\Z_{\geq 1}$. Then $ Rf_{\bh,*}^r\cD_{\widetilde{\frY}^{\wedge}}^r$ concentrates in degree zero and is a locally finite projective module of rank $4$ over $\cO_{\bbU_{\frY^{\wedge}}^r}$. 
\end{proposition}
\begin{proof}
    By a twist, we assume that $h_1=0$. Then by construction, there are inclusions of vector bundles $t^k\Delta^r_{\widetilde{\frY}_n}\hookrightarrow \cD^r_{\widetilde{\frY}_n}\hookrightarrow \Delta^r_{\widetilde{\frY}_n}$. 
    
    Note that $\Delta^r_{\widetilde{\frY}_n}/t^k\Delta^r_{\widetilde{\frY}_n}=\prod_{m\geq m(r)} D_{\dif}^{m,+}(\Delta_{\widetilde{\frY}_n})/t^k$ (cf. \cite[Prop. 2.15]{liu2015triangulation}) and taking inverse limit $\Delta^r_{\widetilde{\frY}^{\wedge}}/t^k\Delta^r_{\widetilde{\frY}^{\wedge}}=\prod_{m\geq m(r)} D_{\dif}^{m,+}(\Delta_{\widetilde{\frY}^{\wedge}})/t^k$. The latter can be viewed as a coherent $\cO_{\bbU^r_{\widetilde{\frY}^{\wedge}}}$-module supported on disjoint divisors cut out by $Q_{m}(X)^k$ for $m\geq m(r)$, where $t=X\prod_{m\geq 1}Q_{m}(X)$, see Appendix \ref{subsectionbeauvillelaszlo}. Write $\bbU^{\wedge,m}$ for the completion of $\bbU^{r}$ along locus $\Sp(L\otimes_{\Q_p}K_{m})$ cut out by $Q_{m}(X)$. The sheaf $D_{\dif}^{m,+}(\Delta_{\widetilde{\frY}^{\wedge}})$ is exactly the $Q_{m}(X)$-adic completion of $\Delta_{\widetilde{\frY}^{\wedge}}^r$ and we can identify $D_{\dif}^{m,+}(\Delta_{\widetilde{\frY}^{\wedge}})/t$ as the pullback of $\Delta_{\widetilde{\frY}^{\wedge}}$ to $\widetilde{\frY}^{\wedge}\times_L\Sp(L\otimes_{\Q_p}K_{m})\subset \widetilde{\frY}^{\wedge}\times_L\bbU^r$. 

    On $\widetilde{\frY}^{\wedge}$ the rank two projective $\cO_{\widetilde{\frY}^{\wedge}}$-module $D_{\pdR}:=D_{\pdR}(\Delta_{\widetilde{\frY}^{\wedge}})$ is equipped with a universal submodule $\Fil^{0}$ stabilized by $\nu_{\widetilde{\frY}^{\wedge}}$ (now $\nu_{\widetilde{\frY}^{\wedge}}$ is only topologically nilpotent). Set the decreasing filtration $\Fil^{\bullet}$ on $D_{\pdR}(\Delta_{\widetilde{\frY}^{\wedge}})$ by $\Fil^{-k}=D_{\pdR}\supsetneq \Fil^{-k+1}=\cdots=\Fil^0\supsetneq\Fil^1=\{0\}$. Under the equivalence in Proposition \ref{propositionappendixequivalencealmostdeRham}, for each $m$, the filtration $\Fil^{\bullet}$ gives the projective sub-$\cO_{\widetilde{\frY}^{\wedge}\times_L\bbU^{\wedge,m}}$-module $D_{\dif}^{m,+}(D_{\widetilde{\frY}^{\wedge}})$, a modification of $D_{\dif}^{m,+}(\Delta_{\widetilde{\frY}^{\wedge}})$. By the identification in Step 1 of the proof of Proposition \ref{propositionappendixequivalencealmostdeRham} and the assumption that $m(r)$ is large enough (which holds modulo $I^n$ and after taking inverse limit as well), as $\frS_{\widetilde{\frY}^{\wedge},m}:=\cO_{\widetilde{\frY}^{\wedge}\times_L\bbU^{\wedge,m}}=(\cO_{\widetilde{\frY}^{\wedge}}\times_{\Q_p}K_m)[[t]]$-modules
    \[D_{\dif}^{m,+}(D_{\widetilde{\frY}^{\wedge}})=D_{\pdR}\otimes_{\cO_{\widetilde{\frY}^{\wedge}}}t^k\frS_{\widetilde{\frY}^{\wedge},m}+\Fil^0\otimes_{\cO_{\widetilde{\frY}^{\wedge}}}\frS_{\widetilde{\frY}^{\wedge},m}\]
    and 
    \begin{equation}\label{equationdirectimage1}
        D_{\dif}^{m,+}(\Delta_{\widetilde{\frY}^{\wedge}})=D_{\pdR}\otimes_{\cO_{\widetilde{\frY}^{\wedge}}}\frS_{\widetilde{\frY}^{\wedge},m},D_{\dif}^{m,+}(t^k\Delta_{\widetilde{\frY}^{\wedge}})=D_{\pdR}\otimes_{\cO_{\widetilde{\frY}^{\wedge}}}t^k\frS_{\widetilde{\frY}^{\wedge},m}.
    \end{equation}
    Since the sequence $0\rightarrow \Fil^0\rightarrow D_{\pdR}\rightarrow D_{\pdR}/\Fil^0\rightarrow 0$ splits locally as $\cO_{\widetilde{\frY}^{\wedge}}$-modules, the injection
    \begin{equation}\label{equationdifferencelattice}
        \Fil^{0}\otimes_{\cO_{\widetilde{\frY}^{\wedge}}}(\cO_{\widetilde{\frY}^{\wedge}}\otimes _{\Q_p}K_m)[[t]]/t^k\rightarrow D_{\dif}^{m,+}(D_{\widetilde{\frY}^{\wedge}})/t^kD_{\dif}^{m,+}(\Delta_{\widetilde{\frY}^{\wedge}})     
    \end{equation}
    is an isomorphism.

    To show that $Rf_{\bh,*}^r\cD_{\widetilde{\frY}^{\wedge}}^r$ is locally free of rank $4$, we can work locally on $\bbU^r_{\frY^{\wedge}}$.  Take $s<r$. Write $\cD_{\widetilde{\frY}^{\wedge}}^{[s,r]},\Delta_{\widetilde{\frY}^{\wedge}}^{[s,r]}$ for their restriction to $\bbU^{[s,r]}_{\widetilde{\frY}^{\wedge}}$ and $f_{\bh}^{[s,r]}: \bbU^{[s,r]}_{\widetilde{\frY}^{\wedge}}\rightarrow \bbU^{[s,r]}_{\frY^{\wedge}}$. It's enough to show $Rf_{\bh,*}^{[s,r]}\cD_{\widetilde{\frY}^{\wedge}}^{[s,r]}$ is projective of rank $4$ over $\cO_{\bbU^{[s,r]}_{\frY^{\wedge}}}$ for all $r,s$ such that $m(s)=m(r)+1$. We may assume that the torsor $\frY^{\wedge,\square}\rightarrow \frY^{\wedge}$ is trivial as in Remark \ref{remarktrivialtorsor}. Choose an affinoid open $U=\Sp(B)\subset \GL_2$ and consider the restriction $f_{\bh,U}^{[s,r]}:\bbU^{[s,r]}_{\widetilde{\frY}^{\wedge}\times U}\rightarrow \bbU^{[s,r]}_{\frY^{\wedge}\times U}$. 
    We get the diagram
    \begin{center}
        \begin{tikzcd}
            &\bbU^{[s,r]}_{\widetilde{\frY}^{\wedge}\times U}\arrow[d,"f_{\bh,U}^{[s,r]}"]\arrow[ld,"\widetilde{\alpha}"]\arrow[rd,"\widetilde{\beta}"]& \\
            \bbU^{[s,r]}_{\widetilde{\frY}^{\wedge}}\arrow[d,"f_{\bh}^{[s,r]}"]&\bbU^{[s,r]}_{\frY^{\wedge}\times U}\arrow[ld,"\alpha"]\arrow[rd,"\beta"]& \widetilde{\frg} \arrow[d,"f"]\\
            \bbU^{[s,r]}_{\frY^{\wedge}}& &\frg
        \end{tikzcd}
    \end{center}
    where the parallelograms are Cartesian (modulo $I^n$).

    The formal rigid space $\bbU^{[s,r]}_{\frY^{\wedge}}$ (resp. $\bbU^{[s,r]}_{\frY^{\wedge}\times U}$) is the completion of an affinoid space and $\bbU^{[s,r]}_{\widetilde{\frY}^{\wedge}}$ (resp. $\bbU^{[s,r]}_{\widetilde{\frY}^{\wedge}\times U}$) comes from a projective scheme over $\Spec(\cO(\bbU^{[s,r]}_{\frY^{\wedge}}))$ (resp. $\Spec(\cO(\bbU^{[s,r]}_{\frY^{\wedge}\times U}))$) in the sense before Corollary \ref{corollaryGAGA}. Hence $R^if_{\bh,*}^{[s,r]}\cD_{\widetilde{\frY}^{\wedge}}^{[s,r]}$ is a coherent $\cO_{\bbU^{[s,r]}_{\frY^{\wedge}}}$-module for all $i\geq 0$ and can be computed on the scheme level. The same statement holds similarly for $R^if_{\bh,U,*}^{[s,r]}\widetilde{\alpha}^*\cD_{\widetilde{\frY}^{\wedge}}^{[s,r]}$. Also the maps $\bbU^{[s,r]}_{\frY_n\times U}\rightarrow \frg\rightarrow \frg^{\alg}$ in Remark \ref{remarktrivialtorsor} factor through $\Spec(\cO(\bbU^{[s,r]}_{\frY_n\times U}))\rightarrow \frg^{\alg}$. Then $\alpha^*R^if_{\bh,*}^{[s,r]}\cD_{\widetilde{\frY}^{\wedge}}^{[s,r]}=R^if_{\bh,U,*}^{[s,r]}\widetilde{\alpha}^*\cD_{\widetilde{\frY}^{\wedge}}$ by the flat base change \cite[\href{https://stacks.math.columbia.edu/tag/02KH}{Tag 02KH}]{stacks-project} and $R^if_{\bh,*}^{[s,r]}\cD_{\widetilde{\frY}^{\wedge}}^{[s,r]}$ is projective of rank $4$ if and only if so is $R^if_{\bh,U,*}^{[s,r]}\widetilde{\alpha}^*\cD_{\widetilde{\frY}^{\wedge}}^{[s,r]}=R^if_{\bh,U,*}^{[s,r]}\cD_{\widetilde{\frY}^{\wedge}\times U}^{[s,r]}$ by the faithfully flat descent \cite[\href{https://stacks.math.columbia.edu/tag/058S}{Tag 058S}]{stacks-project} and Lemma \ref{lemmafaithfullyflatcompletion} below. 
    
    Consider the short exact sequences of coherent $\cO_{\bbU^{[s,r]}_{\tildefrY^{\wedge}\times U}}$-modules from the first part of the proof (e.g., (\ref{equationdifferencelattice}) and $m=m(r)$)
    \begin{align*}
        0\rightarrow t^k\Delta_{\tildefrY^{\wedge}\times U}^{[s,r]}\rightarrow \cD_{\tildefrY^{\wedge}\times U}^{[s,r]}\rightarrow \Fil^0\otimes_{\cO_{\widetilde{\frY}^{\wedge}}}(\cO_{\widetilde{\frY}^{\wedge}\times U}\otimes _{\Q_p}K_m)[[t]]/t^k\rightarrow 0,\\
        0\rightarrow \cD_{\tildefrY^{\wedge}\times U}^{[s,r]}\rightarrow \Delta_{\tildefrY^{\wedge}\times U}^{[s,r]}\rightarrow D_{\pdR}/\Fil^0\otimes_{\cO_{\widetilde{\frY}^{\wedge}}}(\cO_{\widetilde{\frY}^{\wedge}\times U}\otimes _{\Q_p}K_m)[[t]]/t^k\rightarrow 0.
    \end{align*}
    By our construction of $\widetilde{\frY}^{\wedge,\square}$, the pullback $D_{\pdR}\otimes_{\cO_{\widetilde{\frY}^{\wedge}}}\cO_{\widetilde{\frY}^{\wedge}\times U}$ admits a tautological trivialization $D_{\pdR}\otimes_{\cO_{\widetilde{\frY}^{\wedge}}}\cO_{\widetilde{\frY}^{\wedge}\times U}\simeq \cO_{\widetilde{\frY}^{\wedge}\times U}^{\oplus 2}$ and the subsheaf $\Fil^0\otimes_{\cO_{\widetilde{\frY}^{\wedge}}}\cO_{\widetilde{\frY}^{\wedge}\times U}=\widetilde{D}_{\pdR}^*\cO_{\tildefrg}(-1)$ is the tautological subbundle pulled back from $\widetilde{\frg}\rightarrow G/B=\bP^1$ where we write $D_{\pdR}$ for the flat map $\Spec(\cO(\frY^{\wedge}\times U))\rightarrow \frg$ (induced by the nilpotent operator $\nu$) and $\widetilde{D}_{\pdR}$ for the base change of $D_{\pdR}$.

    Let $f_{\bh,U}:\tildefrY^{\wedge}\times U\rightarrow \frY^{\wedge}\times U$ be the base change of $f_{\bh}$. The sheaf $\Fil^0\otimes_{\cO_{\widetilde{\frY}^{\wedge}}}(\cO_{\widetilde{\frY}^{\wedge}\times U}\otimes _{\Q_p}K_m)$ is supported on the closed subscheme of $\bbU^{[s,r]}_{\tildefrY^{\wedge}\times U}$ cut out by $Q_{m}(X)$ and 
    \[Rf_{\bh,U,*}^{[s,r]}\Fil^0\otimes_{\cO_{\widetilde{\frY}^{\wedge}}}(\cO_{\widetilde{\frY}^{\wedge}\times U}\otimes _{\Q_p}K_m)[[t]]/t^k=(Rf_{\bh,U,*}\widetilde{D}_{\pdR}^*\cO_{\tildefrg}(-1))\otimes_{\Q_p}K_m[[t]]/t^k\] 
    as a coherent sheaf supported on the closed subscheme of $\bbU^{[s,r]}_{\frY^{\wedge}\times U}$ cut out by $Q_{m}(X)^k$. By the flat base change and (1) of Proposition \ref{propositiondirectimagelinebundle}, $Rf_{\bh,U,*}\widetilde{D}_{\pdR}^*\cO_{\tildefrg}(\pm 1)=D_{\pdR}^*Rf_{*}\cO_{\tildefrg}(\pm 1)$ concentrate in degree zero and are locally free of rank two over $\frY^{\wedge}\times U$. 

    By the projection formula, $Rf^{[s,r]}_{\bh,U,*}\Delta^{[s,r]}_{\widetilde{\frY}^{\wedge}\times U}=\Delta_{\frY^{\wedge}\times U}^{[s,r]}\otimes_{\cO_{\bbU^{[s,r]}_{\frY^{\wedge}\times U}}}Rf^{[s,r]}_{\bh,U,*}Rf^{[s,r],*}_{\bh,U}\cO_{\bbU^{[s,r]}_{\frY^{\wedge}\times U}}$. Since $\beta$ is flat by the assumption that the map $\frY^{\wedge,\square}\rightarrow\frg$ is flat, $Rf^{[s,r]}_{\bh,U,*}Rf^{[s,r],*}_{\bh,U}\cO_{\bbU^{[s,r]}_{\frY^{\wedge}\times U}}=\beta^*Rf_{*}Rf^*\cO_{\frg}=\beta^*f_{*}\cO_{\tildefrg}$ is free of rank two over $\cO_{\bbU^{[s,r]}_{\frY^{\wedge}\times U}}$. Then we get short exact sequences
    \begin{align}\label{equationdirectimage2}
        0\rightarrow f^{[s,r]}_{\bh,U,*} t^k\Delta_{\tildefrY^{\wedge}\times U}^{[s,r]}\rightarrow f^{[s,r]}_{\bh,U,*}\cD_{\tildefrY^{\wedge}\times U}^{[s,r]}\rightarrow D_{\pdR}^*f_{*}\cO_{\tildefrg}(-1)\otimes _{\Q_p}K_m[[t]]/t^k\rightarrow 0,\\
        0\rightarrow f^{[s,r]}_{\bh,U,*} \cD_{\tildefrY^{\wedge}\times U}^{[s,r]}\rightarrow f^{[s,r]}_{\bh,U,*} \Delta_{\tildefrY^{\wedge}\times U}^{[s,r]}\rightarrow D_{\pdR}^*f_{*}\cO_{\tildefrg}(1)\otimes_{\Q_p}K_m[[t]]/t^k\rightarrow 0\nonumber.
    \end{align}
    Finally to see that $f^{[s,r]}_{\bh,U,*}\cD_{\tildefrY^{\wedge}\times U}^{[s,r]}$ is locally free of rank $4$ over $\cO_{\bbU^{[s,r]}_{\frY^{\wedge}\times U}}$, by \cite{beauville1995lemme}, it is enough to consider its completion along the divisor $(\widetilde{\frY}^{\wedge}\times_L U)\times_{L}\Sp(L\otimes_{\Q_p}K_m)$ cut out by $Q_{m}(X)$, which is a $(\cO_{\widetilde{\frY}^{\wedge}\times U}\times_{\Q_p}K_m)[[t]]$-module. Then the result follows from Lemma \ref{lemmaprojectivelattice}, that $f^{[s,r]}_{\bh,U,*} \Delta_{\tildefrY^{\wedge}\times U}^{[s,r]}$ is free of rank $4$ over $\cO_{\bbU^{[s,r]}_{\frY^{\wedge}\times U}}$ and that $D_{\pdR}^*f_{*}\cO_{\tildefrg}(1)\otimes_{\Q_p}K_m[[t]]/t^k$ is finite flat over $\cO_{\frY^{\wedge}\times U}\times_{\Q_p}K_m[[t]]/t^k$.
\end{proof}
\begin{lemma}\label{lemmafaithfullyflatcompletion}
    Let $g:X=\Sp(B)\rightarrow Y=\Sp(A)$ be a flat (resp. faithfully flat) morphism of affinoid spaces and let $I\subset A$ be an ideal. Let $A^{\wedge},B^{\wedge}$ be the $I$-adic completions. Then the ring maps $A\rightarrow B$ and $A^{\wedge}\rightarrow B^{\wedge}$ are flat (resp. faithfully flat). 
\end{lemma}
\begin{proof}
    By definition, for any $x\in X$ and $f(x)\in Y$ corresponding to maximal ideals $\frm\subset B,\frn\subset A$, the local ring map $\cO_{Y,f(x)}\rightarrow \cO_{X,y}$ is flat. Since both rings are Noetherian, flatness can be checked after completion and thus also on the Zariski local rings $ A_{\frn}\rightarrow B_{\frm}$. By \cite[\href{https://stacks.math.columbia.edu/tag/00HT}{Tag 00HT}]{stacks-project}, $B$ is flat over $A$. The assertion about flatness after completion is by Lemma \ref{lemmaflatcompletion}. For faithfully flatness, if $\Spec(B^{\wedge})\rightarrow \Spec(A^{\wedge})$ is surjective on closed points and flat, it is surjective by \cite[\href{https://stacks.math.columbia.edu/tag/00HS}{Tag 00HS}]{stacks-project}. As $I$ is topologically nilpotent in $A^{\wedge}$ and $A^{\wedge}$ is $I$-adically complete, $I$ is in the Jacobson radical of $A^{\wedge}$. The surjectivity on closed points follows from the surjectivity of $\Spec(B/I)\rightarrow \Spec(A/I)$.
\end{proof}

Over $\tildefrY^{\wedge}$ the topologically nilpotent operator $\nu_{\frY^{\wedge}}$ acts on the universal sub-line bundle $\Fil^0$ of $D_{\pdR}(\Delta_{\tildefrY^{\wedge}})$ and the quotient $D_{\pdR}(\Delta_{\tildefrY^{\wedge}})/\Fil^0$. This gives elements $z_{\tildefrY^{\wedge}}+h_{\tildefrY^{\wedge}}\in \cO(\tildefrY^{\wedge})\simeq \End_{\cO_{\tildefrY^{\wedge}}}(\Fil^0)$ and $z_{\tildefrY^{\wedge}}-h_{\tildefrY^{\wedge}}\in \cO(\tildefrY^{\wedge})\simeq \End_{\cO_{\tildefrY^{\wedge}}}(D_{\pdR}(\Delta_{\tildefrY^{\wedge}})/\Fil^0)$. We view $z_{\tildefrY^{\wedge}}$ and $h_{\tildefrY^{\wedge}}$ as global sections of $f_{\bh,*}\cO_{\tildefrY^{\wedge}}$ (cf. Corollary \ref{corollaryGAGA}). As a corollary of the proof of Proposition \ref{propositiondirectimagephigammamodule} and (3) of Proposition \ref{propositiondirectimagelinebundle}, we get the following explicit description of $Rf_{\bh,*}D^r_{\widetilde{\frY}^{\wedge}}$ when $\bh=(0,k)$.

\begin{corollary}\label{corollarydirectimagephigammamodule}
    Suppose $\bh=(0,k), k\geq 1$ and assume Hypothesis \ref{hypothesisflat}. Then $Rf_{\bh,*}\Delta^{r}_{\tildefrY^{\wedge}}=\Delta^{r}_{\frY^{\wedge}}\otimes_{\cO_{\frY^{\wedge}}}f_{\bh,*}\cO_{\tildefrY^{\wedge}}$ is a $(\varphi,\Gamma)$-module of rank $4$ over $\cR^r_{\frY^{\wedge}}$. And $Rf_{\bh,*}D^{r}_{\tildefrY^{\wedge}}$ is the rank $4$ sub-$(\varphi,\Gamma)$-module of $f_{\bh,*}\Delta^{r}_{\tildefrY^{\wedge}}$ containing $t^kf_{\bh,*}\Delta^{r}_{\tildefrY^{\wedge}}$ determined by
    \begin{equation}\label{equationCorollarydirectimagephigamamodule0}
        f_{\bh,*}D^{r}_{\tildefrY^{\wedge}}/t^kf_{\bh,*}\Delta^{r}_{\tildefrY^{\wedge}}=\oplus_{i=0}^{k-1}(f_{\bh,*}\Delta^{r}_{\tildefrY^{\wedge}}/t^kf_{\bh,*}\Delta^{r}_{\tildefrY^{\wedge}})[\nabla_{\Sen}=(z_{\tildefrY^{\wedge}}+h_{\tildefrY^{\wedge}})+i].
    \end{equation}
\end{corollary}
\begin{proof}
    We write $\widetilde{\Delta}_{\frY^{\wedge}}^r$ for $\Delta^{r}_{\frY^{\wedge}}\otimes_{\cO_{\frY^{\wedge}}}f_{\bh,*}\cO_{\tildefrY^{\wedge}}$. By Proposition \ref{propositiondirectimagephigammamodule}, the map $f_{\bh,*}D^{r}_{\tildefrY^{\wedge}}\rightarrow \widetilde{\Delta}_{\frY^{\wedge}}^r$ induced by $D^{r}_{\tildefrY^{\wedge}}\subset \Delta_{\tildefrY^\wedge}^r$ is injective and the image contains $t^k\widetilde{\Delta}_{\frY^{\wedge}}^r$. Hence it suffices to determine the image of 
    \begin{equation}\label{equationCorollarydirectimagephigamamodule1}
        f_{\bh,*}D^{r}_{\tildefrY^{\wedge}}/t^k\widetilde{\Delta}_{\frY^{\wedge}}^r\hookrightarrow \widetilde{\Delta}_{\frY^{\wedge}}^r/t^k\widetilde{\Delta}_{\frY^{\wedge}}^r=\prod_{m\geq m(r)}D_{\dif}^{m,+}(\widetilde{\Delta}_{\frY^{\wedge}})/t^k.
    \end{equation} 
    The operator $\nabla_{\Sen}$ acts on each $D_{\dif}^{m,+}(\widetilde{\Delta}_{\frY^{\wedge}})/t^k$. Under the identification (\ref{equationdirectimage1}), $D_{\pdR}(\widetilde{\Delta}_{\frY^{\wedge}})\otimes_{\Q_p}K_m[[t]]/t^k=(D_{\pdR}(\Delta_{\frY^{\wedge}})\otimes_{\Q_p}K_m[[t]]/t^k)\otimes_{\cO_{\frY^{\wedge}}}f_{\bh,*}\cO_{\tildefrY^{\wedge}}$, and the Sen operator $\nabla_{\Sen}$ corresponds to the topologically nilpotent operator $\nu_{\frY^{\wedge}}$ on $D_{\pdR}(\Delta_{\frY^{\wedge}})$ where we extend the action of $\nu_{\frY^{\wedge}}$ on $v\otimes g\in D_{\pdR}(\Delta_{\frY^{\wedge}})\otimes_{A^{\wedge}}(A^{\wedge}\otimes_{\Q_p}K_m)[[t]]$ by $\nu_{\frY^{\wedge}}(v\otimes g)=\nu_{\frY^{\wedge}}(v)\otimes g+v\otimes \nabla(g)$ (see Step 1 of the proof of Proposition \ref{propositionappendixequivalencealmostdeRham}, and note $\nabla_{\Sen}(tx)=t(\nabla_{\Sen}+1)(x)$). The identity (\ref{equationCorollarydirectimagephigamamodule0}) will follow from that the image of (\ref{equationCorollarydirectimagephigamamodule1}) is equal to the submodule
    \[ \prod_{m\geq m(r)}\oplus_{i=0}^{k-1}(D_{\pdR}(\widetilde{\Delta}_{\frY^{\wedge}})\otimes_{\Q_p}K_m[[t]]/t^k)[\nu_{\frY^{\wedge}}=(z_{\tildefrY^{\wedge}}+h_{\tildefrY^{\wedge}}+i)].\]
 
    We view objects appeared in (\ref{equationCorollarydirectimagephigamamodule1}) as coherent sheaves on $\bbU^r_{\frY^{\wedge}}$ supported on divisors cut out by $Q_{m}(X)$ for $m\geq m(r)$. By faithfully flat descent, we only need to verify the equality after base change to $\frY^{\wedge}\times U$ as in the proof of Proposition \ref{propositiondirectimagephigammamodule} and we adapt the notation there. By (\ref{equationdirectimage2}), we see
    \[f_{\bh,U,*}D^{r}_{\tildefrY^{\wedge}\times U}/t^k\widetilde{\Delta}_{\frY^{\wedge}\times U}^r=
    \prod_{m\geq m(r)}D_{\pdR}^*f_{*}\cO_{\tildefrg}(-1)\otimes _{\Q_p}K_m[[t]]/t^k.\]
    By (3) of Proposition \ref{propositiondirectimagelinebundle}, the righthand side is the sheaf
    \[\prod_{m\geq m(r)}D_{\pdR}^*((f_*f^*(\cO_{\frg}^{\oplus 2}))[\nu=(h+z)])\otimes_{\Q_p}K_m[[t]]/t^k.\]
    Since $D_{\pdR}$ is flat, $D_{\pdR}^*f_*f^*(\cO_{\frg}^{\oplus 2})=f_{\bh,U,*}f_{\bh,U}^*D_{\pdR}(\Delta_{\frY^{\wedge}\times U})=D_{\pdR}(\widetilde{\Delta}_{\frY^{\wedge}\times U})$ (under the canonical trivialization $D_{\pdR}(\Delta_{\frY^{\wedge}\times U})\simeq D_{\pdR}^*\cO_{\frg}^{\oplus 2}$) and that $\nu$, $z$, and $h$ are pulled back to $\nu_{\frY^{\wedge}\times U}$, $z_{\tildefrY^{\wedge}}$, and $h_{\tildefrY^{\wedge}}$ respectively, we get
    \begin{align*}
        f_{\bh,U,*}D^{r}_{\tildefrY^{\wedge}\times U}/t^k\widetilde{\Delta}_{\frY^{\wedge}\times U}^r&=\prod_{m\geq m(r)}D_{\pdR}(\widetilde{\Delta}_{\frY^{\wedge}\times U})[\nu_{\frY^{\wedge}\times U}=(h_{\tildefrY^{\wedge}}+z_{\tildefrY^{\wedge}})]\otimes_{\Q_p}K_m[[t]]/t^k\\
        &=\prod_{m\geq m(r)}\oplus_{i=0}^{k-1}(t^iD_{\pdR}(\widetilde{\Delta}_{\frY^{\wedge}\times U})\otimes_{\Q_p}K_m)[\nu_{\frY^{\wedge}\times U}=(h_{\tildefrY^{\wedge}}+z_{\tildefrY^{\wedge}})+i]
    \end{align*}
    Hence the description (\ref{equationCorollarydirectimagephigamamodule0}) holds.
\end{proof}
We define translations of formal completions of $(\varphi,\Gamma)$-modules.
\begin{definition}
    Let $\frX^{\wedge}=\varinjlim_n\frX_n$ be the formal completion of a quasi-compact rigid space $\frX$ with respect to a coherent ideal sheaf $\cI$ (Definition \ref{definitionformalcompletion}). Let $D^r_{\frX^{\wedge}}$ be a $(\varphi,\Gamma,\frg)$-module over $\cR^r_{\frX^{\wedge}}$ (Definition \ref{definitionphigammamodulecomplete}). Assume that $D_{\frX_1}^r$ is locally $Z(\frg)$-finite (in the sense that there is a finite admissible covering of $\frX_1$ by open affinoids such that $D_{\frX_1}^{r}$ is locally $Z(\frg)$-finite when restricted to these affinoids). Let $\lambda,\mu\in X^*(\frt)$ be integral weights. We define the translation of $D^r_{\frX^{\wedge}}$ from the infinitesimal character associated to $\lambda$ to that of $\mu$ by 
    \[T_{\lambda}^{\mu} D^r_{\frX^{\wedge}}:=\varprojlim_n T_{\lambda}^{\mu}D^r_{\frX_n}.\]
    And if $D_{\frX^{\wedge}}=\cR_{\frX^{\wedge}}\otimes_{\cR^r_{\frX^{\wedge}}}D^r_{\frX^{\wedge}}$,
    \[T_{\lambda}^{\mu} D_{\frX^{\wedge}}:=\varinjlim_{r'\leq r}T_{\lambda}^{\mu} D^{r'}_{\frX^{\wedge}}.\]
    The translation is always a $(\varphi,\Gamma)$-module with a $\frg$-action (cf. Proposition \ref{proptranslationphigammamodule}). 
\end{definition}
Recall that if $\lambda-\mu$ is $(0,k)$ or $(k,0)$ for some $k\geq 0$, then $T_{\lambda}^{\mu}D_{\frX_n}^r=\pr_{|\mu|}(\pr_{|\lambda|}D_{\frX_n}^r\otimes_{L}V_k)$.
\begin{lemma}\label{lemmatranslationdirectimage}
    Let $f_{\bh}:\tildefrY^{\wedge}\rightarrow \frY^{\wedge}$ be the map in Construction \ref{construction}. Suppose that $D^{',r}_{\tildefrY^{\wedge}}$ is a $(\varphi,\Gamma,\frg)$-module over $\cR^{r}_{\tildefrY^{\wedge}}$ such that $Rf_{\bh,*}D^{',r}_{\tildefrY^{\wedge}}$ is a $(\varphi,\Gamma)$-module over $\frY^{\wedge}$ (cf. Lemma \ref{lemmadirectimagephibundles}). 
    \begin{enumerate}
        \item The natural isomorphism $Rf_{\bh,*}(D_{\tildefrY^{\wedge}}^{',r}\otimes_{L}V_k)\simeq (Rf_{\bh,*}D_{\tildefrY^{\wedge}}^{',r})\otimes_{L}V_k$ of $\cO_{\frY^{\wedge}}$-modules (provided by $V_k\simeq L^{\oplus (k+1)}$) is an isomorphism of $(\varphi,\Gamma,\frg)$-modules.
        \item Furthermore, if $D_{\tildefrY_1}^{',r}$ is locally $Z(\frg)$-finite with a generalized infinitesimal character given by $\lambda$, then the isomorphism in (1) induces an isomorphism $Rf_{\bh,*}T_{\lambda}^{\mu}D_{\tildefrY^{\wedge}}^{',r}\simeq T_{\lambda}^{\mu}Rf_{\bh,*}D_{\tildefrY^{\wedge}}^{',r}$.
    \end{enumerate}
\end{lemma}
\begin{proof}
    (1) We verify the isomorphism $f_{\bh,*}(D_{\tildefrY^{\wedge}}^{',r}\otimes_{L}V_k)\simeq (f_{\bh,*}D_{\tildefrY^{\wedge}}^{',r})\otimes_{L}V_k$ is an isomorphism of $(\varphi,\Gamma,\frg)$-module. Recall $V^k=\Sym^kL^2=\cR_L^+/X^{k+1}$. Then $D_{\tildefrY^{\wedge}}^{',r}\otimes_{L}V_k=\oplus_{i=0}^k D_{\tildefrY^{\wedge}}^{',r} \otimes_L X^iL$ and $f_{\bh,*}(D_{\tildefrY^{\wedge}}^{',r}\otimes_{L}V_k)=\oplus_{i=0}^k f_{\bh,*}D_{\tildefrY^{\wedge}}^{',r} \otimes_L X^iL\simeq f_{\bh,*}D_{\tildefrY^{\wedge}}^{',r} \otimes_L V_k$ with obvious maps. And the sheaf $f_{\bh,*}(D_{\tildefrY^{\wedge}}^{',r}\otimes_{L}V_k)$ is determined by its section over $\frY^{\wedge}$ as for $f_{\bh,*}D_{\tildefrY^{\wedge}}^{',r}$ (Lemma \ref{lemmacompletionphigammamodule}). For $g\in \frg$ and $\sum_{i=0}^{k}a_i\otimes X^i\in f_{\bh,*}D_{\tildefrY^{\wedge}}^{',r} \otimes_L V_k$, we have $g.(\sum_{i=0}^{k}a_i\otimes X^i)=\sum_{i=0}^{k}g.a_i\otimes X^i+\sum_{i=0}^ka_i\otimes g.X^i=\sum_{i=0}^k (g.a_i+\sum_{j}c_{ji}a_j)\otimes X^i$ where $g.X^i=\sum_{i,j}c_{ij}X^j$. The $\frg$ action on $\sum_{i=0}^{k}a_i\otimes X^i\in f_{\bh,*}(D_{\tildefrY^{\wedge}}^{',r}\otimes_{L}V_k)$ (resp. on $a_i\in f_{\bh,*}D_{\tildefrY^{\wedge}}^{',r}$) is given by the same formula viewing $\sum_{i=0}^{k}a_i\otimes X^i$ (resp. $a_i$) as global sections on $\tildefrY^{\wedge}$. Hence the $\frg$-actions coincide. Similar statements hold for the actions of $\varphi$ and $\Gamma$. It remains to show that the map is $\cR^r_{\frY^{\wedge}}$-linear. By Lemma \ref{lemmadirectimagephibundles}, we have $f_{\bh,*}(D_{\tildefrY^{\wedge}}^{',r}\otimes_{L}V_k)=\varprojlim_n f_{\bh,*}(D_{\tildefrY_n}^{',r}\otimes_{L}V_k)$ and $f_{\bh,*}D_{\tildefrY^{\wedge}}^{',r}=\varprojlim_n f_{\bh,*}D_{\tildefrY_n}^{',r}$ as sheaves of $\cR^r_{\frY^{\wedge}}$-module. Let $f=(f_n)_n\in \cR_{\frY^{\wedge}}^{r}=\varprojlim_n\cR_{\frY_n}^r$ act on $\varprojlim_n f_{\bh,*}D_{\tildefrY_n}^{',r}$. For $\sum_{i=0}^{k}a_i\otimes X^i\in f_{\bh,*}D_{\tildefrY^{\wedge}}^{',r} \otimes_L V_k$, write similarly $a_i=(a_{i,n})_n\in \varprojlim_n f_{\bh,*}D_{\tildefrY_n}^{',r}$. By the proof of Proposition \ref{proptensoralgebraic}, the action of $f_n$ on $a_{i,n}\otimes X^i$ is given by $f_n.(a_{i,n}\otimes X^i)=\sum_{j=0}^k\frac{1}{j!}(X+1)^jf^{(j)}_n(X)a_{i,n}\otimes X^{j+i}$. Same equation holds if we view $a_{i,n}\otimes X^i$ as a global section of $D_{\tildefrY_n}^{',r}\otimes_L V_k$ for the action of $f_n$ via $f_{\bh}^{-1}\cR^r_{\frY_n}\rightarrow \cR^r_{\tildefrY_n}$. Taking inverse limit we get the compatibility of $\cR^r_{\frY^{\wedge}}$-actions. 
    
    (2) If $D^{',r}_{\tildefrY_n}$ is locally $Z(\frg)$-finite, since $\tildefrY_n$ is quasi-compact, $R^if_{\bh,*}D^{',r}_{\tildefrY_n}$ is locally $Z(\frg)$-finite (one can choose a finite affinoid covering of $\tildefrY_n$ to calculate the cohomology). By taking inverse limit, we have a direct sum decomposition $D_{\tildefrY^{\wedge}}^{',r}\otimes_LV_k=\oplus_{\mu'} T_{\lambda}^{\mu'}D_{\tildefrY^{\wedge}}^{',r}$ for finitely many $\mu'$. We conclude by noticing that $Z(\frg)$ acts on $f_{\bh,*}T_{\lambda}^{\mu}D_{\tildefrY^{\wedge}}^{',r}=\varprojlim_n f_{\bh,*}T_{\lambda}^{\mu}D_{\tildefrY_n}^{',r}$ locally profinitely with generalized infinitesimal character $\mu$.
\end{proof}
\begin{theorem}\label{theorem}
    Let $\lambda=\lambda_{\bh}=(h_2-1,h_1)$ for $\bh=(h_1,h_2)\in \Z^2,h_1<h_2$ and $\mu=\lambda_{(0,0)}=(-1,0)$ be weights in $X^*(\frt)$. Let $f_{\bh}:\tildefrY^{\wedge}\rightarrow \frY^{\wedge}$ and $D_{\tildefrY^{\wedge}}^r, \Delta_{\frY^{\wedge}}^r$ be the map and $(\varphi,\Gamma)$-modules in Construction \ref{construction}. We equip $D_{\tildefrY^{\wedge}}^r$ and  $\Delta_{\frY^{\wedge}}^r$ with the standard $\frg$-module structures in Definition \ref{definitionstandardgmodulestructure}.
    \begin{enumerate}
        \item There is an isomorphism of $(\varphi,\Gamma,\frg)$-modules 
        \begin{equation}\label{equaitontheorem1}
            T_{\lambda}^{\mu}D_{\tildefrY^{\wedge}}^r\simeq f_{\bh}^*\Delta^r_{\frY^{\wedge}}:=\Delta^r_{\tildefrY^{\wedge}}=\varprojlim_n\Delta_{\tildefrY_n}^r
        \end{equation} 
        over $\cR_{\tildefrY^{\wedge}}^r$. 
        \item Under Hypothesis \ref{hypothesisflat}, the adjunction of $f_{\bh}^*T_{\mu}^{\lambda}\Delta^r_{\frY^{\wedge}}=T_{\mu}^{\lambda}\Delta^r_{\tildefrY^{\wedge}}\rightarrow D^r_{\tildefrY^{\wedge}}$ (as $(\varphi,\Gamma)$-bundles over $\bbU^r_{\tildefrY^{\wedge}}$) induces an isomorphism 
        \begin{equation}\label{equationtheorem2}
            T_{\mu}^{\lambda}\Delta^r_{\frY^{\wedge}}\simeq Rf_{\bh,*}D^r_{\tildefrY^{\wedge}}
        \end{equation}
        of $(\varphi,\Gamma,\frg)$-modules of rank $4$ over $\cR^r_{\frY^{\wedge}}$.
    \end{enumerate}
\end{theorem}
\begin{proof}
    (1) The statement follows from the construction of $D_{\tildefrY^{\wedge}}^r=\varprojlim_n D_{\tildefrY_n}^r, T_{\lambda}^{\mu}D_{\tildefrY^{\wedge}}^r=\varprojlim_n T_{\lambda}^{\mu}D_{\tildefrY_n}^r$ and Proposition \ref{propositionregulartononregular}.
    
    (2) We first assume that $\bh=(0,1)$. By Proposition \ref{propositiondirectimagephigammamodule} and Hypothesis \ref{hypothesisflat}, $Rf_{\bh,*}D^r_{\tildefrY^{\wedge}}=f_{\bh,*}D^r_{\tildefrY^{\wedge}}$. Write $\widetilde{\Delta}_{\frY^{\wedge}}^r$ for $\Delta^{r}_{\frY^{\wedge}}\otimes_{\cO_{\frY^{\wedge}}}f_{\bh,*}\cO_{\tildefrY^{\wedge}}$. The unit map for $(\varphi,\Gamma)$-bundles and the projection formula induce a map $T_{\mu}^{\lambda}\Delta_{\frY^{\wedge}}^r\rightarrow T_{\mu}^{\lambda}\Delta_{\frY^{\wedge}}^r\otimes_{\cO_{\frY^{\wedge}}} f_{\bh,*}f_{\bh}^*\cO_{\frY^{\wedge}}=T_{\mu}^{\lambda}\widetilde{\Delta}_{\frY^{\wedge}}^r=f_{\bh,*}T_{\mu}^{\lambda}\Delta_{\tildefrY^{\wedge}}^r$ (the last equation is by Lemma \ref{lemmatranslationdirectimage}). The isomorphism $\Delta^r_{\tildefrY^{\wedge}}\simeq T_{\lambda}^{\mu}D_{\tildefrY^{\wedge}}^r$ induces $T_{\mu}^{\lambda}\Delta^r_{\tildefrY^{\wedge}}\rightarrow D_{\tildefrY^{\wedge}}^r$ and hence $T_{\mu}^{\lambda}f_{\bh,*}\Delta^r_{\tildefrY^{\wedge}}\rightarrow f_{\bh,*}D_{\tildefrY^{\wedge}}^r$. The composite of the two maps gives the desired $\frg$-map 
    \[T_{\mu}^{\lambda}\Delta_{\frY^{\wedge}}^r\rightarrow f_{\bh,*}T_{\mu}^{\lambda}\Delta^r_{\tildefrY^{\wedge}}\rightarrow f_{\bh,*}D_{\tildefrY^{\wedge}}^r.\] 
    We show that this map is an isomorphism of $(\varphi,\Gamma)$-modules over $\cR_{\frY^{\wedge}}^r$.  
    
    By Hypothesis \ref{hypothesisflat}, flat base change and faithfully flat descent as in the proof of Proposition \ref{propositiondirectimagephigammamodule}, statements of Proposition \ref{propositiondirectimagelinebundle} hold replacing $f:\tildefrg\rightarrow \frg$ by $f_{\bh}:\tildefrY\rightarrow\frY$ after suitable modifications. In particular, $f_{\bh,*}\cO_{\tildefrY^{\wedge}}$ is locally free of rank two over $\cO_{\frY^{\wedge}}$ generated by the element $h_{\tildefrY^{\wedge}}$ defined before Corollary \ref{corollarydirectimagephigammamodule}. Write for short $z=z_{\tildefrY^{\wedge}}$ and $h=h_{\tildefrY^{\wedge}}$. The composite $T_{\mu}^{\lambda}\Delta^r_{\frY^{\wedge}}\rightarrow T_{\mu}^{\lambda}\widetilde{\Delta}^r_{\frY^{\wedge}}\stackrel{\frac{1}{4}\frc-h^2-h}{\rightarrow}T_{\mu}^{\lambda}\widetilde{\Delta}^r_{\frY^{\wedge}}\twoheadrightarrow \widetilde{\Delta}^r_{\frY^{\wedge}}$ sends $v_0\otimes 1+v_1\otimes t\in T_{\mu}^{\lambda}\Delta^r_{\frY^{\wedge}}$ to $(-\nabla+z-h)v_0+tv_1$ by the proof of Proposition \ref{propositionunitcounit} (all these maps are inverse limits of maps modulo $I^n$). This map is an injection ($\widetilde{\Delta}_{\frY^{\wedge}}=\Delta_{\frY^{\wedge}}\oplus h \Delta_{\frY^{\wedge}}$) and the image contains $t\widetilde{\Delta}_{\frY^{\wedge}}$ (for $thv_0\in th\Delta_{\frY^{\wedge}}$, let $v_1=\frac{(\nabla-z)tv_0}{t}$). Modulo $t\widetilde{\Delta}_{\frY^{\wedge}}^r$, the image of $T_{\mu}^{\lambda}\Delta_{\frY^{\wedge}}^r$ as a sub-$\cO_{\frY^{\wedge}}$-module of $\widetilde{\Delta}_{\frY^{\wedge}}^r/t=\prod_{m\geq m(r)}D^{m}_{\Sen}(\widetilde{\Delta}_{\frY^{\wedge}})\otimes_{\Q_p}K_m$ is $\prod_{m\geq m(r)}(\nabla-(z-h))D^{m}_{\Sen}(\Delta_{\frY^{\wedge}})\otimes_{\Q_p}K_m$. By (2) of Proposition \ref{propositiondirectimagelinebundle} and the flatness Hypothesis \ref{hypothesisflat}, this image is identified with $\prod_{m\geq m(r)}D^{m}_{\Sen}(\widetilde{\Delta}_{\frY^{\wedge}})[\nabla=(z+h)]\otimes_{\Q_p}K_m$, i.e., is equal to $f_{\bh,*}D_{\widetilde{\frY}^{\wedge}}^r/t\widetilde{\Delta}_{\frY^{\wedge}}^r$ by Corollary \ref{corollarydirectimagephigammamodule}. By Proposition \ref{propositionunitcounit} and taking direct image, the composite $f_{\bh,*}T_{\mu}^{\lambda}\Delta^r_{\tildefrY^{\wedge}}\rightarrow f_{\bh,*}D_{\tildefrY^{\wedge}}^r\hookrightarrow f_{\bh,*}\Delta_{\tildefrY^{\wedge}}^r$ is identified with $T_{\mu}^{\lambda}\widetilde{\Delta}^r_{\frY^{\wedge}}\stackrel{\frac{1}{4}\frc-h^2-h}{\rightarrow}T_{\mu}^{\lambda}\widetilde{\Delta}^r_{\frY^{\wedge}}\twoheadrightarrow \widetilde{\Delta}^r_{\frY^{\wedge}}$. Hence the composite $T_{\mu}^{\lambda}\Delta_{\frY^{\wedge}}^r\rightarrow f_{\bh,*}T_{\mu}^{\lambda}\Delta^r_{\tildefrY^{\wedge}}\rightarrow f_{\bh,*}D_{\tildefrY^{\wedge}}^r$ is an isomorphism, with the same image in $\widetilde{\Delta}^r_{\frY^{\wedge}}$ given by Corollary \ref{corollarydirectimagephigammamodule}. We have finished the case $\bh=(0,1)$.

    Finally, we treat general $\bh'=(0,k), k\geq 1$. Write $\lambda'=\lambda_{\bh'}$. We have simply $f_{\bh}=f_{\bh'}:\tildefrY^{\wedge}\rightarrow \frY^{\wedge}$. Write $D^{',r}_{\tildefrY^{\wedge}}$ for the universal $(\varphi,\Gamma)$-module over $\cR^r_{\tildefrY^{\wedge}}$ of weights $\bh'$ and write $D^{r}_{\tildefrY^{\wedge}}$ for the one with weights $\bh=(0,1)$. By Proposition \ref{propositionregulartononregular}, $D_{\tildefrY^{\wedge}}^{',r}\simeq T_{\lambda}^{\lambda'}D_{\tildefrY^{\wedge}}^{r}$. Then $f_{\bh',*}D_{\tildefrY^{\wedge}}^{',r}\simeq f_{\bh',*}T_{\lambda}^{\lambda'}D_{\tildefrY^{\wedge}}^{r}\simeq T_{\lambda}^{\lambda'}f_{\bh',*}D_{\tildefrY^{\wedge}}^{r}$ by Lemma \ref{lemmatranslationdirectimage}. Since $T_{\mu}^{\lambda}\Delta_{\frY^{\wedge}}^r\simeq f_{\bh,*}D_{\tildefrY^{\wedge}}^{r}$, we see $f_{\bh',*}D_{\tildefrY^{\wedge}}^{',r}\simeq T_{\lambda}^{\lambda'}T_{\mu}^{\lambda}\Delta_{\frY^{\wedge}}^r\simeq T_{\mu}^{\lambda'}\Delta_{\frY^{\wedge}}^r$ (Lemma \ref{lemmatranslationliealgebra}). Note that this isomorphism is induced by $T_{\mu}^{\lambda}\Delta_{\frY^{\wedge}}^r\otimes_L V_{k-1}\simeq f_{\bh,*}D_{\tildefrY^{\wedge}}^{r}\otimes_{L}V_{k-1}$ which, by the case for $\bh$, is induced by the adjunction of $T_{\mu}^{\lambda}\Delta_{\tildefrY^{\wedge}}^r\otimes_L V_{k-1}\rightarrow D_{\tildefrY^{\wedge}}^{r}\otimes_{L}V_{k-1}$. Taking $\pr_{|\lambda'|}$ everywhere we see the isomorphism $f_{\bh',*}D_{\tildefrY^{\wedge}}^{',r}\simeq T_{\mu}^{\lambda'}\Delta_{\frY^{\wedge}}^r$ is induced by $T_{\mu}^{\lambda'}\Delta_{\tildefrY^{\wedge}}^r\rightarrow D_{\tildefrY^{\wedge}}^{',r}$.
\end{proof}
Taking colimit for $r>0$ and considering Lemma \ref{lemmadirectimagephibundles}, we get the following corollary.
\begin{corollary}\label{corollary}
    Under Hypothesis \ref{hypothesisflat}, we have
    \[T_{\lambda}^{\mu}D_{\tildefrY^{\wedge}}\simeq f_{\bh}^*\Delta_{\frY^{\wedge}}\]
    and
    \[T_{\mu}^{\lambda}\Delta_{\frY^{\wedge}}\simeq Rf_{\bh,*}D_{\tildefrY^{\wedge}}.\]        
\end{corollary}
\subsection{Specialization to points}\label{subsectionspecialization}
We show how to use Theorem \ref{theorem} to recover some results of Ding. Suppose we are in the situation of Construction \ref{construction}. We fix an $L$-point $y$ of $\frY^{\wedge}$ and write $\Delta=\Delta_y$ for the specialization of $\Delta_{\frY^{\wedge}}$ at the point $y$. Let $i_y:y\hookrightarrow \frY_1\hookrightarrow\frY^{\wedge}$ be the closed embedding and consider the restriction $f_y:f_{\bh}^{-1}(y)\rightarrow y$ of $f_{\bh}$.
\begin{center}
    \begin{tikzcd}
        f_{\bh}^{-1}(y) \arrow[r,hook,"i'_y"] \arrow[d,"f_y"]
        & \tildefrY^{\wedge}\arrow[d,"f_{\bh}"]\\
        y\arrow[r, "i_y",hook]
        &  \frY^{\wedge}
    \end{tikzcd}
\end{center}

By the projection formula \cite[\href{https://stacks.math.columbia.edu/tag/08EU}{Tag 08EU}]{stacks-project} (and the arguments in the proof of Proposition \ref{propositiondirectimagephigammamodule} of reducing to scheme-theoretic cohomologies), we have
\begin{equation}\label{equationbasechange}
    Rf_{\bh,*}D^r_{\tildefrY^{\wedge}}\otimes_{\cO_{\frY}}^Li_{y,*}\cO_{y}\simeq Rf_{\bh,*}(D^r_{\tildefrY^{\wedge}}\otimes_{\cO_{\tildefrY^{\wedge}}}^LLf_{\bh}^*i_{y,*}\cO_{y}).  
\end{equation}
By Theorem \ref{theorem}, the left-hand side equals $T_{\mu}^{\lambda}\Delta^r_y$ assuming Hypothesis \ref{hypothesisflat}. 
\begin{lemma}\label{lemmaderivedfiber}
    Assume Hypothesis \ref{hypothesisflat}.
    \begin{enumerate}
        \item If $\Delta_y$ is de Rham, then $f_{\bh}^{-1}(y)=G/B=\bP^1$ is reduced and $Lf_{\bh}^*i_{y,*}\cO_{y}$ fits into an exact triangle in the derived category of coherent $
        \cO_{\tildefrY^{\wedge}}$-modules:
        \[i'_{y,*}\cO_{\bP^1}(-2)[1]\rightarrow Lf_{\bh}^*i_{y,*}\cO_{y}\rightarrow \cO_{\bP^1}[0]\rightarrow (i'_{y,*}\cO_{\bP^1}(-2)[1])[1]\] 
        where $\cO_{\bP^1}(-2)[1]$ denotes the line bundle $\cO_{\bP^1}(-2)$ sitting in cohomological degree $-1$.
        \item If $\Delta_y$ is not de Rham, then $f_{\bh}^{-1}(y)$ is a finite ramified cover of degree $2$ over $y$ and $Lf_{\bh}^*i_{y,*}\cO_{y}=f_{\bh}^*i_{y,*}\cO_{y}=i'_{y,*}\cO_{f_{\bh}^{-1}(y)}$.
    \end{enumerate}
\end{lemma}
\begin{proof}
    Consider the diagram (of schemes, to be lazy) below:
    \begin{center}
        \begin{tikzcd}
            f_{\bh}^{-1}(y) \arrow[r,hook,"i'_y"] \arrow[d,"f_y"]
            & \tildefrY^{\wedge}\arrow[d,"f_{\bh}"]
            & \tildefrY^{\wedge,\square}\arrow[d,"f_{\bh}^{\square}"]\arrow[l,"\widetilde{\alpha}"]\arrow[r,"\widetilde{\beta}"]& \widetilde{\frg}\arrow[d,"f"]\\ 
            y\arrow[r, "i_y",hook]
            &  \frY^{\wedge}&\frY^{\wedge,\square}\arrow[l,"\alpha"]\arrow[r,"\beta"]
            &\frg
        \end{tikzcd}
    \end{center}
    where all squares are Cartesian. Choose a lift $y^{\square}$ of $y$ in $\frY^{\wedge,\square}$ and let $\nu_{y}$ be the image of $y$ in $\frg$. If the nilpotent element $\nu_y\neq 0$, all vertical maps are finite flat of rank $2$ near $y,y^{\square}$ or $\nu_y$ by the statement on $\frg$ (Lemma \ref{lemmapropertygrothendieckresolutiongl2}) and we get (2). 
    
    From now on we assume $\nu_y=0$ and prove (1). Consider the embedding $\widetilde{\frY}^{\wedge,\square}\stackrel{j}{\rightarrow} H:=\frY^{\wedge,\square}\times G/B$ and still write $\beta:\frY^{\wedge,\square}\times G/B\rightarrow \frg\times G/B$. Let $\cI_0$ be the ideal sheaf for the regular closed embedding $\tildefrg\hookrightarrow \frg\times G/B$, which is locally free of rank one. Since $\beta$ is flat, $\cI:=\beta^{*}\cI_0$ is the ideal sheaf cutting out $\widetilde{\frY}^{\wedge,\square}$ from $\frY^{\wedge,\square}\times G/B$. Let $\cJ$ be the ideal sheaf for the closed embedding $\alpha^{-1}(y)\times G/B\hookrightarrow \frY^{\wedge,\square}\times G/B$. Then $\cI\subset \cJ$ as points on $\alpha^{-1}(y)$ are all de Rham ($\alpha^{-1}(y)\times G/B\hookrightarrow \tildefrY^{\wedge,\square}$). 
    \begin{center}
        \begin{tikzcd}
            f_{\bh}^{\square,-1}(\alpha^{-1}(y))\arrow[d,"f_{\bh}^{\square}"]=\alpha^{-1}(y)\times G/B\arrow[r]
            & \tildefrY^{\wedge,\square}\arrow[d,"f_{\bh}^{\square}"]\arrow[r,hook, "j"]&H:=\frY^{\wedge,\square} \times G/B \arrow[dl,"k"]\\ 
            \alpha^{-1}(y)\arrow[r, "i_{\alpha^{-1}(y)}"]&\frY^{\wedge,\square}
        \end{tikzcd}
    \end{center}
    In the notation of the above diagram, using that $k$ is smooth, we get 
    \[Lf^{\square,*}_{\bh}i_{\alpha^{-1}(y),*}\cO_{\alpha^{-1}(y)}=Lj^*k^*i_{\alpha^{-1}(y),*}\cO_{\alpha^{-1}(y)}=Lj^*\cO_{H}/\cJ.\]
    While
    \[Lj^*\cO_{H}/\cJ=\cO_{H}/\cI\otimes^L_{\cO_{H}}\cO_{H}/\cJ=[\cI\rightarrow \cO_H]\otimes_{\cO_H}\cO_{H}/\cJ=[\cI\otimes_{\cO_H}\cO_H/\cJ\rightarrow \cO_H/\cJ].\]
    Since $\cI\subset \cJ$, the map $\cI\otimes_{\cO_H}\cO_H/\cJ\rightarrow \cO_H/\cJ$ is zero. We calculate the restriction of the line bundle $\cI$ to $\alpha^{-1}(y)\times G/B$ whose structure sheaf is $\cO_{H}/\cJ$. The ideal $\cI_0$ on $\frg\times G/B$ corresponds to the condition that $(\nu,gB)\in\frg\times G/B$ such that $\Ad(g^{-1})(\nu)\in\frb.$ Write $\Ad(g^{-1})(\nu)=\mtwo{a}{b}{c}{d}$, then change $g$ to $gh$ for $h=\mtwo{x}{y}{}{x^{-1}}\in B$ sends a local generator $c\in \cI_0$ to $x^2c$. This means that the restriction of $\cI_0$ to $G/B$ is $\cO_{\bP^1}(-2)$ ($G\times^{B}\lambda$ for $\lambda=(1,-1)$). Thus $H^{-1}(Lf^{\square,*}_{\bh}i_{\alpha^{-1}(y),*}\cO_{\alpha^{-1}(y)})= \cO_{f_{\bh}^{\square,-1}(\alpha^{-1}(y))}(-2)$, $H^{0}(Lf^{\square,*}_{\bh}i_{\alpha^{-1}(y),*}\cO_{\alpha^{-1}(y)})=\cO_{f_{\bh}^{\square,-1}(\alpha^{-1}(y))}$ and $H^{i}(Lf^{\square,*}_{\bh}i_{\alpha^{-1}(y),*}\cO_{\alpha^{-1}(y)})=0$ for $i\neq 0,-1$. The statement of the exact triangle for $Lf_{\bh}^*i_{y,*}\cO_{y}$ follows from \cite[\href{https://stacks.math.columbia.edu/tag/08J5}{Tag 08J5}]{stacks-project} and descent.
\end{proof}
\begin{remark}
   Write $i_0:\{0\}\hookrightarrow \frg$ and $\bP^1=f^{-1}(0)$. With the proof of the lemma above we see there exists a short exact sequence $0\rightarrow \cO_{\{0\}}\rightarrow Li_0^*Rf_*\cO_{\tildefrg}\rightarrow \cO_{\{0\}}\rightarrow 0$. The calculation matches the fact that $Rf_*\cO_{\tildefrg}(s)$ is locally free of rank two for $s=0,\pm 1$ in \S \ref{sectinGrothendieckresolution}. 
\end{remark}
Write $D_{\bP^1}$ for the universal $(\varphi,\Gamma)$-module on $f_{\bh}^{-1}(y)$ (the restriction of $D_{\frY_1}$ to $f_{\bh}^{-1}(y)$).
\begin{lemma}\label{lemmaderhampushforward}
    If $\bh=(0,1)$, then $Rf_{y,*}(D^r_{\bP^1})=t\Delta^r$ and $Rf_{y,*}(D^r_{\bP^1}(-2)[1])=\Delta^r$.
\end{lemma}
\begin{proof}
    The proof goes as for Proposition \ref{propositiondirectimagephigammamodule} and Corollary \ref{corollarydirectimagephigammamodule}. We only do formal calculations here. Let $\Delta_{\bP^1}$ be the pullback of $\Delta$. The inclusions $t\Delta_{\bP^1}^r\subset D_{\bP^1}^r \subset \Delta_{\bP^1}^r$ gives short exact sequences 
    \[0\rightarrow t\Delta_{\bP^1}^r\rightarrow D_{\bP^1}^r\rightarrow D_{\bP^1}^r/t\Delta^r_{\bP^1}\rightarrow 0 \]
    and 
    \[ 0\rightarrow D_{\bP^1}^r\rightarrow \Delta_{\bP^1}^r\rightarrow D_{\bP^1}^r/\Delta_{\bP^1}^r\rightarrow 0.\]
    The inclusion $D_{\bP^1}^r/t\Delta^r_{\bP^1}\hookrightarrow \Delta^r_{\bP^1}/t\Delta^r_{\bP^1}=\cR_L^r/t\widehat{\otimes}_L \cO_{\bP^1}^2$ has image $\cR_L^r/t\widehat{\otimes}_L\cO_{\bP^1}(-1)$ and the quotient $\Delta^r_{\bP^1}/t\Delta^r_{\bP^1}\twoheadrightarrow \Delta^r_{\bP^1}/D^r_{\bP^1}$ corresponds to the quotient $\cR_L^r/t\widehat{\otimes}_L\cO_{\bP^1}^2\twoheadrightarrow \cR_L^r/t\widehat{\otimes}_L\cO_{\bP^1}(1)$. Hence we have
    \begin{align*}
        &Rf_{y,*}(D_{\bP^1}^r/t\Delta^r_{\bP^1})=Rf_{y,*}(\cR_L^r/t\widehat{\otimes}_L \cO_{\bP^1}(-1))=\cR_L^r/t\widehat{\otimes}_LRf_{y,*}\cO_{\bP^1}(-1)=0,\\
        &Rf_{y,*}(\Delta_{\bP^1}^r/D_{\bP^1}^r(-2))=Rf_{y,*}(\cR_L^r/t\widehat{\otimes}_L\cO_{\bP^1}(-1))=\cR_L^r/t\widehat{\otimes}_LRf_{y,*}\cO_{\bP^1}(-1)=0.
    \end{align*}
    Moreover,
    \begin{align*}
        &Rf_{y,*}t\Delta^r_{\bP^1}=t\Delta^r\widehat{\otimes}_L Rf_{y,*}\cO_{\bP^1}=t\Delta^r;\\
        &Rf_{y,*}(\Delta^r_{\bP^1}(-2)[1])=\Delta^r\widehat{\otimes}_LRf_{y,*}\cO_{\bP^1}(-2)[1]=\Delta^r.
    \end{align*} 
    The result follows.
\end{proof}

We recover part of \cite[Lem. 2.17]{ding2023change} below.
\begin{proposition}\label{propositiontranslationpointDing}
    Suppose that $\bh=(0,1)$ and assume Hypothesis \ref{hypothesisflat}.
    \begin{enumerate}
        \item If $\Delta=\Delta_y$ is de Rham, then $T_{\mu}^{\lambda}\Delta=i_y^*f_{\bh,*}D_{\tildefrY^{\wedge}}$ is an extension of $t\Delta$ by $\Delta$.
        \item If $\Delta=\Delta_y$ is not de Rham, then $T_{\mu}^{\lambda}\Delta=i_y^*f_{\bh,*}D_{\tildefrY^{\wedge}}$ is a self extension of $D_y$, where $D_y$ is the unique $(\varphi,\Gamma)$-module of rank two over $\cR_L$ of weight $\bh$ such that $T_{\lambda}^{\mu}D_y=\Delta$.
    \end{enumerate}
\end{proposition}
\begin{proof}
    Since $D^r_{\tildefrY^{\wedge}}$ is flat over $\cO_{\tildefrY^{\wedge}}$, there is an exact triangle 
    \[i'_{y,*}(D^r_{\bP^1}(-2)[1])\rightarrow D^r_{\tildefrY^{\wedge}}\otimes_{\cO_{\tildefrY^{\wedge}}}^LLf_{\bh}^*i_{y,*}\cO_{y}\rightarrow i'_{y,*}D^r_{\bP^1}\rightarrow (i'_{y,*}(D^r_{\bP^1}(-2)[1]))[1]\] 
    if $\Delta_y$ is de Rham and $D^r_{\tildefrY^{\wedge}}\otimes_{\cO_{\tildefrY^{\wedge}}}^LLf_{\bh}^*i_{y,*}\cO_{y}$ is equal to $D_{f_{\bh}^{-1}(y)}$ otherwise by Lemma \ref{lemmaderivedfiber}. By Lemma \ref{lemmaderhampushforward}, we have $Rf_{\bh,*}(D_{\tildefrY^{\wedge}}^r\otimes_{\cO_{\tildefrY^{\wedge}}}^LLf^{*}_{\bh}i_{y,*}\cO_{y})$ is an extension of $i_{y,*}Rf_{y,*} D_{\bP^1}^r=t\Delta^r$ by $i_{y,*}Rf_{y,*}D^r_{\bP^1}(-2)[1]= \Delta^r$ in the de Rham case.  Use (\ref{equationbasechange}) and take the direct limit over $r$, we get (1). (2) follows similarly using that $\cO_{f_{\bh}^{-1}(y)}\simeq L[h]/h^2$. Note that by Proposition \ref{propositionunitcounit} and the proof of Theorem \ref{theorem}, the Casimir $\frc$ acts by $4h^2-4h$ on $T_{\mu}^{\lambda}\Delta$ for some choice of $h$.
\end{proof}
\begin{remark}
    In general for $\bh=(0,k)$ and in the de Rham case, there is a filtration $t^k\Delta_{\bP^1}\subset D_{(k-1,k)}\subset \cdots\subset D_{(1,k)}\subset D_{(0,k)}=D_{\bP^1}$ of $(\varphi,\Gamma)$-modules of rank two where $D_{(i,k)}$ has Hodge-Tate weight $(i,k)$ by the proof of Proposition \ref{propositionregulartononregular}. The graded pieces $D_{(i,k)}/D_{(i+1,k)}$ are isomorphic to $\cO_{\bP^1}(-1)\widehat{\otimes}_L \cR_L/t$ as in Lemma \ref{lemmaderhampushforward} (see (\ref{equationdifferencelattice})) and we can get similarly $T_{\mu}^{\lambda}\Delta$ is an extension of $t^k\Delta$ by $\Delta$ as in \cite[Prop. 2.19 (4)]{ding2023change}.
\end{remark}

\subsection{Translation of $D\boxtimes\bP^1(\Q_p)$}
Let $A$ be an affinoid algebra over $L$ and $D_A$ be a $(\varphi,\Gamma)$-module over $\cR_A$ of rank $2$. Write $\omega=\omega_A$ be the character such that $\cR_{A}(\omega_A\epsilon)=\det(D_A)$. Pointwisely for $x\in\Sp(A)$, Colmez constructed a $\GL_2(\Q_p)$-representation (or a $\GL_2(\Q_p)$-equivariant sheaf on $\bP^1(\Q_p)$) $D_x\boxtimes_{\omega}\bP^1(\Q_p)$ (\cite{colmez2016representations,colmez2018poids,colmez2010representations}). It is expected that the construction can vary in family and obtain a $\GL_2(\Q_p)$-module $D_A\boxtimes_{\omega} \bP^1(\Q_p)$. We will not discuss the $\GL_2(\Q_p)$-module, but only construct a $U(\frg)$-module $D_A\boxtimes\bP^1(\Q_p)$. Recall $\bP^1(\Q_p)\setminus\Z_p=\Pi.\Z_p$ where $\Pi=\mtwo{}{1}{p}{}=\mtwo{}{1}{1}{}\mtwo{p}{}{}{1}\in\GL_2(\Q_p)$.

\begin{definition}
    Equip $D_A$ with the standard $\frg$-module structure in Lemma \ref{lemmastandardgstructure}. Define $D_A\boxtimes\Z_p=D_A$ to be the $U(\frg)\otimes_LA$-module $D_A$. Let $D_A\boxtimes(\bP^1(\Q_p)\setminus\Z_p)$ be the $U(\frg)\otimes_LA$-module $\Pi.D_A$ which has the underlying $A$-module $D_A$ and $\frg$ acts on $\Pi x,x\in D_A$ by $g.\Pi x=\Pi(\Ad(\Pi^{-1})(g).x)$. Define $D_A\boxtimes\bP^1(\Q_p)=D_A\boxtimes\Z_p\oplus D_A\boxtimes(\bP^1(\Q_p)\setminus\Z_p)$.
\end{definition}
\begin{remark}
    The character $\omega$ is not important for the $U(\frg)$-modules due to our definition. But it matters for $\GL_2(\Q_p)$-representations, namely for how $\mtwo{}{1}{1}{}$ acts on $D\boxtimes\Z_p^{\times}$.
\end{remark}
Hence $D_A\boxtimes\bP^1(\Q_p)$ is just two copies of $D_A$ with certain $\frg$-action. We can similarly define $D_A^r\boxtimes\bP^1(\Q_p)$ ($=D_A^r\oplus \mtwo{}{1}{1}{}\varphi(D_A^r)$). Let us go back to the situation in Construction \ref{construction}. Sheafify the above definition we obtain sheaves of $U(\frg)$-modules $D_{\widetilde{\frY}^{\wedge}}\boxtimes\bP^1(\Q_p)$ and $\Delta_{\frY^{\wedge}}\boxtimes\bP^1(\Q_p)$.  
\begin{corollary}\label{corollaryDboxtimesP1}
    Under Hypothesis \ref{hypothesisflat}, there are isomorphisms of sheaves of $U(\frg)$-modules 
    \[T_{\lambda}^{\mu}(D_{\tildefrY^{\wedge}}\boxtimes\bP^1(\Q_p))\simeq Rf_{\bh}^*(\Delta_{\frY^{\wedge}}\boxtimes\bP^1(\Q_p))\]
    and
    \[T_{\mu}^{\lambda}(\Delta_{\frY^{\wedge}}\boxtimes\bP^1(\Q_p))\simeq Rf_{\bh,*}(D_{\tildefrY^{\wedge}}\boxtimes\bP^1(\Q_p)),\]
    where all notations appeared should be understood as for sheaves of $(\varphi,\Gamma)$-modules with certain $\frg$-actions.  
\end{corollary}
\begin{proof}
    Over $\Z_p$, this is just Corollary \ref{corollary}. For the copy on $\bP^1(\Q_p)\setminus \Z_p$, just notice that for $k\geq 1$, $(\Pi.D_A)\otimes_L\Sym^kL^2\simeq \Pi.(D_A\otimes_L\Sym^kL^2)$ as $U(\frg)$-modules (since the $\frg$-action on $\Sym^kL^2$ integrates to a $\GL_2(\Q_p)$-action) and the adjoint action of $\Pi$ acts trivially on the center of $U(\frg)$. 
\end{proof}

\appendix
\section{On families of $(\varphi,\Gamma_K)$-modules}\label{sectionappendixfamilies}
\subsection{Beauville-Laszlo glueing}\label{subsectionbeauvillelaszlo}
We let $K$ be a $p$-adic local field. Recall 
\[t=\log([\epsilon])=X_{\Q_p}\prod_{m\geq 1}Q_m(X_{\Q_p})/p.\] 
Here $Q_m$ is the minimal polynomial of $\zeta_{p^m}-1$ over $\Q_p$, $X_{\Q_p}=[\epsilon]-1,\epsilon=(1,\zeta_p,\cdots,\zeta_{p^m},\cdots)\in \cO_{\widehat{K}_{\infty}}^{\flat}$ where $\widehat{K}_{\infty}$ denotes the $p$-adic completion and $\zeta_{p^m}$ is a primitive $p^m$-th root of unity. For $m\geq m(r)$ (see \S\ref{subsectionnotation} for $m(r)$), the continuous $\Gamma_K$-equivariant injection $\iota_{m}:\cR_{\Q_p,K}^{r}\hookrightarrow K_{m}[[t]]$ can be seen as the completion with respect to the kernel $(Q_m(X_{\Q_p}))$ of $\cR_{\Q_p,K}^{r}\rightarrow K_m=\cR_{\Q_p,K}^{r}/Q_m(X_{\Q_p})$ (see \cite[\S1.2]{berger2008construction} and \cite[Lem. 4.9]{berger2002representations}). For an affinoid algebra $A$ over $\Q_p$, the ring $(A\otimes_{\Q_p}K_m)[[t]]$ is the completion of $\cR_{A,K}^r$ with respect to the ideal $(Q_m(X_{\Q_p}))$ and we still write $\iota_m:\cR_{A,K}^r\rightarrow (A\otimes_{\Q_p}K_m)[[t]]$.\par

Suppose that $D_A$ is a $(\varphi,\Gamma_K)$-module over $\cR_{A,K}$ of rank $n$, base changed from a $(\varphi,\Gamma_K)$-module $D_A^{r}$ from $\cR_{A,K}^{r}$ for some $r$. 
Define $D_{\mathrm{dif}}^{m,+}(D_A):=(A\widehat{\otimes}K_m)[[t]]\otimes_{\iota_m,\cR_{A,K}^{r}}D_A^{r}$ for $m\geq m(r)$. Since by definition $\iota_{m+1}=\iota_m\circ \varphi^{-1}$, the map $K_m[[t]]\rightarrow K_{m+1}[[t]]$ induced by $\varphi:\cR_{K}^{r}\rightarrow\cR_{K}^{r/p}$ is $K_m[[t]]$-linear and $\varphi^*D_A^{r}=\cR_{A,K}^{r/p}\otimes_{\varphi,\cR_{A,K}^{r}}D^{r}_A\simeq D^{r/p}_A$ induces a $\Gamma_K$-linear isomorphism $(A\otimes_{\Q_p}K_{m+1})[[t]]\otimes_{(A\otimes_{\Q_p}K_m)[[t]]}D_{\mathrm{dif}}^{m,+}(D_A) \simeq D_{\mathrm{dif}}^{m+1,+}(D_A)$.
\begin{definition}
    A $(\varphi,\Gamma_K)$-module $M_A$ (resp. $M_A^r$) over $\cR_{A,K}[\frac{1}{t}]$ (resp. $\cR_{A,K}^r[\frac{1}{t}]$) is a finite projective $\cR_{A,K}[\frac{1}{t}]$-module (resp. finite projective $\cR_{A,K}^{r}[\frac{1}{t}]$-module) equipped with commuting continuous semi-linear actions of $\varphi,\Gamma_K$ (resp. $\Gamma_K$-isomorphism $\varphi^*M_A^r\simeq M_A^{r/p}$) such that there exists a $(\varphi,\Gamma_K)$-module $D_A^{r}$ over $\cR_A^{r}$ and an isomorphism $M_A\simeq \cR_{A,K}[\frac{1}{t}]\otimes_{\cR_{A,K}^{r}}D_A^{r}$ (resp. $M_A^r\simeq \cR_{A,K}^r[\frac{1}{t}]\otimes_{\cR_{A,K}^{r}}D_A^{r}$). 
\end{definition}

\begin{lemma}
    The ring $\cR_{A,K}^r$ is $t$ torsion-free.
\end{lemma}
\begin{proof}
    The short exact sequence $0\rightarrow \cR_{K}^{[s,r]}\stackrel{\times t}{\rightarrow}\cR_{K}^{[s,r]}\rightarrow \cR_{K}^{[s,r]}/t\rightarrow 0$ splits as $\Q_p$-Banach spaces as in the proof of \cite[Prop. 2.15]{liu2015triangulation}. Taking the completed tensor with $A$ we see $\cR_{A,K}^{[s,r]}$ is $t$-torsion free for all $r,s$. The limit $\cR_{A,K}^{r}=\varprojlim_s\cR_{A,K}^{[s,r]}$ is still $t$-torsion free.
\end{proof}
The functor $D_{\dif}^{m}(-)$ extends naturally for $(\varphi,\Gamma_K)$-modules over $\cR^{r}_{A,K}[\frac{1}{t}]$ and $m>m(r)$ by inverting $t$. And the $\varphi$-action induces $D_{\dif}^{m+1}(M_A)\simeq (A\otimes_{\Q_p} K_{m+1})[[t]][\frac{1}{t}]\otimes_{(A\otimes_{\Q_p} K_{m})[[t]][\frac{1}{t}]}D_{\dif}^{m}(M_A)$.  
\begin{proposition}\label{propositionappendixglueing}
    Suppose that $r$ is chosen such that $t$ is invertible in $\cR_{A,K}^{[r,r]}$ and $m=m(r)$. The functor 
    \[D_A^{r}\mapsto (M_A^{r}:=D_A^{r}[\frac{1}{t}],D_{\dif}^{m,+}(D_A^r),D_{\dif}^{m}(M_A^{r})=D_{\dif}^{m,+}(D_A^{r})[\frac{1}{t}])\] 
    induces an equivalence of categories 
    \[\Phi\Gamma_{A,K}^{r,+}\simeq \Phi\Gamma_{A,K}^{r}\times_{\mathrm{Rep}_{\mathrm{dif},A}^{m}(\Gamma_K)}\mathrm{Rep}_{\mathrm{dif},A}^{m,+}(\Gamma_K)\]
    between the category of $(\varphi,\Gamma_K)$-modules over $\cR_{A,K}^{r}$ and the category of triples $(M^{r}_A, D_{\dif,A}^{m,+},\alpha_m)$ where $M_A^r$ is a $(\varphi,\Gamma_K)$-module over $\cR_{A,K}^{r}[\frac{1}{t}]$, $D_{\dif,A}^{m,+}$ is a continuous semilinear $\Gamma_K$-representation over a projective $(A\otimes_{\Q_p}K_{m})[[t]]$-module of rank $n$ and $\alpha_m:D_{\dif,A}^{m,+}[\frac{1}{t}]\simeq D_{\dif}^m(M_A^r)$. The inverse functor is given by (write $D_{\dif,A}^{m',+}:=(K_{m'}\otimes_{\Q_p}A)[[t]]\otimes_{(K_m\otimes_{\Q_p}A)[[t]]}D_{\dif,A}^{m,+}$ for $m'\geq m$)
    \[(M_A^r,D_{\dif,A}^{m,+},\alpha_m)\mapsto \{x\in M_A^r\mid \iota_{m'}(x)\in D_{\dif,A}^{m',+}\subset D_{\dif,A}^{m',+}[\frac{1}{t}]\stackrel{\alpha_m}{\simeq} D_{\dif}^{m'}(M_A^r),\forall m'\geq m\}.\]
    Moreover, the equivalence commutes with arbitrary base change.
\end{proposition}
\begin{proof}
    This follows from the Beauville-Laszlo lemma \cite{beauville1995lemme} and the consideration of $\varphi$-actions. A finite projective $\varphi$-module $D_{A,K}^{r}$ over $\cR_A^{r}$ is equivalently the global section of a $\varphi$-bundle $\cD_{A}^r$ over the relative annulus $\bbU^{r}\times_{\Q_p}\Sp(K_0'\otimes_{\Q_p}A)$ (see \cite[Prop. 2.2.7]{kedlaya2014cohomology}). In terms of vector bundles, a $\varphi$-module $M_A^r$ over $\cR_{A,K}^r[\frac{1}{t}]$ corresponds to a $\varphi$-bundle over $\bbU^{r}\times_{\Q_p}\Sp(K_0'\otimes_{\Q_p}A)$ up to modifications along the divisors cut out by $Q_{m'}(X_{\Q_p})$ for $m'\geq m=m(r)$. Using the $\varphi$-action, modifications of $D_{A}^r$ are determined by the modification at $Q_m(X_{\Q_p})$ which is recorded as the $\Gamma_K$-lattice $D_{\dif}^{m,+}(D_A^r)$ inside $D_{\dif}^{m}(D_A^r)$. See \cite[Thm. 5.11]{fruttidoro2023non} for more details.
\end{proof}
\subsection{Almost de Rham families}\label{subsectionalmostderham}
We prove \cite[Prop. 5.3.28]{emerton2023introduction}. We need to write the proof here because some constructions in the proof are used for the main theorem. 

Let $K/\Q_p$ be a local field, $A$ be an affinoid algebra over $L$. We fix an embedding $\tau:K\hookrightarrow L\subset A$. A $\Gamma_K$-representation $D^{m,+}_{\dif, A}$ of rank $n$ with coefficients in $A\otimes_KK_m$ is a finite projective $(A\otimes_KK_m)[[t]]$-module of rank $n$ with a continuous semilinear action of $\Gamma_K$, where $\Gamma_K$ acts on $K_m$ by the Galois action and on $t$ via the cyclotomic character. Write $D_{\Sen,A}^{m,+}=D_{\dif,A}^{m,+}/t$. As in \cite[Prop. 3.7]{fontaine2004arithmetique}, differentiate the $\Gamma_{K}$-action, we can obtain a connection $\nabla$ on $D_{\dif,A}^{m,+}=\varprojlim_{k}D_{\dif,A}^{m,+}/t^k$ which is the Sen operator after modulo $t$. We say that $D_{\dif,A}^{m,+}$ is almost de Rham of weights $\bh=(h_1,\cdots,h_n)\in\Z^n, h_1\leq  \cdots\leq h_n$ if all the Sen polynomial of the specialization $D_{\Sen,x}^{m,+}$ is equal to $\prod_{i=1}^n(T-h_i)$ for any $x\in \Sp(A)$. Write $D_{\dif,A}^{m}=D_{\dif,A}^{m,+}[\frac{1}{t}]$. Denote by $\frS_{m,A}=(A\otimes_KK_m)[[t]]$. Let $\Rep_{\frS_{m,A}}(\Gamma_K)_{\bh}$ be the groupoid of almost de Rham semilinear $\Gamma_K$-representations over projective $\frS_{m,A}$-modules of weight $\bh$. For $m'\geq m$, the tensor product $-\otimes_{K_m}K_{m'}$ induces a functor $\Rep_{\frS_{m,A}}(\Gamma_K)_{\bh}\rightarrow \Rep_{\frS_{m',A}}(\Gamma_K)_{\bh}$.
\begin{definition}\label{definitioninfinitedif}
    Write $\frS_{\infty,A}=(A\otimes_KK_{\infty})[[t]]$.
    \begin{enumerate}
        \item Define the groupoid $\Rep_{\frS_{\infty,A}}(\Gamma_K)_{\bh}:=\varinjlim_m \Rep_{\frS_{m,A}}(\Gamma_K)_{\bh}$. Objects of $\Rep_{\frS_{\infty,A}}(\Gamma_K)_{\bh}$ consist of representations $D_{\dif,A}^{m,+}\in \Rep_{\frS_{m,A}}(\Gamma_K)_{\bh}$ for some $m$ identified with $D_{\dif,A}^{m,+}\otimes_{K_m}K_m'\in \Rep_{\frS_{m',A}}(\Gamma_K)_{\bh}$ for all $m'\geq m$. And for $D_{\dif,A}^{',m',+}\in \Rep_{\frS_{m',A}}(\Gamma_K)_{\bh}$, 
        \begin{align*}
            \Hom_{\Rep_{\frS_{\infty,A}}(\Gamma_K)_{\bh}}(D_{\dif,A}^{m,+},D_{\dif,A}^{',m',+})=\varinjlim_{m''\geq m,m'}\Hom_{\Rep_{\frS_{m'',A}}(\Gamma_K)_{\bh}}(D_{\dif,A}^{m'',+},D_{\dif,A}^{',m'',+})
        \end{align*}
        where $D_{\dif,A}^{m'',+}:=D_{\dif,A}^{m,+}\otimes_{K_m}K_{m''},D_{\dif,A}^{',m'',+}:=D_{\dif,A}^{',m',+}\otimes_{K_{m'}}K_{m''}$.
        \item For $D_{\dif,A}^{m,+}\in \Rep_{\frS_{m,A}}(\Gamma_K)_{\bh}$, we say that $m$ is large enough if the action of $\Gamma_{K_{m}}$ on $D_{\dif,A}^{m,+}/t$ is analytic, namely for any $\gamma\in \Gamma_{K_m}$, the action of $\gamma$ is given by the convergent series $\exp\left(\log(\epsilon(\gamma))\nabla\right)=\sum_{i=0}^{\infty}\frac{\left(\log(\epsilon(\gamma))\nabla\right)^i}{i!}$. 
    \end{enumerate}
\end{definition}
For any $D_{\dif,A}^{m,+}\in \Rep_{\frS_{m,A}}(\Gamma_K)_{\bh}$, there exists always $m'>m$ such that $m''$ is large enough for $D_{\dif,A}^{m'',+}$ for any $m''\geq m'$ (such that the series $\log(\gamma_{K_{m'}})$ in $\End_{A\times_K K_{m'}}(D_{\dif,A}^{m',+}/t)$ converges and $\gamma_{K_{m'}}=\exp(\log(\gamma_{K_{m'}}))$, cf. \cite[Prop. 2.2.14]{kedlaya2014cohomology}).

We take $G=\GL_{n/L}$ and $P_{\bh}$ a standard parabolic subgroup of $G$ such that the Weyl group of the Levi of $P_{\bh}$ is the stabilizer of $\bh$ in $\mathcal{S}_n$ as in \S\ref{sectionchangeofweight}. Let $ (\tildefrg_{\bh}/G)^{\wedge}_{0}(A)$ be the groupoid of triples $(D_{\rm pdR,A}, \nu_A,\Fil^{\bullet}D_{\pdR,A})$ where $D_{\rm pdR,A}$ is a finite projective $A$-module of rank $n$, $\Fil^{\bullet}D_{\pdR,A} $ is a decreasing filtration of projective sub-$A$-modules of type $\bh$ as in Definition \ref{defstacknilpotent} and $\nu_A\in \End_A(D_{\pdR,A})$ is a nilpotent endomorphism which keeps the filtration. 

We define the functor $D_{\rm pdR}: \Rep_{\frS_{\infty,A}}(\Gamma_K)_{\bh}\rightarrow (\tildefrg_{\bh}/G)^{\wedge}_{0}(A)$. 
\begin{definition}\label{definitionappendixDpdR}
    Let $\log(t)$ be the formal variable such that $\gamma.\log(t)=\log(\epsilon(\gamma))+\log(t)$ for $\gamma\in\Gamma_K$ and let the operator $\nu_A$ act on $\frS_{m,A}[\log(t)]$ as a $\frS_{m,A}$-linear derivative such that $\nu_A(\log(t))=-1$. 
    \begin{enumerate}
        \item For $D_{\dif,A}^{m,+}\in \Rep_{\frS_{m,A}}(\Gamma_K)_{\bh}$ such that $m$ is large enough, define the  $A$-module with a filtration
        \begin{align*}
            &D_{\rm pdR}(D_{\dif,A}^{m,+})=D_{\rm pdR}(D_{\dif,A}^{m}):=(D_{\dif,A}^{m}\otimes_{\frS_{m,A}}\frS_{m,A}[\log(t)])^{\Gamma_K},\\
            &\Fil^iD_{\rm pdR}(D_{\dif,A}^{m,+}):= (t^iD_{\dif,A}^{m,+}\otimes_{\frS_{m,A}}\frS_{m,A}[\log(t)])^{\Gamma_K}  
        \end{align*}
        for $i\in \Z$, equipped with an $A$-linear operator $\nu_A$ induced by the derivative $\nu_A$ on $\frS_{m,A}[\log(t)]$. 
        \item Given $(D_{\rm pdR,A}, \nu_A,\Fil^{\bullet}D_{\pdR,A})\in (\tildefrg_{\bh}/G)^{\wedge}_{0}(A)$, we let 
        \[D_{\dif}^{m}(D_{\rm pdR,A}, \nu_A):= (D_{\pdR,A}\otimes_{A}\frS_{m,A}[\frac{1}{t}][\log(t)])^{\nu_A=0}\]
        and 
        \[D_{\dif}^{m,+}(D_{\rm pdR,A}, \nu_A,\Fil^{\bullet}D_{\pdR,A}):= (\sum_{i\in\Z} \Fil^iD_{\pdR,A}\otimes_{A}t^{-i}\frS_{m,A}[\log(t)])^{\nu_A=0}\]
        be $\frS_{m,A}$-modules equipped with actions of $\Gamma_K$. 
    \end{enumerate}    
\end{definition}

\begin{proposition}\label{propositionappendixequivalencealmostdeRham}
    The functors $D_{\pdR}(-),D_{\dif}^{m,+}(-)$ are well defined and the following statements hold.
    \begin{enumerate}
        \item If $m$ is large enough for $D_{\dif,A}^{m,+}\in \Rep_{\frS_{m,A}}(\Gamma_K)_{\bh}$ in the sense of Definition \ref{definitioninfinitedif}, then the natural map
        \[\rho_{\pdR,m}:\sum_{i\in\Z}\Fil^iD_{\pdR}(D_{\dif,A}^{m,+})\otimes_A t^{-i}\frS_{m,A}[\log(t)]\rightarrow  D_{\dif,A}^{m,+}\otimes_{\frS_{m,A}}\frS_{m,A}[\log(t)].\]
        is an isomorphism and induces an isomorphism in $\Rep_{\frS_{m,A}}(\Gamma_K)_{\bh}$:
        \[D_{\dif}^{m,+}(D_{\rm pdR}(D_{\dif,A}^{m,+}), \nu_A,\Fil^{\bullet}D_{\pdR}(D_{\dif,A}^{m,+}))\simeq D_{\dif,A}^{m,+}.\]
        \item The functors $D_{\pdR}(-)$ and $D_{\dif}^{m,+}(-)$ induce an equivalence of groupoids $\Rep_{\frS_{\infty,A}}(\Gamma_K)_{\bh}\simeq (\tildefrg_{\bh}/G)^{\wedge}_{0}(A)$. Moreover, the equivalence commutes with arbitrary base change.
    \end{enumerate}
\end{proposition}
\begin{proof}
\textit{Step 1}. We show that the functor $D_{\dif}^{m,+}(-)$ is well-defined. Suppose that 
\[(D_{\rm pdR,A}, \nu_A,\Fil^{\bullet}D_{\pdR,A})\in (\tildefrg_{\bh}/G)^{\wedge}_{0}(A).\] 
Note that $D_{\dif}^{m}(D_{\rm pdR,A}, \nu_A)=(D_{\pdR,A}\otimes_{A}\frS_{m,A}[\log(t)])^{\nu_A=0}[\frac{1}{t}]$. We show that $(D_{\pdR,A}\otimes_{A}\frS_{m,A}[\log(t)])^{\nu_A=0}$ is an almost de Rham $\Gamma_K$-representation of weight $0$. First, 
\[\nu_A(\sum_{i\geq 0}x_i\log(t)^i)=\sum_{i\geq 0}(\nu_A(x_i)\log(t)^i-x_{i+1}(i+1)\log(t)^{i})=0\] 
for $x_i\in D_{\pdR,A}\otimes_A\frS_{m,A}$ if and only if $x_{i+1}=\frac{1}{i+1}\nu_A(x_{i})$ for $i\geq 0$. Since $\nu_A$ is nilpotent on $D_{\pdR}\otimes_A \frS_{m,A}$, we have an identification
\begin{align*}
     D_{\pdR,A}\otimes_{A}\frS_{m,A}&\simeq(D_{\pdR,A}\otimes_{A}\frS_{m,A}[\log(t)])^{\nu_A=0}\\
    x&\mapsto \sum_{i\geq 0}\frac{1}{i!}\nu_A^i(x)\log(t)^i.    
\end{align*}
as $\frS_{m,A}$-modules (but not $\Gamma_K$-equivariantly). Under this identification, the connection $\nabla$ on $(D_{\pdR,A}\otimes_{A}\frS_{m,A}[\log(t)])^{\nu_A=0}$ corresponds to $\nabla+\nu_A$ on $D_{\pdR,A}\otimes_{A}\frS_{m,A}$: $\nabla(\sum_ix_i\log(t)^i)=\sum_{i}(\nabla(x_i)+(i+1)x_{i+1})\log(t)^i=\sum_{i}(\nabla(x_i)+\nu_A(x_i))\log(t)^i$ where $\nu_A$ is linear for $\frS_{m,A}$ and $\nabla$ kills $D_{\pdR,A}$. Since $\nu_A$ is nilpotent, the Sen weights are pointwisely all zero.

Same argument shows that 
\[D_{\dif}^{m,+}(D_{\rm pdR,A}, \nu_A,\Fil^{\bullet}D_{\pdR,A})\simeq \sum_i \Fil^iD_{\pdR,A}\otimes_{A}t^{-i}\frS_{m,A}\] 
as an $\frS_{m,A}$-module. Since $\Fil^{\bullet}D_{\pdR,A}$ has projective graded pieces, we may choose a splitting $D_{\pdR,A}=D_{1}\oplus D_2\cdots\oplus D_{s}$ where $0\neq D_{i}=\Fil^{k_i}D_{\pdR,A}/\Fil^{k_i+1}D_{\pdR,A}$ and $k_1>\cdots>k_s$ such that $\{-k_1,\cdots,-k_s\}=\{h_1,\cdots,h_n\}$ as sets. Then 
\[D_{\dif}^{m,+}(D_{\rm pdR,A}, \nu_A,\Fil^{\bullet}D_{\pdR,A})=D_1\otimes_{A}t^{-k_1}\frS_{m,A}\oplus \cdots \oplus D_s\otimes_At^{-k_s}\frS_{m,A}\] 
is projective of rank $n$ over $\frS_{m,A}$ of weights $\bh$. If $m'\geq m$, by the above description, we see 
\[D_{\dif}^{m',+}(D_{\rm pdR,A}, \nu_A,\Fil^{\bullet}D_{\pdR,A})=D_{\dif}^{m,+}(D_{\rm pdR,A}, \nu_A,\Fil^{\bullet}D_{\pdR,A})\otimes_{K_m}K_{m'}.\]
Hence the image in $\Rep_{\frS_{\infty,A}}(\Gamma_K)_{\bh}$ is independent of $m$.

\textit{Step 2}. Take $D_{\dif,A}^{m,+}\in \Rep_{\frS_{m,A}}(\Gamma_K)_{\bh}\subset \Rep_{\frS_{\infty,A}}(\Gamma_K)_{\bh}$ such that $m$ is large enough. We show that $(D_{\rm pdR}(D_{\dif,A}^{m,+}),\nu_A,\Fil^{\bullet}D_{\rm pdR}(D_{\dif,A}^{m,+}))\in (\tildefrg_{\bh}/G)^{\wedge}_{0}(A)$ and is independent of $m$. 

Recall $D_{\Sen,A}^{m}=D_{\dif,A}^{m,+}/t$. Consider the $\Gamma_K$-representation $t^iD_{\Sen,A}^{m}$. There is a canonical $\Gamma_K$-decomposition $D_{\Sen,A}^{m}=\oplus_{i=1}^sD_{\Sen,A}^{m}\{\nabla_{\Sen}=-k_i\}$ according to the generalized eigenvalues of $\nabla_{\Sen}$ where each $D_{\Sen,A}^{m}\{\nabla_{\Sen}=-k_i\}$ is projective over $A\otimes_KK_m$ of rank $m_i$. 

By Lemma \ref{lemmaappendixDsen}, each $(t^iD_{\Sen,A}^{m}\otimes_{K_m}K_m[\log(t)])^{\Gamma_{K}}$ is finite projective over $A$ of rank the multiplicity of $-i$ in $\bh$ and the map
\begin{equation}\label{equationappendixDsen}
    \oplus_{i\in\Z}t^{-i}(t^iD_{\Sen,A}^{m}\otimes_{K_m}K_m[\log(t)])^{\Gamma_{K}}\otimes_{K}K_m[\log(t)]\rightarrow D_{\Sen,A}^{m}\otimes_{K_m}K_m[\log(t)]
\end{equation}
is an isomorphism of $\Gamma_K$-representations over $(A\otimes_KK_m)[\log(t)]$. Set
\[\mathrm{gr}^iD_{\pdR}(D_{\dif,A}^{m,+}):=(t^iD_{\dif,A}^{m,+}\otimes_{\frS_{m,A}}\frS_{m,A}[\log(t)])^{\Gamma_K}/(t^{i+1}D_{\dif,A}^{m,+}\otimes_{\frS_{m,A}}\frS_{m,A}[\log(t)])^{\Gamma_K}.\] 
There is an injection for all $i$
\[\mathrm{gr}^iD_{\pdR}(D_{\dif,A}^{m,+})\hookrightarrow (t^iD_{\Sen,A}^{m}\otimes_{K_m}K_m[\log(t)])^{\Gamma_K}\] 
which we claim is an isomorphism. 

Suppose that $v\in t^iD_{\dif,A}^{m,+}$ with image $\overline{v}\in t^iD_{\Sen, A}^m$ such that \[\sum_{i=0}^{\infty}\frac{(-1)^i}{i!}\nabla^i(\overline{v})\log(t)^i\in (t^iD_{\Sen,A}^{m}\otimes_{K_m}K_m[\log(t)])^{\nabla=0}.\] 
Take $a\leq b\in\Z$ such that $i,h_1,\cdots,h_n\in [a,b]$. In the beginning of the proof of \cite[Prop. A.10]{wu2021local}, there exists $l\geq 1$ and maps $\beta_{k}:t^{-b}D_{\dif,A}^{m,+}\rightarrow t^{-b}D_{\dif,A}^{m,+},k\geq 1$ such that $\beta_{k}(v)-\beta_{k+1}(v)\in t^{k+1-a}D_{\dif,A}^{m,+}$ and $(\gamma_{K_m}-1)^l\beta_k(x)\in t^{k+1-b}D_{\dif,A}^{m,+}$. Moreover, $l$ is large enough such that  $(\gamma_{K_m}-1)^l(t^{i}D_{\dif,A}^{m,+}/t^{i+1}D_{\dif,A}^{m,+})^{(\gamma_{K_m}-1)-\mathrm{nil}}=0$. By the construction, $\beta_k$ maps $t^{i}D_{\dif,A}^{m,+}$ to $t^{i}D_{\dif,A}^{m,+}$ and induces an automorphism of $(t^{i}D_{\dif,A}^{m,+}/t^{i+1}D_{\dif,A}^{m,+})^{(\gamma_{K_m}-1)-\mathrm{nil}}$ independent of $k$. Take any $v'$ such that $\beta_k(v')$ has image $\overline{v}$ in $t^{i}D_{\Sen,A}^m$. Then $\tilde{v}:=\varinjlim_k \beta_k(v')$ is a lift of $\overline{v}$ in $(t^{i}D_{\dif,A}^{m,+})^{(\gamma_{K_m}-1)-\mathrm{nil}}$. We get that $\nabla$ acts nilpotently on $\tilde{v}$ and $\sum_{i=0}^{\infty}\frac{(-1)^i}{i!}\nabla^i(\tilde{v})\log(t)^i\in (t^iD_{\dif,A}^{m,+}\otimes_{K_m}K_m[\log(t)])^{\Gamma_{K_m}}$ (cf. the proof of \cite[Lem. A.2]{wu2021local}). We conclude that the map $(t^iD_{\dif,A}^{m,+}\otimes_{\frS_{m,A}}\frS_{m,A}[\log(t)])^{\Gamma_{K_m}}\rightarrow (t^iD_{\Sen,A}^{m}\otimes_{K_m}K_m[\log(t)])^{\Gamma_{K_m}}$ is surjective. Take $\Gamma_K$-invariants we see
\[ \mathrm{gr}^iD_{\pdR}(D_{\dif,A}^{m,+})\simeq (t^iD_{\Sen,A}^{m}\otimes_{K_m}K_m[\log(t)])^{\Gamma_K}.\]
Thus $D_{\pdR}(D_{\dif,A}^{m,+})$ equipped with the filtration $\Fil^iD_{\pdR}(D_{\dif,A}^{m,+})=(t^iD_{\dif,A}^{m,+}\otimes_{\frS_{m,A}}\frS_{m,A}[\log(t)])^{\Gamma_K}$ has type $\bh$ (we use that extensions of projective modules are still projective). And if $m'\geq m$, $\mathrm{gr}^iD_{\pdR}(D_{\dif,A}^{m,+})=\mathrm{gr}^{i}D_{\pdR}(D_{\dif,A}^{m',+})$ for all $i$ by Lemma \ref{lemmaappendixDsen}. Hence the functor $D_{\pdR}(-)$ is independent of large enough $m$.

\textit{Step 3.} We show that the functors induce an equivalence of categories. We first show that the map $\rho_{\pdR,m}$ in (1) is an isomorphism.  By Step 1, both sides are finite projective over $\frS_{m,A}[\log(t)]$. We may choose a splitting of the filtration $\Fil^{\bullet}D_{\pdR}(D_{\dif,A}^{m,+})$ with graded pieces identified with $(t^iD_{\Sen,A}^{m}\otimes_{K_m}K_m[\log(t)])^{\Gamma_K}$. Modulo $t$, $\rho_{\pdR,m}$ coincides with the isomorphism (\ref{equationappendixDsen}), which is an injection. We get that $\rho_{\pdR,m}$ is an injection itself (since the source is $t$-adically separated). To see the surjectivity, we only need to show that $\nu=0$ part of the lefthand side $D_{\dif,A}^{m,+}$ is contained in the image. The $\nu=0$ part of the righthand side, which is a projective $\frS_{m,A}$-module, admits an explicit description in terms of $\mathrm{gr}^{\bullet}D_{\pdR}(D_{\dif,A}^{m,+})$ by Step 1. Modulo $t$, the map between $\nu=0$ part is given by the $\nu=0$ part of (\ref{equationappendixDsen}), the surjectivity follows. This also shows that the map $D_{\pdR}(-)$ is essentially surjective. By the definition of the groupoid $\Rep_{\frS_{\infty,A}}(\Gamma_K)_{\bh}$, $\Hom(D_{\dif,A}^{m,+},D_{\dif,A}^{m',+})=\varinjlim_{m''}\Hom(D_{\dif,A}^{m,+}\otimes_{K_m}K_{m''},D_{\dif,A}^{m',+}\otimes_{K_{m'}}K_{m''})$. To show fully faithfulness, we only need to show that for $m$ large enough, the map $\End_{\Rep_{\frS_{m,A}}(\Gamma_K)_{\bh}}(D_{\dif,A}^{m,+})\rightarrow \mathrm{End}_{(\tildefrg_{\bh}/G)^{\wedge}_{0}(A)}(D_{\pdR}(D_{\dif,A}^{m,+}))$ is an isomorphism. The map is an injection by taking $\nu=0$ part of the canonical isomorphism $\rho_{\pdR,m}$. The surjectivity follows similarly using $\rho_{\pdR,m}$. 

\textit{Step 4.} Assume that $\Sp(B)\rightarrow\Sp(A)$ is a morphism of affinoids. We need to show that for $m$ large enough, the natural map $\Fil^iD_{\pdR}(D_{\dif,A}^{m,+})\otimes_AB\rightarrow \Fil^iD_{\pdR}(D_{\dif,A}^{m,+}\otimes_{\frS_{m,A}}\frS_{m,B})$ is an isomorphism. Both sides are finite projective $B$-modules of the same rank, we can reduce to show that $\mathrm{gr}^iD_{\pdR}(D_{\dif,A}^{m,+})\otimes_AB\rightarrow \mathrm{gr}^iD_{\pdR}(D_{\dif,A}^{m,+}\otimes_{\frS_{m,A}}\frS_{m,B})$ is an isomorphism for all $i$. This follows from the proof of Lemma \ref{lemmaappendixDsen} (see also \cite[\S 3]{bellovin2015p}).
\end{proof}
\begin{lemma}\label{lemmaappendixDsen}
    Let $D_{\Sen, 0}^m$ be a semilinear $\Gamma_K$-representation over a finite projective $A\otimes_KK_m$-module of rank $n$ and with Sen weights pointwisely $0$ such that $m$ is large enough as in Definition \ref{definitioninfinitedif}. Then for $i\neq 0$, we have $(t^iD_{\Sen,0}^{m}\otimes_{K_m}K_{\infty}[\log(t)])^{\Gamma_{K}}=0$, and the map
    \[(D_{\Sen,0}^{m}\otimes_{K_m}K_{m}[\log(t)])^{\Gamma_{K}}\otimes_{K}K_{m}[\log(t)]\rightarrow D_{\Sen,0}^{m}\otimes_{K_m}K_{m}[\log(t)]\]
    is an isomorphism of $\Gamma_K$-representations. Moreover, for $m'\geq m$, $(D_{\Sen,0}^{m}\otimes_{K_m}K_{m}[\log(t)])^{\Gamma_{K}}=(D_{\Sen,0}^{m}\otimes_{K_m}K_{m'}[\log(t)])^{\Gamma_{K}}$.
\end{lemma}
\begin{proof}
    Weights zero means that the Sen operator $\nabla$ acts nilpotently on $D_{\Sen, 0}^m$ by Lemma \ref{lemmanilpotentmatrix}. Since $m$ is large enough, we have $D_{\Sen,0}^{m}\simeq (D_{\Sen,0}^{m}\otimes_{K_{m}}K_{m}[\log(t)])^{\Gamma_{K_{m}}}=(D_{\Sen,0}^{m}\otimes_{K_{m}}K_{m}[\log(t)])^{\nabla=0}$, cf. \cite[Lem. A.2]{wu2021local}. Actually, we have $D_{\Sen,0}^{m}=D_{\Sen,0}^{m}\{\gamma_{K_{m}}=1\}$ and an identification of $A\otimes_KK_{m}$-modules
    \[F:D_{\Sen,0}^{m}\simeq (D_{\Sen,0}^{m}\otimes_{K_{m}}K_{m}[\log(t)])^{\Gamma_{K_{m}}}:x\mapsto \sum_{i=0}^{\infty}\frac{(-1)^i}{i!}\nabla^i(x)\log(t)^i.\]
    We show that the natural $\Gamma_K$-map
    \begin{equation*}
        \rho_{\Sen,m}:(D_{\Sen,0}^{m}\otimes_{K_{m}}K_{m}[\log(t)])^{\Gamma_{K_{m}}}\otimes_{K_{m}}K_{m}[\log(t)]\rightarrow D_{\Sen,0}^{m}\otimes_{K_{m}}K_{m}[\log(t)]
    \end{equation*}
    is an isomorphism. Since $F:\nabla(x)\mapsto \sum_{i=0}^{\infty}\frac{(-1)^i}{i!}\nabla^{i+1}(x)\log(t)^{i}$ and $\nabla$ is nilpotent, one can verify that the map $\rho_{\Sen,m}$ is a surjection by a decreasing induction: for any $x\in D_{\Sen,0}^m$ and $i$, we have $\nabla^j(x)\otimes 1$ is in the image of $\rho_{\Sen, m}$ for all $j\geq i$. To show the injectivity, consider the $A\otimes_{K}K_{m}$-linear derivation of $\nu$ on the two sides: $\nu(\log(t))=-1$. Under the identification $F$, $\nu$ corresponds to $\nabla$ on $D_{\Sen,0}^{m}$. The map $\rho_{\Sen,m}$ induces an isomorphism on $\nu=0$ part: $\nu(\sum_i F(x_i)\log(t)^i)=0$ if and only if $(i+1)x_{i+1}=\nabla(x_{i})$ and in this case we have 
    \[\rho_{\Sen,m}(\sum_{i}F(x_i)\log(t)^i)=\sum_{i+j=k}\frac{(-1)^i}{i!j!}\nabla^k(x_0)\log(t)^k\neq 0\]
    if $x_0\neq 0$. Now suppose that $\rho_{\Sen,m}(x)=0$ for some $x=\sum_{i}F(x_i)\log(t)^i\neq 0$. Let $i$ be the minimal integer such that $\nu^i(x)=0$. Then $i\geq 1$ and $\rho_{\Sen,m}(\nu^{i-1}(x))=0$ which forces $\nu^{i-1}(x)=0$, contradiction! Hence $\rho_{\Sen,m}$ is an isomorphism. 

    Finally, by Galois descent, we have $(D_{\Sen,0}^{m}\otimes_{K_{m}}K_{m}[\log(t)])^{\Gamma_{K_{m}}}=(D_{\Sen,0}^{m}\otimes_{K_{m}}K_{m}[\log(t)])^{\Gamma_{K}}\otimes_KK_{m}$ as subspaces of $D_{\Sen,0}^{m}\otimes_{K_{m}}K_{m}[\log(t)]$. And for $m'\geq m$, we have an isomorphism
    \[(D_{\Sen,0}^{m}\otimes_{K_{m}}K_{m}[\log(t)])^{\Gamma_{K}}\otimes_{K}K_{m'}[\log(t)]\rightarrow D_{\Sen,0}^{m}\otimes_{K_{m}}K_{m'}[\log(t)].\]
    Taking $\Gamma_{K_{m'}}$-invariants, we see $(D_{\Sen,0}^{m}\otimes_{K_{m}}K_{m'}[\log(t)])^{\Gamma_{K_{m'}}}=(D_{\Sen,0}^{m}\otimes_{K_{m}}K_{m'}[\log(t)])^{\nabla=0}=(D_{\Sen,0}^{m}\otimes_{K_{m}}K_{m}[\log(t)])^{\Gamma_{K}}\otimes_{K}K_{m'}$. Taking $\Gamma_{K}$-invariants we get $(D_{\Sen,0}^{m}\otimes_{K_{m}}K_{m'}[\log(t)])^{\Gamma_{K}}=(D_{\Sen,0}^{m}\otimes_{K_{m}}K_{m}[\log(t)])^{\Gamma_{K}}$.
\end{proof}
\section{GAGA and formal functions}\label{sectionappendixformalfunction}
We consider formal completions of rigid spaces and coherent modules on these spaces.
\begin{definition}\label{definitionformalcompletion}
    Let $\frX$ be a rigid space over $L$ and let $\cI$ be a coherent sheaf of ideals of $\cO_{\frX}$. Let $\frX_{n}$ be the analytic closed subspace of $\frX$ defined by $\cI^n$. The formal completion of $\frX$ along $\frX_1$, denoted by $\frX^{\wedge}$, is the ringed site $(\frX_1,\cO_{\frX^{\wedge}})$ which has the same underlying Grothendieck topological space as $\frX_1$ and the structure sheaf $\cO_{\frX^{\wedge}}:=\varprojlim_n \cO_{\frX_n}$. 
\end{definition}
The space above should be considered as a formal rigid analytic space, except that we will ignore the topology on the sheaf of the topological rings $\cO_{\frX^{\wedge}}$.

We consider affinoid cases first. In the following, let $A$ be an affinoid algebra over a $p$-adic field $L$ and let $I\subset A$ be an ideal. Let $\frY^{\wedge}=\varinjlim_n \frY_n=\varinjlim_n \Sp(A/I^n)$ be the formal completion of $\frY=\Sp(A)$ along $\Sp(A/I)$. For an affinoid open subspace $\Sp(B)\subset \Sp(A)$, $\cO_{\frY_n}(\Sp(B/I))=B\widehat{\otimes}_AA/I^n=B/I^n$ and $\cO_{\frY^{\wedge}}(B/I)=B^{\wedge}:=\varprojlim_nB/I^n$. 
\begin{lemma}
    An admissible open subset of $\Sp(A/I)$ admits a covering by open affinoids of the form $\Sp(B/I)$ for affinoid opens $\Sp(B)\subset \Sp(A)$. 
\end{lemma}
\begin{proof}
    A rational subdomain of $\Sp(A/I)$ has the form $\Sp((A/I)\langle x_1,\cdots,x_n \rangle/(x_1g-f_1,\cdots,x_ng-f_n) )$ for some $f_i,g\in A/I$ such that $f_i,g$ generate the unit ideal of $A/I$. Take lifts $\widetilde{f}_i,\widetilde{g}$ in $A$ and add possibly some $\widetilde{f}_{n+1},\cdots,\widetilde{f}_{n+k}\in I$, we see it has the form $\Sp(B/I)$ for a rational subdomain $\Sp(B)\subset \Sp(A)$ where $B=A\langle x_1,\cdots,x_{n+k} \rangle/(x_1\widetilde{g}-\widetilde{f}_1,\cdots,x_{n+k}\widetilde{g}-\widetilde{f}_{n+k})$. Then the statement follows from that rational subdomains form a basis for the Grothendieck topology \cite[Cor. 4.2/12]{bosch2014formal}.
\end{proof}

\begin{definition}\label{definitioncoherentmodules}
    An $\cO_{\frY^{\wedge}}$-module $\cF$ on the $I$-adic formal affinoid space $\frY^{\wedge}$ is coherent if $\cF$ has the form $\cF=\varprojlim \cF_n$ where $\cF_n$ are coherent $\cO_{\frY_n}$-modules such that $\cF_n/I^{n-1}=\cF_{n-1}$.
\end{definition}

Suppose that $F$ is a finitely generated $A^{\wedge}$-module (a coherent $\cO_{\Spec(A^{\wedge})}$-module), we can associate a coherent $\cO_{\frY^{\wedge}}$-module $\cF=\varprojlim \cF_n$ where $\cF_n$ is the coherent $\cO_{\frY_n}$-module attached to $F/I^n$. The following lemma is an analogue of \cite[Prop. 10.10.5, Ch.I]{EGAI}.
\begin{lemma}\label{lemmaGAGAaffinoid} Let $\frY^{\wedge}$ be as above.
    \begin{enumerate}
        \item Let $\cF=\varprojlim_n\cF_n$ be a coherent $\cO_{\frY^{\wedge}}$-module. Then we have $\cF(\frY^{\wedge})=\varprojlim_n \cF_n(\frY_n)$ and $R^i\Gamma(\frY^{\wedge},\cF)=0$ for $i>0$. Moreover, $\cF_n$ is uniquely determined by $F:=\cF(\frY^{\wedge})$ which is a finitely generated $A^{\wedge}$-module and $\cF_n(\frY^{\wedge})=F/I^n$. And for any affinoid open $\Sp(B)\subset \Sp(A)$ which defines an affinoid open $\Sp(B/I)\subset \Sp(A/I)$, we have $\cF(\Sp(B/I))=F\otimes_{A^{\wedge}}B^{\wedge}$. 
        \item The functor $F\mapsto \cF$ from the category of finitely generated $\cO(\frY^{\wedge})$-modules to the category of coherent $\cO_{\frY^{\wedge}}$-modules induces an equivalence of abelian categories. 
        \item An $\cO_{\frY^{\wedge}}$-module $\cF$ is coherent if and only if it is coherent in the sense of \cite[\href{https://stacks.math.columbia.edu/tag/03DK}{Tag 03DK}]{stacks-project}. 
    \end{enumerate}
\end{lemma}
\begin{proof}
    (1) Since the maps $\cF_n\rightarrow \cF_{n-1}$ are surjective, we have $\cF=R\varprojlim_n\cF_n$. Since $\frY^{\wedge}$ is affinoid, the maps $\cF_n(\frY^{\wedge})\rightarrow \cF_{n-1}(\frY^{\wedge})$ are surjective. Hence $R\Gamma(\frY^{\wedge},\cF)=\varprojlim_n \cF_n(\frY_n)$ by \cite[\href{https://stacks.math.columbia.edu/tag/0D60}{Tag 0D60}]{stacks-project}. By \cite[\href{https://stacks.math.columbia.edu/tag/09B8}{Tag 09B8}]{stacks-project}, we have $\cF(\frY^{\wedge})/I^n=\cF_n(\frY^{\wedge})$ and $\cF(\frY^{\wedge})$ is finitely generated by Nakayama lemma, see \cite[\href{https://stacks.math.columbia.edu/tag/087W}{Tag 087W}]{stacks-project}. Finally, since $\cF_n$ is coherent over $\frY_{n}$ and $\Sp(B/I^n)$ is an affinoid open in $\Sp(A/I^n)$, $\cF_n(\Sp(B/I^n))=\cF_n(A/I^n)\otimes_{A/I^n}B/I^n=F/I^n\otimes_{A/I^n}B/I^n$. Taking inverse limit $\cF(\Sp(B/I^n))=\varprojlim_n F\otimes_{A^{\wedge}}B/I^n= F\otimes_{A^{\wedge}}B^{\wedge}$ using that $F$ is finitely generated over $A^{\wedge}$.

    (2) The essential surjectivity is by (1). We need verify fully faithfulness. Let $F,G$ be two finitely generated $A^{\wedge}$-modules and let $\cF,\cG$ be the corresponding coherent sheaves. Then $\Hom_{A^{\wedge}}(F,G)=\varprojlim_n\Hom_{A/I^n}(F/I^n,G/I^n)=\varprojlim_n\Hom_{\cO_{\frY^{\wedge}}}(\cF_n,\cG_n)$ by \cite[\href{https://stacks.math.columbia.edu/tag/0EHN}{Tag 0EHN}]{stacks-project}. Apply this argument for the formal affinoid subspace $\Sp(B/I)$ of $\frY^{\wedge}$, we get 
    \[\cH om_{\cO_{\frY^{\wedge}}}(\cF,\cG)(\Sp(B/I))=\Hom_{B^{\wedge}}(\cF(B^{\wedge}),\cG(B^{\wedge}))=\varprojlim_n\cH om_{\cO_{\frY^{\wedge}}}(\cF_n,\cG_n)(\Sp(B/I)).\] 
    Hence $\cH om_{\cO_{\frY^{\wedge}}}(\cF,\cG)=\varprojlim_n\cH om_{\cO_{\frY^{\wedge}}}(\cF_n,\cG_n).$
    Taking global sections we see $\Hom_{\cO_{\frY^{\wedge}}}(\cF,\cG)=\Hom_{A^{\wedge}}(F,G)$. This proves the equivalence. See \cite[\href{https://stacks.math.columbia.edu/tag/087X}{Tag 087X}]{stacks-project} for the structure of abelian categories.
    
    (3) The proof is the same as for \cite[Prop. 10.10.5, Ch.I]{EGAI}, using that affinoid algebras as well as their completions are Noetherian.
\end{proof} 

Let $B$ be an affinoid algebra over $L$. Write $S=\Spec(B)$ and $\frS=\Sp(B)$. Suppose that $f:X_S\rightarrow \Spec(B)$ is a projective scheme over $S$ and let $X_{\frS}$ be its relative analytification over $\frS$ which is equipped with a map of locally ringed spaces $\iota: X_{\frS}\rightarrow X_S$ (see \cite[Exa. 2.2.11, Exa. 2.3.11]{conrad2006relative}). The following result was firstly proved in \cite{kopfeigentliche}. See also \cite[Exa. 3.2.6]{conrad2006relative} or \cite[Ann. A]{poineau2010raccord}.
\begin{theorem}[Relative GAGA theorem]\label{theoremGAGA}
    In the above situation, the functor $\cF\mapsto \iota^*\cF$ induces an equivalence of categories of coherent $\cO_{X_S}$-modules and coherent $\cO_{X_{\frS}}$-modules. And for any coherent $\cO_{X_S}$-module $\cF$ on $X_S$, the natural morphism $H^i(X_S, \cF)\rightarrow H^i(X_{\frS},\iota^*\cF)$ of finite $B$-modules is an isomorphism for all $i\geq 0$.
\end{theorem}
We go to the formal setting. Suppose that $f_Y:X\rightarrow Y$ is a projective scheme over $Y=\Spec(A)$. Let $X_{Y_n}=X\times_YY_n$ where $Y_n=\Spec(A/I^n)$. We form relative analytification $X_{\frY_n}$ for $X_{Y_n}$ over $\Spec(A/I^n)$ with proper maps $X_{\frY_n}\rightarrow \frY_n$ as well as $X_{\frY}\rightarrow \frY$. Then $X_{\frY^{\wedge}}=\varinjlim_n X_{\frY_n}$ is the formal completion of $X_{\frY}$ along the closed subspace $X_{\frY_1}$. We can define the category of coherent $\cO_{X_{\frY^{\wedge}}}$-modules as in Definition \ref{definitioncoherentmodules}, which is equivalent to the usual definition on a ringed site by Lemma \ref{lemmaGAGAaffinoid} (cf. \cite[Thm. 10.11.3, Ch.I]{EGAI}).

We also have a formal scheme $X_{\Spf(A^{\wedge})}=\varinjlim_n X_{Y_n}$ and a scheme $X_{\Spec(A^{\wedge})}$ the base change of $X_Y=X$ to $\Spec(A^{\wedge})$. There are natural morphisms of ringed sites $X_{\frY^{\wedge}}\rightarrow X_{\Spf(A^{\wedge})}\rightarrow X_{\Spec(A^{\wedge})}$. Write $f_{\frY_n}: X_{\frY_n}\rightarrow \frY_{n}$. The following corollary generalizes Lemma \ref{lemmaGAGAaffinoid}.
\begin{corollary}\label{corollaryGAGA}
    Let $X_{\frY^{\wedge}},X_{\Spf(A^{\wedge})},X_{\Spec(A^{\wedge})}$ be as above.
    \begin{enumerate}
        \item The category of coherent $\cO_{X_{\frY^{\wedge}}}$-modules is equivalent to the category of coherent modules on the formal scheme $X_{\Spf(A^{\wedge})}$ (in the sense of \cite[\href{https://stacks.math.columbia.edu/tag/089N}{Tag 089N}]{stacks-project}). Both categories are equivalent to the category of coherent $\cO_{X_{\Spec(A^{\wedge})}}$-modules on the scheme $X_{\Spec(A^{\wedge})}$.
        \item Let $\cF=\varprojlim_n\cF_n$ be a coherent module on $X_{\frY^{\wedge}}$ and $\cF_{A^{\wedge}}$ be the corresponding coherent module on $X_{\Spec(A^{\wedge})}$ by (1). Then for all $i\geq 0$, $R^if_{\frY^{\wedge},*}\cF=\varprojlim_nR^if_{\frY_n,*}\cF_n$ and is the coherent $\cO_{\frY^{\wedge}}$-module attached to the finite $A^{\wedge}$-module $H^i(X_{\Spec(A^{\wedge})},\cF_{A^{\wedge}})$.
    \end{enumerate}
\end{corollary}
\begin{proof}
    (1) The first equivalence is an application of Theorem \ref{theoremGAGA} above for each $X_{\frY_n}$ and by the definition of the category of coherent modules. The equivalence of coherent modules on $X_{\Spf(A^{\wedge})}$ and $X_{\Spec(A^{\wedge})}$ is Grothendieck's existence theorem \cite[\href{https://stacks.math.columbia.edu/tag/08BE}{Tag 08BE}]{stacks-project}.

    (2) Since $Rf_{\frY^{\wedge},*}\cF=Rf_{\frY^{\wedge},*}R\varprojlim_n\cF_n=R\varprojlim_nRf_{\frY^{\wedge},*}\cF_n$, we have $R^if_{\frY^{\wedge},*}\cF=\varprojlim_nR^if_{\frY_n,*}\cF_n$ by \cite[\href{https://stacks.math.columbia.edu/tag/0D60}{Tag 0D60}]{stacks-project} for $i\geq 0$ provided that $R^1\varprojlim_n R^if_{\frY_n,*}\cF_n=0$ for all $i$. We need to show that the inverse system $R^if_{\frY_n,*}\cF_n$ is Mittag-Leffler (cf. \cite{emmanouil1996mittag}). Write $(\cF_{A/I^n})_{n}$ for the coherent modules over $X_{\Spf(A^{\wedge})}$ by the equivalence in (1). Then $(\cF_{A/I^n})_{n}=(\cF_{A^{\wedge}}/I^n)_n$. Under the equivalence in Theorem \ref{theoremGAGA}, $R^if_{\frY_n,*}\cF_n=R^if_{Y_n,*}\cF_{A/I^n}$ as coherent sheaves associated to the same $A/I^n$-module $H^i(X_{Y_n},\cF_{A/I^n})=H^i(X_{\Spec(A^{\wedge})},\cF_{A^{\wedge}}/I^n)$. Apply \cite[\href{https://stacks.math.columbia.edu/tag/02OB}{Tag 02OB}]{stacks-project} for the proper morphism $X_{\Spec(A^{\wedge})}\rightarrow \Spec(A^{\wedge})$ and the sheaf $\cF_{A^{\wedge}}$, we see the system $(R^if_{\frY_n,*}\cF_n)_n$ is Mittag-Leffler. By \cite[\href{https://stacks.math.columbia.edu/tag/087U}{Tag 087U}]{stacks-project}, $H^i(X_{\Spec(A^{\wedge})},\cF_{A^{\wedge}})=\varprojlim_n H^i(X_{\frY_n},\cF_{n})$ as finite $A^{\wedge}$-modules. 
    
    Hence $R^if_{\frY^{\wedge},*}\cF(\frY^{\wedge})=\varprojlim_nR^if_{\frY_n,*}\cF_n(\frY_n)=H^i(X_{\Spec(A^{\wedge})},\cF_{A^{\wedge}})$ as $A^{\wedge}$-modules. For an affinoid subdomain $\Sp(B)\subset \Sp(A)$, a similar statement holds replacing $A^{\wedge}$ by $B^{\wedge}$. The ring map $A\rightarrow B$ is flat \cite[Cor. 4.1/5]{bosch2014formal}. Hence the maps $A/I^n\rightarrow B/I^n$ and $A^{\wedge}\rightarrow B^{\wedge}$ are flat by Lemma \ref{lemmaflatcompletion} below. By the flat base change, we have $H^i(X_{\Spec(B^{\wedge})},\cF_{B^{\wedge}})=H^i(X_{\Spec(A^{\wedge})},\cF_{A^{\wedge}})\otimes_{A^{\wedge}}B^{\wedge}$. Thus $R^if_{\frY^{\wedge},*}\cF=\varprojlim_nR^if_{\frY_n,*}\cF_n$, as an $\cO_{\frY^{\wedge}}$-module, is the coherent $\cO_{\frY^{\wedge}}$-module attached to the finite $A^{\wedge}$-module $H^i(X_{\Spec(A^{\wedge})},\cF_{A^{\wedge}})$ (see Lemma \ref{lemmaGAGAaffinoid}).
\end{proof}
We used frequently the following lemma.
\begin{lemma}\label{lemmaflatcompletion}
    Let $A$ be a ring and let $I$ be an ideal of $A$. Suppose that $(M_n)_n$ is an inverse system of $A$-modules such that $M_n$ is a flat $A/I^n$-module for any $n$. Let $M:=\varprojlim_nM_n$.
    \begin{enumerate}
        \item Suppose that $A$ is Noetherian and that the transition maps $M_{n+1}\rightarrow M_n$ are surjective for all $n$, then $M$ is flat over $A$ and $Q\otimes_AM=\varprojlim_n Q\otimes_AM_n$ for any finite $A$-module $Q$.
        \item Suppose that $M_{n+1}\otimes_{A/I^{n+1}}A/I^n=M_n$  and $M_n$ is finite flat over $A/I^n$ for all $n$. If $M_1$ is finite projective over $A/I$, then $M$ is finite projective over $A^{\wedge}=\varprojlim A/I^n$ and $M\otimes_{A^{\wedge}}A/I^n=M_n$ for all $n$. 
    \end{enumerate} 
\end{lemma}
\begin{proof}
    (1) is \cite[\href{https://stacks.math.columbia.edu/tag/0912}{Tag 0912}]{stacks-project}. (2) follows from \cite[\href{https://stacks.math.columbia.edu/tag/0D4B}{Tag 0D4B}]{stacks-project}. 
\end{proof}
\bibliographystyle{alpha}
\bibliography{mybib}

\begin{thebibliography}{BMRR08}

\bibitem[ABV92]{adams1992langlands}
Jeffrey Adams, Dan Barbasch, and David~A. Vogan, Jr.
\newblock {\em The {Langlands} classification and irreducible characters for
  real reductive groups}, volume 104 of {\em Prog. Math.}
\newblock Boston, MA etc.: Birkh{\"a}user, 1992.

\bibitem[Bel15]{bellovin2015p}
Rebecca Bellovin.
\newblock {$p$-adic Hodge theory in rigid analytic families}.
\newblock {\em Algebra \& Number Theory}, 9(2):371--433, 2015.

\bibitem[Ber02]{berger2002representations}
Laurent Berger.
\newblock Repr{\'e}sentations $p$-adiques et {\'e}quations diff{\'e}rentielles.
\newblock {\em Inventiones mathematicae}, 148:219--284, 2002.

\bibitem[Ber08a]{berger2008construction}
Laurent Berger.
\newblock Construction de {($\varphi$, $\Gamma$)}-modules: repr{\'e}sentations
  $p$-adiques et {$B$}-paires.
\newblock {\em Algebra \& Number Theory}, 2(1):91--120, 2008.

\bibitem[Ber08b]{berger2008equations}
Laurent Berger.
\newblock {{\'E}quations diff{\'e}rentielles $p$-adiques et $(\varphi,
  N)$-modules filtr{\'e}s}.
\newblock {\em Ast{\'e}risque}, 319:13--38, 2008.

\bibitem[BG80]{bernstein1980tensor}
Joseph~N Bernstein and Sergei~I Gelfand.
\newblock Tensor products of finite and infinite dimensional representations of
  semisimple lie algebras.
\newblock {\em Compositio Mathematica}, 41(2):245--285, 1980.

\bibitem[BG99]{beilinson1999wall}
Alexander Beilinson and Victor Ginzburg.
\newblock Wall-crossing functors and $\mathcal{D}$-modules.
\newblock {\em Representation Theory of the American Mathematical Society},
  3(1):1--31, 1999.

\bibitem[BHS19]{breuil2019local}
Christophe Breuil, Eugen Hellmann, and Benjamin Schraen.
\newblock A local model for the trianguline variety and applications.
\newblock {\em Publications math{\'e}matiques de l'IH{\'E}S}, 130(1):299--412,
  2019.

\bibitem[BK07]{brion2007frobenius}
Michel Brion and Shrawan Kumar.
\newblock {\em Frobenius splitting methods in geometry and representation
  theory}, volume 231.
\newblock Springer Science \& Business Media, 2007.

\bibitem[BK15]{backelin2015singular}
Erik Backelin and Kobi Kremnitzer.
\newblock Singular localization of $\mathfrak{g}$-modules and applications to
  representation theory.
\newblock {\em Journal of the European Mathematical Society},
  17(11):2763--2787, 2015.

\bibitem[BL95]{beauville1995lemme}
Arnaud Beauville and Yves Laszlo.
\newblock Un lemme de descente.
\newblock {\em Comptes Rendus de l'Academie des Sciences-Serie I-Mathematique},
  320(3):335--340, 1995.

\bibitem[BMR06]{bezrukavnikov2006singular}
Roman Bezrukavnikov, Ivan Mirkovi{\'c}, and Dmitriy Rumynin.
\newblock {Singular localization and intertwining functors for reductive Lie
  algebras in prime characteristic}.
\newblock {\em Nagoya Mathematical Journal}, 183:1--55, 2006.

\bibitem[BMRR08]{bezrukavnikov2008localization}
Roman Bezrukavnikov, Ivan Mirkovi{\'c}, Dmitriy Rumynin, and Simon Riche.
\newblock {Localization of modules for a semisimple Lie algebra in prime
  characteristic}.
\newblock {\em Annals of Mathematics}, pages 945--991, 2008.

\bibitem[Bos14]{bosch2014formal}
Siegfried Bosch.
\newblock {\em Lectures on formal and rigid geometry}, volume 2105.
\newblock Springer, 2014.

\bibitem[Bro93]{broer1993line}
Bram Broer.
\newblock Line bundles on the cotangent bundle of the flag variety.
\newblock {\em Inventiones mathematicae}, 113(1):1--20, 1993.

\bibitem[Che11]{Chenevier2011fern}
Ga\"etan Chenevier.
\newblock On the infinite fern of {Galois} representations of unitary type.
\newblock {\em Annales scientifiques de l'\'Ecole Normale Sup\'erieure}, Ser.
  4, 44(6):963--1019, 2011.

\bibitem[Col10]{colmez2010representations}
Pierre Colmez.
\newblock {Repr{\'e}sentations de $\mathrm{GL}_2(\mathbb{Q}_p)$ et $(\varphi,
  \Gamma)$-modules}.
\newblock {\em Ast{\'e}risque}, 330(281):509, 2010.

\bibitem[Col16]{colmez2016representations}
Pierre Colmez.
\newblock {Repr{\'e}sentations localement analytiques de
  $\mathrm{GL}_2(\mathbb{Q}_p)$ et ($\varphi$, $\Gamma$)-modules}.
\newblock {\em Represent. Theory}, 20:187--248, 2016.

\bibitem[Col18]{colmez2018poids}
Pierre Colmez.
\newblock {Correspondance de Langlands locale $p$-adique et changement de
  poids}.
\newblock {\em Journal of the European Mathematical Society}, 21(3):797--838,
  2018.

\bibitem[Con06]{conrad2006relative}
Brian Conrad.
\newblock Relative ampleness in rigid geometry.
\newblock In {\em Annales de l'institut Fourier}, volume~56, pages 1049--1126,
  2006.

\bibitem[CT09]{conrad2009non}
Brian Conrad and Michael Temkin.
\newblock Non-archimedean analytification of algebraic spaces.
\newblock {\em Journal of Algebraic Geometry}, 18(4):731--788, 2009.

\bibitem[Din23]{ding2023change}
Yiwen Ding.
\newblock Change of weights for locally analytic representations of
  $\mathrm{GL}_2(\mathbb{Q}_p)$.
\newblock {\em arXiv:2307.04332}, 2023.

\bibitem[Din24]{Ding2024wall}
Yiwen Ding.
\newblock {Towards the wall-crossing of locally $\mathbb{Q}_p$-analytic
  representations of $\mathrm{GL}_n(K)$ for a $p$-adic field $K$}.
\newblock {\em arXiv:2404.06315}, 2024.

\bibitem[Dos12]{dospinescu2012actions}
Gabriel Dospinescu.
\newblock {Actions infinit{\'e}simales dans la correspondance de Langlands
  locale $p$-adique}.
\newblock {\em Mathematische Annalen}, 354(2):627--657, 2012.

\bibitem[DPS20]{dospinescu2020infinitesimal}
Gabriel Dospinescu, Vytautas Pa{\v{s}}k{\=u}nas, and Benjamin Schraen.
\newblock Infinitesimal characters in arithmetic families.
\newblock {\em arXiv:2012.01041}, 2020.

\bibitem[EGH25]{emerton2023introduction}
Matthew Emerton, Toby Gee, and Eugen Hellmann.
\newblock {An introduction to the categorical $p$-adic Langlands program}.
\newblock {\em to appear in the Proceedings of the 2022 IHES Summer School on
  the Langlands program}, 2025.

\bibitem[Emm96]{emmanouil1996mittag}
Ioannis Emmanouil.
\newblock {Mittag-Leffler condition and the vanishing of $\varprojlim^1$}.
\newblock {\em Topology}, 35(1):267--271, 1996.

\bibitem[Fon04]{fontaine2004arithmetique}
Jean-Marc Fontaine.
\newblock Arithm{\'e}tique des repr{\'e}sentations galoisiennes $p$-adiques.
\newblock {\em Ast{\'e}risque}, 295:1--115, 2004.

\bibitem[Fru23]{fruttidoro2023non}
Martina Fruttidoro.
\newblock Non-regular trianguline representations.
\newblock {\em arXiv:2310.19499}, 2023.

\bibitem[Gro60]{EGAI}
Alexander Grothendieck.
\newblock {\'El\'ements de g\'eom\'etrie alg\'ebrique : {I.} {Le} langage des
  sch\'emas}.
\newblock {\em Publications Math\'ematiques de l'IH\'ES}, 4:5--228, 1960.

\bibitem[GV06]{grothendieck2006SGA4}
A~Grothendieck and JL~Verdier.
\newblock {\em Th{\'e}orie des Topos et Cohomologie Etale des Sch{\'e}mas.
  S{\'e}minaire de G{\'e}om{\'e}trie Alg{\'e}brique du Bois-Marie 1963-1964
  (SGA 4): Tome 2}, volume 270.
\newblock Springer, 2006.

\bibitem[Har13]{hartshorne2013algebraic}
Robin Hartshorne.
\newblock {\em Algebraic geometry}, volume~52.
\newblock Springer Science \& Business Media, 2013.

\bibitem[Hel23]{hellmann2023derived}
Eugen Hellmann.
\newblock {On the derived category of the Iwahori-Hecke algebra}.
\newblock {\em Compositio Mathematica}, 159(5):1042--1110, 2023.

\bibitem[JLS21]{jena2021translation}
Akash Jena, Aranya Lahiri, and Matthias Strauch.
\newblock Translation functors for locally analytic representations.
\newblock {\em arXiv preprint arXiv:2107.08493}, 2021.

\bibitem[Ked09]{kedlaya2009some}
Kiran~S Kedlaya.
\newblock Some new directions in $p$-adic hodge theory.
\newblock {\em Journal de theorie des nombres de Bordeaux}, 21(2):285--300,
  2009.

\bibitem[Kie67]{kiehl1967theorem}
Reinhardt Kiehl.
\newblock {Theorem A und Theorem B in der nichtarchimedischen
  Funktionentheorie}.
\newblock {\em Inventiones mathematicae}, 2:256--273, 1967.

\bibitem[K{\"o}p74]{kopfeigentliche}
Ursula K{\"o}pf.
\newblock {\"U}ber eigentliche familien algebraischer variet{\"a}ten {\"u}ber
  affinoiden r{\"a}umen.
\newblock {\em Schriftenreihe des Mathematischen Institutes Münster, 2.Serie.
  Heft 7}, 1974.

\bibitem[KPX14]{kedlaya2014cohomology}
Kiran Kedlaya, Jonathan Pottharst, and Liang Xiao.
\newblock Cohomology of arithmetic families of {$(\varphi, \Gamma)$}-modules.
\newblock {\em Journal of the American Mathematical Society}, 27(4):1043--1115,
  2014.

\bibitem[Liu15]{liu2015triangulation}
Ruochuan Liu.
\newblock Triangulation of refined families.
\newblock {\em Commentarii Mathematici Helvetici}, 90(4):831--904, 2015.

\bibitem[Nak09]{nakamura2009classification}
Kentaro Nakamura.
\newblock Classification of two-dimensional split trianguline representations
  of $p$-adic fields.
\newblock {\em Compositio Mathematica}, 145(4):865--914, 2009.

\bibitem[Nak14]{nakamura2014deformations}
Kentaro Nakamura.
\newblock Deformations of trianguline {$B$}-pairs and {Zariski} density of two
  dimensional crystalline representations.
\newblock {\em Journal of mathematical sciences, the University of Tokyo},
  20(4):461--568, 2014.

\bibitem[Poi10]{poineau2010raccord}
J{\'e}r{\^o}me Poineau.
\newblock {Raccord sur les espaces de Berkovich}.
\newblock {\em Algebra \& Number Theory}, 4(3):297--334, 2010.

\bibitem[ST03]{schneider2003algebras}
Peter Schneider and Jeremy Teitelbaum.
\newblock Algebras of $p$-adic distributions and admissible representations.
\newblock {\em Inventiones mathematicae}, 153:145--196, 2003.

\bibitem[{Sta}24]{stacks-project}
The {Stacks Project Authors}.
\newblock \textit{Stacks Project}.
\newblock \url{https://stacks.math.columbia.edu}, 2024.

\bibitem[Ste61]{steinberg1961general}
Robert Steinberg.
\newblock {A general Clebsch-Gordan theorem}.
\newblock {\em Bull. Amer. Math. Soc.}, 67(6):406--407, 1961.

\bibitem[Str14]{straser2014geometric}
Oliver Straser.
\newblock {\em On the geometric tensor functor for real reductive groups}.
\newblock PhD thesis, Albert-Ludwigs-Universit{\"a}t Freiburg, 2014.

\bibitem[Wu22]{wu:tel-04116114}
Zhixiang Wu.
\newblock {\em {Trianguline variety and eigenvariety at points with non-regular
  Hodge-Tate weights}}.
\newblock Theses, {Universit{\'e} Paris-Saclay}, July 2022.

\bibitem[Wu24]{wu2021local}
Zhixiang Wu.
\newblock Local models for the trianguline variety and partially classical
  families.
\newblock {\em Annales scientifiques de l'Ecole normale supérieure},
  57(3):615--711, 2024.

\end{thebibliography}
\end{document}